\documentclass[11pt]{amsart}
\usepackage[colorlinks,citecolor=red,pagebackref,hypertexnames=false]{hyperref}

\usepackage{verbatim} 

\makeatletter
\def\@tocline#1#2#3#4#5#6#7{\relax
  \ifnum #1>\c@tocdepth 
  \else
    \par \addpenalty\@secpenalty\addvspace{#2}%
    \begingroup \hyphenpenalty\@M
    \@ifempty{#4}{%
      \@tempdima\csname r@tocindent\number#1\endcsname\relax
    }{%
      \@tempdima#4\relax
    }%
    \parindent\z@ \leftskip#3\relax \advance\leftskip\@tempdima\relax
    \rightskip\@pnumwidth plus4em \parfillskip-\@pnumwidth
    #5\leavevmode\hskip-\@tempdima
      \ifcase #1
       \or\or \hskip 1em \or \hskip 2em \else \hskip 3em \fi%
      #6\nobreak\relax
    \dotfill\hbox to\@pnumwidth{\@tocpagenum{#7}}\par
    \nobreak
    \endgroup
  \fi}
\makeatother

\usepackage{tikz,amsthm,amsmath,amstext,amssymb,amscd,epsfig,euscript, mathrsfs, dsfont,pspicture,multicol, mathtools,graphpap,graphics,graphicx,times,enumerate,subfig,sidecap,wrapfig,color}

\usepackage{ hyperref}
\usepackage{enumerate}

 \numberwithin{equation}{section}

\def\M{{\mathcal{M}}}

\def\ve{\varepsilon}

\def\wt{\widetilde}
\def\DS{\mathop\textup{DS}}
\def\DL{\mathop\textup{DL}}
\def\DMO{\mathop\textup{DMO}}
\def\DDMO{\mathop\textup{DDMO}}
\def\pv{\mathop\textup{pv}}
\def\I{\mathop\textup{I}}
\def\II{\mathop\textup{II}}

\def\Rn1{\mathbb{R}^{n+1}}

\DeclareMathOperator{\diam}{diam}

\def\Lip{\mathrm{Lip}} 						

\def\adj{\mathop\mathrm{adj}} 					
\def\dist{\textup{dist}} 						
\def\supp{\mathop\mathrm{supp}}					
\def\loc{\mathop\mathrm{loc}}						
\renewcommand{\div}{\mathop\mathrm{div }}			

\def\Xint#1{\mathchoice
{\XXint\displaystyle\textstyle{#1}}%
{\XXint\textstyle\scriptstyle{#1}}%
{\XXint\scriptstyle\scriptscriptstyle{#1}}%
{\XXint\scriptscriptstyle\scriptscriptstyle{#1}}%
\!\int}
\def\XXint#1#2#3{{\setbox0=\hbox{$#1{#2#3}{\int}$ }
\vcenter{\hbox{$#2#3$ }}\kern-.58\wd0}}

\def\avint{\Xint-}
\def\grad{\nabla}

\newcommand{\Rd}{\color{red}}
\newcommand{\Bl}{\color{blue}}

\theoremstyle{plain}
\newtheorem{theorem}{Theorem}

\newtheorem{corollary}[theorem]{Corollary}

\newtheorem{lemma}[theorem]{Lemma}

\newtheorem{proposition}[theorem]{Proposition}

\theoremstyle{definition}

\newtheorem{definition}[theorem]{Definition}
\newtheorem{remark}[theorem]{Remark}

\numberwithin{equation}{section}
\numberwithin{theorem}{section}

  \DeclareFontFamily{U}{mathb}{\hyphenchar\font45} 
\DeclareFontShape{U}{mathb}{m}{n}{
      <5> <6> <7> <8> <9> <10> gen * mathb
      <10.95> mathb10 <12> <14.4> <17.28> <20.74> <24.88> mathb12
      }{}
\DeclareSymbolFont{mathb}{U}{mathb}{m}{n}

\DeclareMathSymbol{\toitself}      {3}{mathb}{"FD}  

\def\HH{\mathcal{H}}
\newcommand{\vv}{\vspace{2mm}}
\newcommand{\vvv}{\vspace{4mm}}

\newcommand{\dv}{\mathop{\rm div}}
\def\R{\mathbb{R}}

\begin{document}

\title[Layer potentials for operators with Dini mean oscillation-type coefficients]{$L^2$-boundedness of gradients of single layer potentials\\ for  elliptic operators  with coefficients of \\Dini mean oscillation-type}

\author[A. Molero]{Alejandro Molero}
\address{Departament de Matem\`atiques, Universitat Aut\`onoma de Barcelona, 08193 Bellaterra, Catalonia.}
\email{amolero@mat.uab.cat}

\author[M. Mourgoglou]{Mihalis Mourgoglou}
\address{Departamento de Matem\'aticas, Universidad del Pa\' is Vasco (UPV/EHU), Barrio Sarriena s/n 48940 Leioa, Spain and\\
Ikerbasque, Basque Foundation for Science, Bilbao, Spain.}
\email{michail.mourgoglou@ehu.eus}

\author[C. Puliatti]{Carmelo Puliatti}
\address{Departamento de Matem\'aticas, Universidad del Pa\' is Vasco (UPV/EHU), Barrio Sarriena s/n 48940 Leioa, Spain.}
\email{carmelo.puliatti@ehu.eus}

\author[X. Tolsa]{Xavier Tolsa}
\address{ICREA, Passeig Llu\'is Companys 23 08010 Barcelona, Catalonia; Departament	de Matem\`atiques, Universitat Aut\`onoma de Barcelona, 08193 Bellaterra, Catalonia; and Centre de Recerca Matem\`atica, 08193 Bellaterra, Catalonia.}
\email{xtolsa@mat.uab.cat}

\thanks{
	A.M. was supported by the predoctoral grant BES-2017-081272  and was partially supported by the grant MTM-2016-77635-P of the Ministerio de Econom\'ia y Competitividad (Spain). M.M. was supported  by IKERBASQUE and partially supported by the grant PID2020-118986GB-I00 of the Ministerio de Econom\'ia y Competitividad (Spain), and by  IT-1247-19 (Basque Government). C.P. was supported by the grant IT-1247-19 (Basque Government) and partially supported by PID2020-118986GB-I00 (Ministerio de Econom\'ia y Competitividad, Spain) and PGC2018-094522-B-I00 (Ministerio de Ciencia e Innovaci\'on, Spain). X.T. is supported by the European Research Council (ERC) under the European Union's Horizon 2020 research and innovation programme (grant agreement 101018680) and Mar\'ia de Maeztu Program for Centers and Units of Excellence (CEX2020-001084-M). He is also partially supported by the grant PID2020-114167GB-I00 of the Ministerio de Econom\'ia y Competitividad (Spain).}

\newcommand{\mih}[1]{\marginpar{\color{red} \scriptsize \textbf{Mi:} #1}}
\newcommand{\car}[1]{\marginpar{\color{blue} \scriptsize \textbf{Carmelo:} #1}}
\newcommand{\xavi}[1]{\marginpar{\color{green} \scriptsize \textbf{Xavi:} #1}}
\maketitle

\begin{abstract}
We consider a uniformly elliptic operator $L_A$ in divergence form associated with an $(n+1)\times(n+1)$-matrix  $A$ with real, merely bounded, and possibly non-symmetric coefficients. If 
\[
	\omega_{A}(r)=\sup_{x\in \mathbb{R}^{n+1}} \avint_{B(x,r)} \Big|A(z)-\avint_{B(x,r)}A \Big|\, dz,
\] 
then, under  suitable Dini-type assumptions on $\omega_A$, we prove the following: if $\mu$ is a compactly supported Radon measure in $\mathbb{R}^{n+1}$, $n \geq 2$,   and
\(
	T_\mu f(x)=\int \nabla_x\Gamma_A (x,y)f(y)\, d\mu(y)
\)
denotes the gradient of the single layer potential associated with $L_A$, then 
\[
	1+ \|T_\mu\|_{L^2(\mu)\to L^2(\mu)}\approx 1+ \|\mathcal R_\mu\|_{L^2(\mu)\to L^2(\mu)},
\]
where $\mathcal R_\mu$ indicates the $n$-dimensional Riesz transform. This allows us to provide a direct generalization of some deep geometric results, initially obtained for $\mathcal R_\mu$, which were recently extended to  $T_\mu$ associated with $L_A$ with H\"older continuous coefficients. In particular, we show the following:
\begin{enumerate}
	\item If $\mu$ is an $n$-Ahlfors-David-regular measure on $\mathbb{R}^{n+1}$ with compact support, then $T_\mu$ is bounded on $L^2(\mu)$  if and only if $\mu$ is uniformly $n$-rectifiable.
		\item Let $E\subset \mathbb{R}^{n+1}$ be compact and $\mathcal H^n(E)<\infty$. If $T_{\mathcal H^n|_E}$ is bounded on  $L^2(\mathcal H^n|_E)$, then $E$ is $n$-rectifiable.
	\item If $\mu$ is a non-zero measure on $\mathbb{R}^{n+1}$ such that $\limsup_{r\to 0}\tfrac{\mu(B(x,r))}{(2r)^n}$ {is positive and finite} for $\mu$-a.e. $x\in \mathbb{R}^{n+1}$ and $\liminf_{r\to 0}\tfrac{\mu(B(x,r))}{(2r)^n}$ vanishes for $\mu$-a.e. $x\in \mathbb{R}^{n+1}$, then the operator $T_\mu$ is not bounded on $L^2(\mu)$.
\item Finally, we prove that if $\mu$ is a Radon measure on $\mathbb R^{n+1}$ with compact support which satisfies a proper set of local conditions at the level of a ball $B=B(x,r)\subset \mathbb R^{n+1}$ such that $\mu(B)\approx r^n$ and $r$ is small enough, then a significant portion of the support of $\mu|_B$ can be covered by a uniformly $n$-rectifiable set. These assumptions include a flatness condition, the $L^2(\mu)$-boundedness of $T_\mu$ on a large enough dilation of $B$, and the smallness of  the mean oscillation of $T_\mu$ at the level of $B$.
\end{enumerate}
\end{abstract}

\tableofcontents
\section{Introduction}
The aim of this paper is to extend and provide a unified approach to several recent results on the connection of the $L^2$-boundedness of gradients of single-layer potentials associated with an elliptic operator in divergence form defined on a set  $E$ and the geometry of $E$. The importance of these operators stems from their role in the study of boundary value problems and free boundary problems for harmonic and  elliptic measure, as well as the study of analytic capacity (see for instance \cite{AHMMMTV16}, \cite{AGMT17}, \cite{AMT17a}, \cite{AMT20}, \cite{AMTV19}, \cite{BH16}, \cite{GMT18}, \cite{KS11}, \cite{MT20}, \cite{MT21}, \cite{PT20}, \cite{To03} and the references therein). 

The investigation of geometric properties of singular integrals has produced many important results starting with Calder\'on's proof in \cite{Ca77} of the boundedness of Cauchy transform on Lipschitz graphs with small Lipschitz constant. A prototypical example of a singular integral operator is the \textit{Riesz transform,} which is the higher dimensional analogue of the Cauchy transform. If $\mu$ is a Radon measure on $\Rn1$, $n\geq 1$, its associated ($d$-dimensional) Riesz transform is defined as
\[
\mathcal R^d_\mu f(x)=\int \frac{x-y}{|x-y|^{d+1}}\, d\mu(y),\qquad \text{ for } f\in L^1_{\loc}(\mu),
\]
whenever the expression above makes sense. For $\delta>0$ we define the \textit{$\delta$-truncated Riesz transform} as
\[
\mathcal R^d_{\mu,\delta}f(x)\coloneqq \int_{|x-y|>\delta} \frac{x-y}{|x-y|^{d+1}}f(y)\, d\mu(y)
\]
and if $f\equiv 1$ on $\Rn1$, we use the notation $\mathcal R^d\mu (x)= \mathcal R^d_\mu 1(x)$ and $\mathcal R^d_\delta\mu(x)=\mathcal R^d_{\mu,\delta}1(x)$. We say that $\mathcal R^d_\mu$ is \textit{bounded on} $L^2(\mu)$ if $\mathcal R^d_{\mu,\delta}$ is bounded on  $L^2(\mu)$ uniformly on $\delta>0$. In this case, we write
\[
\|\mathcal R^d_\mu\|_{L^2(\mu)\to L^2(\mu)}\coloneqq \sup_{\delta>0}\|\mathcal R^d_{\mu,\delta}\|_{L^2(\mu)\to L^2(\mu)}.
\]

\vv

Given $x\in\mathbb{R}^{n+1}$ and $r>0,$ we denote by $B(x,r)$ the open ball of center $x$ and radius $r$.
We say that a non-negative Borel measure has \textit{growth of degree $d$} or, for brevity, $d$-growth, and we write $\mu \in M^d_+( \Rn1)$, if there exists $c_0>0$ such that
\[
	\mu\bigl(B(x,r)\bigr)\leq c_0 r^{d}\qquad \text{ for all }\, x\in \Rn1, \,\,r>0,
\]
 Any such measure is in fact a Radon measure. Measures with polynomial growth are crucial for the study of singular integrals; for instance, if $\mu$ is a non-negative measure on $\Rn1$ without atoms and its associated Riesz transform $\mathcal R^d_\mu$ is bounded on $L^2(\mu)$, then $\mu\in M^d_+(\Rn1)$ (see \cite[p. 56]{Da92}, where this is proved for more general singular integral operators, and also Lemma \ref{lem:T_bounded_polynomial_growth}). 
 
A Borel measure $\mu$ is said \textit{$d$-Ahlfors-David regular} (also abbreviated by $d$-AD-regular) if there exists $C>0$ such that
\[
C^{-1}r^d\leq \mu\big(B(x,r)\big)\leq Cr^d \,\qquad \text{ for all }x\in\supp\mu, \,0<r<\diam(\supp \mu).
\]
If $\mathcal H^d$ stands for the $d$-dimensional Hausdorff measure in $\Rn1$, we say that a set $E\subset \Rn1$ is \textit{$d$-AD-regular} if $\mathcal H^d|_E$ is a $d$-AD-regular measure. 

A set $E\subset \Rn1$ is called \textit{$d$-rectifiable} if there exists a countable family of Lipschitz maps $f_j\colon \mathbb R^d\to \Rn1$ such that
\[
\mathcal H^d\Bigl(E\setminus \bigcup_jf_j(\mathbb R^d) \Bigr)=0.
\]
A measure $\mu$ is \textit{$d$-rectifiable} if it vanishes outside a $d$-rectifiable set $E$ and it is absolutely continuous with respect to $\mathcal H^d|_E$.

We say that a set $E\subset \Rn1$ is \textit{uniformly $d$-rectifiable} if it is $d$-AD regular and there exist $\theta,M>0$ such that for all $x\in E$ and all $r>0$ there is a Lipschitz mapping $g$ from the ball $B_d(0,r)\subset\mathbb R^d$ to $\Rn1$ with $\Lip(g)\leq M$ such that
\[
\mathcal H^d\big(E\cap B(x,r)\cap g(B_d(0,r))\big)\geq \theta r^d.
\]
We also  say that a measure $\mu$ is \textit{uniformly $n$-rectifiable} if it is $d$-AD-regular and it vanishes outside of a uniformly $d$-rectifiable set.

The notion of uniform rectifiability of a set $E$ was introduced by  David and  Semmes in their seminal works  \cite{DS91, DS93} as the optimal geometric property that $E$ should have  so that operators in a pretty general subclass of singular integral operators are $L^2(\HH^n|_E)$-bounded. They proved in \cite{DS91} that a $d$-AD-regular measure $\mu$ on $\Rn1$ is uniformly $d$-rectifiable if and only if \textit{all} the singular integral operators with smooth and anti-symmetric convolution-type kernel are bounded on $L^2(\mu)$. They also raised the question, commonly referred to as \textit{David and Semmes' problem}, if the $L^2(\mu)$-boundedness of the $d$-Riesz transform $\mathcal R^d_\mu$ associated with a $d$-AD-regular measure $\mu$ implies its uniform $d$-rectifiability.

A positive answer to this question was first provided in the planar case $d=n=1$ by Mattila, Melnikov, and Verdera in \cite{MMV96}, who used the  connection of the Cauchy transform with the so-called Menger curvature of a measure. However, their method cannot be generalized to higher dimensions.
More recently, Nazarov and Volberg along with the fourth named author proved in \cite{NToV14} the analogous result in the case $d=n$ for any integer $n\geq 1$  using a different set of delicate techniques (we will often refer to it as the $1$-codimensional case).
We point out that the full David-Semmes' conjecture is still open for $d$-AD-regular measures of dimension $d=2,\ldots, n-2$.

\vv
The $n$-dimensional Riesz transform in $\Rn1$ has a natural generalization to the context of elliptic PDEs.
Let $A(\cdot)=(a_{ij})_{i,j\in \{1,\ldots,n+1\}}$ be an $(n+1)\times(n+1)$-matrix whose entries $a_{ij}$ are measurable real-valued functions in $L^\infty(\Rn1)$. We say that $A$ is \textit{uniformly elliptic} if there exists $\Lambda>0$ such that
\begin{align}
	\langle A(x)\xi,\xi\rangle &\geq \Lambda^{-1}|\xi|^2\qquad \text{ for all } \xi\in \Rn1\text{ and a.e. }x\in \Rn1,\label{eq:ellip1_scalar}\\
	\langle A(x)\xi,\eta\rangle &\leq \Lambda |\xi||\eta|\qquad \text{ for all } \xi,\eta\in \Rn1 \text{ and a.e. }x\in \Rn1.\label{eq:ellip2_scalar}
\end{align}
We consider the second order equation in divergence form
\begin{equation}\label{eq:snd_order_diff_eq}
	L_A u(x)\coloneqq -\div(A(\cdot)\nabla u(\cdot))(x)=0,\qquad x\in \Rn1,
\end{equation}
to be understood in the sense of distributions. If $A$ is a uniformly elliptic matrix with bounded measurable coefficients, the operator $L_A$ has a \textit{fundamental solution} $\Gamma_A(x,y)$ which, if $\delta_y$ is the Dirac  mass at $y$, satisfies \(L_{A}\Gamma_A(\cdot,y)=\delta_y\) in the sense of distributions.
For the construction of the fundamental solution associated with $L_A$ we refer to \cite{HK07}.

For a non-negative Radon measure $\mu$ on $\Rn1$ we define the \textit{gradient of the single layer potential}
\begin{equation}\label{eq:single layer}
T_\mu f(x)\coloneqq \int \nabla_1 \Gamma_A(x,y)f(y)\, d\mu(y),\qquad \text{ for }f\in L^1_{\loc}(\mu), 
\end{equation}
to be interpreted in the sense of the truncations
\[
T_{\mu,\delta} f(x)\coloneqq \int_{|x-y|>\delta} \nabla_1 \Gamma_A(x,y)f(y)\, d\mu(y),\qquad \text{ for }f\in L^1_{\loc}(\mu).
\]
For $f\equiv 1$ we use the notations $T_\delta\mu\coloneqq T_{\mu,\delta}1$ and $T\coloneqq T_{\mu}1$. We also denote \[
\|T_\mu\|_{L^2(\mu)\to L^2(\mu)}\coloneqq \sup_{\delta>0}\|T_{\mu,\delta}\|_{L^2(\mu)\to L^2(\mu)}.
\]

Furthermore, we observe that in the case $A\equiv Id$, it readily follows by definition that $L_A=-\Delta$ and so $\nabla_1\Gamma_{Id}$ equals the Riesz kernel up to a dimensional multiplicative constant.

Under the sole  assumption that the entries of $A$ are in $L^\infty$, the kernel $\nabla_1\Gamma_A(\cdot,\cdot)$ does not necessarily satisfy local $L^\infty$ estimates, let alone a modulus of  continuity, and so it is not necessarily of Calder\'on-Zygmund type. We need to impose some additional regularity conditions on $A$ for this to happen. For instance, an adequate framework is provided by matrices with \textit{H\"older continuous entries}. Many important geometric results that are known for the Riesz transform, such as the $1$-codimensional version of the David-Semmes' problem, have been successfully generalized by Conde-Alonso, Prat, and the last three named authors (see  \cite{CMT19}, \cite{PPT21}, \cite{Pu19}). For more details we refer to the discussion of the corollaries of Theorem \ref{theorem:bound_L2_norm_operators}.

\vv
In the present paper, we are concerned with elliptic operators whose coefficients  may have a Lebesgue measure zero set of points of discontinuity. Namely, we will assume that they are of \textit{Dini mean oscillation-type}.

\vv

Let $\kappa \geq 1$.	We say that a function $\theta\colon[0,\infty]\to [0,\infty]$ is $\kappa$-\textit{doubling} if
\begin{equation}\label{eq:dini1_new}
 \theta (t) \leq \kappa \,\theta (s)\, \qquad \text{ for }\,  \frac{1}{2}t\leq s \leq t\,  \text{ and } \, t>0.
\end{equation}
%

We denote by $\mathcal L^d$ the Lebesgue measure on $\R^d$ and for a set $E\subset \Rn1$, we will also use the notation $\mathcal L^{n+1}(E)=|E|.$ When we write integrals, we often prefer the more compact and standard notation $d\mathcal L^d (x) =dx$.

\vv

We say that a $\kappa$-doubling function $\theta$ belongs to the class {$\DS(\kappa)$} ({\it Dini in small scales}), if it is $\mathcal L^1$-measurable and
\begin{equation}\label{eq:sDini}
	\int_0^1 \theta(t)\,\frac{dt}{t}<\infty.
\end{equation}
Given $d>0$, we say that $\theta$ belongs to the class $\DL_{d}(\kappa)$ ($d$-{\it Dini in large scales}) if it is $\mathcal L^1$-measurable and 
\[
\int_1^\infty \theta(t)\,\frac{dt}{t^{d+1}}<\infty.
\]
We remark that, if $0<d_1\leq d_2$, then $\DL_{d_1}(\kappa)\subset \DL_{d_2}(\kappa)$. Moreover, for $\theta\in \DS(\kappa)$ we define
\begin{equation}\label{eq:def_mathfrak_I}
	\mathfrak I_{\theta}(r)\coloneqq \int_0^r \theta (t)\, \frac{dt}{t}, \qquad r>0
\end{equation}
and, for $d>0$ and $\theta\in \DL_d(\kappa)$,
\begin{equation}\label{eq:def_mathfrak_L}
	\mathfrak L^d_{\theta}(r)\coloneqq r^{d}\int_r^\infty \theta(t)\, \frac{dt}{t^{d+1}} , \qquad r>0
\end{equation}	

\vv
{For $x\in \Rn1$, $r>0$, and an $(n+1)\times (n+1)$-matrix $A$ we denote
	\[
		\bar A_{x,r}\coloneqq \avint_{B(x,r)}A \coloneqq \frac{1}{|B(x,r)|}\int_{B(x,r)}A(y)\, dy
	\]
	and, for $p\geq 1$, define its \textit{mean oscillation} function $\omega_{A}\colon [0, \infty) \to [0, \infty)$ as
	\[
		\omega_{A}(r)\coloneqq \sup_{x\in \R^{n+1}} \avint_{B(x,r)}\bigl|A(z)-\bar A_{x,r}\bigr|\, dz.
	\]
By  \cite[p.495]{Li17}\footnote{The doubling property was proved in \cite{Li17} for slightly different Dini moduli of oscillation, but a minor variant of that argument works also under a Dini mean oscillation assumption (see also the use of the doubling property in \cite[p.424]{DK17}).},  there exists a dimensional constant  $\kappa$ such that $\omega_{A}$ satisfies \eqref{eq:dini1_new}.   

We say that an $(n+1)\times (n+1)$-matrix $A \in \DMO_s$ (resp. $A \in \DMO_\ell$)  if $ \omega_A \in \DS(\kappa)$ (resp. $ \omega_A \in \DL_{n-1}(\kappa)$).	We also say that $A \in \DDMO_s$ if $A \in \DMO_s$ and $\mathfrak I_{ \omega_A}$ satisfies \eqref{eq:sDini}, i.e., 
\begin{equation}\label{eq:logdini}
		\int_0^1 \int_0^r  \omega_A(t)\, \frac{dt}{t}\frac{dr}{r}=-\int_0^1  \omega_A(t)\,\log{t} \,\frac{dt}{t}<+ \infty.	
\end{equation}
	Finally, we define 
	\[
		\widetilde \DMO\coloneqq \textup{DDMO}_s \cap \textup{DMO}_\ell.
	\]
	The acronym $\DMO$ (resp. $\DDMO$) stands for \textit{Dini mean oscillation} (resp. \textit{double Dini mean oscillation}), and the subscripts in $\DMO_s$ and $\DMO_\ell$ indicate that the associated Dini condition is required at small and large scales respectively. Due to \eqref{eq:logdini}, we may also use the terminology  \textit{log-Dini mean oscillation} instead of  double Dini mean oscillation.
	
	\vv
	{Furthermore, $\widetilde \DMO$ includes the class of matrices with  $\alpha$-H\"older continuous coefficients for $\alpha\in(0,1)$. Indeed, if there exists $C_h>0$ such that 
		\begin{equation}\label{eq:Holder}
		|a_{ij}(x)-a_{ij}(y)|\leq C_h |x-y|^\alpha, \qquad \text{ for all }i,j\in\{1,\ldots,n+1\},\, x,y\in\Rn1,
		\end{equation}
		then $\omega_A(t)\lesssim t^\alpha$ and so $A\in \widetilde \DMO$. {Our condition even includes matrices that satisfy \eqref{eq:Holder} for $\alpha \in (0,1)$  when  $|x-y| \lesssim 1$ and $(n-1-\alpha)$ when $|x-y| \gtrsim 1$}.} In fact, it is clear that if $A$ is uniformly continuous with a Dini modulus of continuity then it is of Dini mean oscillation.	In the converse direction, as proved in \cite[Appendix A]{HwK20}, if $A$ is of Dini mean oscillation, then it agrees (Lebesgue) almost everywhere with a uniformly continuous function with modulus of continuity $\mathfrak I_{\omega_A}$. {However, as we are mostly interested in sets with Lebesgue measure zero, we highlight that we cannot assume that $A$ is uniformly continuous and thus more delicate arguments are required.}
	
	A variant of the example in \cite[p.~418]{DK17} shows that the condition $A\in \widetilde\DMO$ is strictly stronger than requiring the matrix $A$ to be H\"older continuous. Indeed, if we define the matrix $a_{ij}(x)=\delta_{ij}$ for $|x|>1$ and 
	\[
	a_{ij}(x)\coloneqq \delta_{ij} \left(1 + (-\ln|x|)^{-\gamma -1} \right),\qquad \text{ for }0<|x|\ll 1, 0<\gamma<1/2,
	\]
then, as remarked in \cite[p.~418]{DK17}, we have that $\omega_A(r)\approx (-\ln r)^{-\gamma -2}$ for $r\ll 1$. Since $ \omega_A$ is an increasing function, $A\in \DDMO_s$ but its modulus of continuity does not  satisfy the double Dini condition.

	The $\DMO_s$ assumption on $A$  guarantees that $\nabla_1\Gamma(\cdot,\cdot)$ is locally of Calder\'on-Zygmund type, see Lemma \ref{lem:estim_fund_sol}. Indeed, this is possible because of the work of Dong and Kim \cite{DK17}  who proved that, under this hypothesis, weak solutions of $L_Au=0$ are continuously differentiable providing also  local $L^\infty$ and regularity estimates for $\nabla u$.	
		We highlight that one of the crucial technical difficulties  in \cite{DK17} is that the modulus of oscillation $\omega_A$ is not monotone as it would be the case if one used  $\widetilde \omega_A(r)\coloneqq \sup_{0<\rho\leq r}\omega(\rho)$. The  proof of the regularity theorem of Dong and Kim is significantly easier for $\widetilde \omega_A$. Note  that if $A$ is a compactly supported perturbation of the identity matrix $Id$, then $ \omega_A(r)\to 0$ as $r \to \infty$ but $\widetilde \omega_{A}(r)$ does not\footnote{We would like to thank Seick Kim for bringing those facts  to our attention motivating us to improve on a previous version of our results where we had used $\widetilde \omega_A$.}.
		
	\vv

	Let us now state the main result of the paper.
	\begin{theorem}\label{theorem:bound_L2_norm_operators}
		Let  $A$ be a uniformly elliptic matrix satisfying  $A\in \widetilde \DMO$ and let $\mu\in M^n_+(\Rn1)$ with compact support, $n\geq 2$. If  $T_\mu$ is the associated operator given by \eqref{eq:single layer},   it holds that
		\begin{equation}\label{eq:maineq}
			1 + \|\mathcal R_\mu\|_{L^2(\mu)\to L^2(\mu)}\approx 1 + \|T_\mu\|_{L^2(\mu)\to L^2(\mu)},
		\end{equation}
		where the implicit constant depends on $n, \Lambda, c_0$, and $\diam(\supp\mu)$. In fact, the following is true: for every $\epsilon>0$, there exists $R$ small enough  such that if $\diam (\supp \mu) \leq R$ then 
		\begin{align}\label{eq:maineq-small_R-T}
			\|\mathcal R_\mu\|_{L^2(\mu)\to L^2(\mu)} &\leq  \epsilon +C' \|T_\mu\|_{L^2(\mu)\to L^2(\mu)}\\
				\|T_\mu\|_{L^2(\mu)\to L^2(\mu)} &\leq  \epsilon +C' \|\mathcal R_\mu\|_{L^2(\mu)\to L^2(\mu)},\label{eq:maineq-small_T-R}
		\end{align}
		for some $C'>0$ depending on $n, \Lambda$, and $c_0$.
	\end{theorem}

	The role of the $\widetilde\DMO$-condition on the matrix $A$ in Theorem \ref{theorem:bound_L2_norm_operators} can be better understood if we relate it to the technical framework of the recent works in the H\"older continuous setting. Indeed, one of the key methods of \cite{CMT19} and the subsequent papers consists in using a proper pointwise estimate of the difference $\nabla_1\Gamma_A(x,y)-\nabla_1\Gamma_{A(x)}(x,y)$, often referred to as \textit{frozen coefficients method} which was  proved in \cite{KS11}. This approach is particularly important because it allows to reduce the study of the operator $T_\mu$ to the gradient of the single layer potential associated with a uniformly elliptic equation with \textit{constant coefficients}, which in turn coincides with the Riesz transform modulo a linear change of variables (that depends on $x$ though).
	
	Hence, a crucial difficulty in the proof of Theorem \ref{theorem:bound_L2_norm_operators} is the identification of the right substitute of the frozen coefficients method which adapts to the mean oscillation setting. This issue is resolved in Lemma \ref{lem:main_pw_estimate}, where we estimate the difference of $\nabla_1\Gamma_A(x,y)$ and $\nabla_1\Gamma_{\bar A_{x,r}}(x,y)$, for $r\coloneqq |x-y|/2$.
	The bound depends on the scale $R>0$ such that $x,y\in B(0,R)$ and it involves the quantity
	\begin{equation}\label{eq:defin_tau_A}
		r^{-n}\,	\tau_A(r) \coloneqq r^{-n}\,	 \left( \mathfrak I_{\omega_A}(r) + \mathfrak L^n_{\omega_A}(r)\right)= \frac{1}{r^n}\int_0^r 	\omega_A(t)\, \frac{dt}{t} +  \int_r^\infty \omega_A(t)\, \frac{dt}{t^{n+1}}
	\end{equation}
	and the term
	\[
		R^{-n}\widehat \tau_A(R)\coloneqq R^{-n}\left( \mathfrak I_{\omega_A}(R)+\mathfrak L^{n-1}_{\omega_A}(R) \right)= \frac{1}{R^{n}}\int_0^R\omega_A(t)\, \frac{dt}{t}+\frac{1}{R}\int_R^\infty \omega_A(t) \, \frac{dt}{t^n}.
	\]
	Observe that the $\widetilde \DMO$ assumption implies that $\mathfrak I_{\tau_A}(1)<\infty$ and  $\widehat \tau_A(R)<\infty$ for any $R>0$.
	Moreover, if $A$ is $\alpha$-H\"older continuous, $\tau_A(r)\lesssim r^\alpha$.
	
	In the proof of Lemma \ref{lem:main_pw_estimate} we use a slight variation of the result of Dong and Kim \cite{DK17} (see Theorem \ref{theorem:dong_kim}) and make some delicate PDE estimates obtaining a  sharp bound in terms of $\tau_A$ and the term $R^{-n}\widehat \tau_A(R)$.	 Since we allow the implicit constant in \eqref{eq:maineq} to depend on the $\diam (\supp\mu)$, picking up the term $R^{-n}\widehat \tau_A(R)$ not only is it harmless,  but it is a term that can become  small  if the support of the measure has small enough diameter. Its importance will become evident in the proof of Corollary  \ref{cor:elliptic_GSTo}.
	
	One of the main difficulties is that we do not have scale invariant estimates and in large scales this creates a significant complication. This stands in contrast to the case of H\"older continuous and periodic coefficients, where scale invariant local $L^\infty$ estimates for the gradient of a solution are at our disposal, which makes things work smoothly in scales much  larger than $R$.  Let us  highlight that, in the present manuscript, we do not require any periodicity assumption on the matrix. In fact, we fill a gap in the use of \cite[Lemma 2.2]{KS11} even for H\"older continuous matrices in the previous works \cite{CMT19}, \cite{PPT21}, \cite{Pu19}, and \cite{BMR21}, where \cite[Lemma 2.2]{KS11} was invoked for non-periodic matrices without any additional justification. To be precise, the bound of \eqref{eq:3rlem} is the missing component. Another  obstacle when working with elliptic operators $L_A$ associated with non-constant matrices is that the kernel $\nabla_1\Gamma(\cdot, \cdot)$ is not anti-symmetric, which, in principle, is rather inconvenient when dealing with its associated single layer potential. 
	
	\vv
	Our strategy to prove Theorem \ref{theorem:bound_L2_norm_operators}  consists in using the frozen coefficients type method  in order to bound the $L^2$-operator norm of the difference of the $\delta$-truncated gradient of the single layer potential and the $\delta$-truncated Riesz transform at the level of a cube in terms of the operator norm of $\mathcal R_\mu$.  This is a three-step perturbation argument:
	\begin{enumerate}
	\item  The first step is the comparison of $\nabla_1\Gamma_A$ with $\nabla\Theta(\cdot;\bar A_{x,|x-y|/2})$ that has already been described above (see   Lemma \ref{lem:main_pw_estimate}). The dependence of the second kernel on both $x$ and $y$ requires  an additional step.
	\item The second step is to compare $\nabla\Theta(\cdot;\bar A_{x,|x-y|/2})$ with $\nabla\Theta(\cdot;\bar A_{x,\delta/2})$, where $\delta$ is the level of the truncation of the single layer potential (see Lemmas \ref{lem:lem_1_prep_main_lemma} and  \ref{lem:lem_2_prep_main_lemma}). This is crucial since it allows us to reduce case to a smooth and odd kernel which is homogeneous of degree $-n$ and  independent of the $y$ variable (it is the variable with respect to which we integrate the kernel to construct the integral operator). 
	\item The third and final step is the estimate of the difference between $\nabla\Theta(\cdot;\bar A_{x,\delta/2})$ and the normalized Riesz kernel. Here we assume that our measure is supported on a cube $Q$ centered at $x_Q$ and $x, y \in Q$. Modulo a change of variables argument, we can assume that the average of $A$ over a ball centered at $x_Q$ with radius comparable to the side-length of the cube is the identity matrix. Contrarily to the previous case, we want  to compare $\nabla\Theta(\cdot;\bar A_{x,\delta/2})$ with $\nabla\Theta(\cdot;\bar A_{x_Q,M \ell(Q)})$ moving up from $\delta$ to a higher scale $\ell(Q)$.  Pure PDE estimates do not give satisfactory upper bounds and hence, inspired by the approach of \cite[Section 1]{MiT99}, we study $\nabla\Theta(\cdot;\bar A_{x,\delta/2})-\nabla \Theta(\cdot; Id)$ via the method of \textit{spherical harmonics expansion} (see Lemma \ref{lem:estimate_norm_K3}).
\end{enumerate}
	
	More specifically, in the latter step, we prove suitable bounds on the coefficients of the expansion again via PDE estimates. However, we also need proper estimates for the operator norm of singular integrals associated with harmonic polynomials in terms of the norm of the Riesz transform.
	In order to accomplish this, our argument also relies  on some powerful results  which have been recently proved by the last named author in \cite{To21} (this paper relaxed the assumptions of the main theorem of its companion paper by Dabrowski and the last named author \cite{DT21},  and provided an extension of \cite{To05} to higher dimensions). In particular, that work characterizes non-atomic  Radon measures in $\Rn1$ 	 whose Riesz transform is $L^2$-bounded. From that result it was possible to derive the invariance of the $L^2$-boundedness of the Riesz transform under bilipschitz transformations of the measure.
	Furthermore, as proved in \cite[Corollary 1.4]{To21}, if $\mu$ is a measure in $\Rn1$ with no point masses, and $\mathcal T_{K,\mu}$ is the singular integral operator of convolution-type formally defined as
	\[
	\mathcal T_{K,\mu} f(x)=\int K(x-y)f(y)\, d\mu(y)\, \qquad\text{ for } f\in L^1_{\loc}(\mu),
	\]
	where $K$ is antisymetric and satisfies
	\[
	|\nabla^j K(x)|\lesssim \frac{1}{|x|^{n+j}}, \qquad x\in \Rn1\setminus \{0\},\,  0\leq j\leq 2,
	\]
	then
	\[
	\|\mathcal T_{K,\mu}\|_{L^2(\mu)\to L^2(\mu)}\leq C \|\mathcal R_\mu\|_{L^2(\mu)\to L^2(\mu)}.
	\]
	
	\vvv
	
	Theorem \ref{theorem:bound_L2_norm_operators} has some direct and important applications, which we present below. 
	If $\mu$ is a non-zero Borel measure on $\Rn1$ and $s\in (0,n+1]$, we define its \textit{upper $s$-dimensional density}
	\[
	\Theta^{*,s}(x,\mu)\coloneqq \limsup_{r\to 0} \frac{\mu(B(x,r))}{(2r)^s} \qquad \text{ for }x\in \Rn1
	\]
	and its \textit{lower $s$-dimensional density}
	\[
	\Theta^s_*(x,\mu)\coloneqq \liminf_{r\to 0}\frac{\mu(B(x,r))}{(2r)^s}\qquad \text{ for }x\in \Rn1.
	\]
	
	One of the key tools used in the solution of the codimension-$1$ David-Semmes' problem in \cite{NToV14} is a variational technique partially inspired  by a previous argument of Eiderman, Nazarov and Volberg in \cite{ENV12}. The main result of \cite{ENV12} is that, for $n\leq s<n+1$, if a measure $\mu$ on $\Rn1$ is such that $0<\Theta^{*,s}(x,\mu)<\infty$ for $\mu$-almost every $x$ and $\Theta^s_*(x,\mu)=0$ $\mu$-almost everywhere, then its $s$-dimensional Riesz transform $\mathcal R^s_\mu$ is \textit{not} bounded on $L^2(\mu)$. This was  recently generalized  in the case $s=n$ by Conde-Alonso together with the second and fourth named authors in \cite{CMT19} for the gradient of the single layer potential $T_\mu$ associated with a H\"older continuous matrix $A$ (see also \cite{BMR21} for a version of \cite{CMT19} for Schr\"odinger operators). Since  H\"older continuous matrices belong to $\widetilde \DMO$, Theorem \ref{theorem:bound_L2_norm_operators} allows us to obtain an alternative approach to \cite[Theorem A]{CMT19}, and to extend it to a more general class of elliptic equations.
	
	\begin{corollary}\label{theorem:ENV_elliptic}
Let	$A$ be a uniformly elliptic matrix satisfying  $A\in \widetilde \DMO$ and let $\mu$ be a non-zero measure on $\Rn1$, $n\geq 2$, such that $0<\Theta^{*,n}(x,\mu)<\infty$ and $\Theta^n_*(x,\mu)=0$ for $\mu$-a.e. $x\in \Rn1$. Then the associated operator $T_\mu$ given by \eqref{eq:single layer} is not bounded on $L^2(\mu)$.
	\end{corollary}

	Another important application of our main theorem is the elliptic version of the David and Semmes problem in codimension $1$. Via a decomposition in spherical harmonics, it was proved in \cite{CMT19} that, if $A$ is H\"older continuous, $T_\mu$ is bounded on $L^2(\mu)$ on uniformly $n$-rectifiable measures $\mu$ with compact support. The converse implication was obtained by Prat, Puliatti, and Tolsa in \cite{PPT21}, via a non-trivial adaptation of the scheme of \cite{NToV14}. In particular, we remark that their proof relies on a delicate reflection argument for the matrix across hyperplanes, and it is not clear how to adapt it to the context of uniformly elliptic matrices with Dini mean oscillation. Nevertheless, Theorem \ref{theorem:bound_L2_norm_operators} readily shows the following:
	\begin{corollary}\label{theorem:DS_elliptic}
Let	$A$ be a uniformly elliptic matrix satisfying  $A\in \widetilde \DMO$ and let $\mu$ be an $n$-AD-regular measure on $\Rn1$, $n\geq 2$, with compact support.  If $T_\mu$ is the associated operator  given by \eqref{eq:single layer}, then  $T_\mu$ is bounded on $L^2(\mu)$ if and only if $\mu$ is uniformly $n$-rectifiable.
	\end{corollary}
	The combination of \cite{NToV14}, \cite{ENV12}, and a covering argument of Pajot allowed Nazarov, Volberg, and the fourth named author to prove in \cite{NToV14b} that if $E\subset\Rn1$ is such that $\mathcal H^n(E)<\infty$ and $\mathcal R_{\mathcal H^n|_E}$ is bounded on $L^2(\mathcal H^n|_E)$, then the set $E$ is $n$-rectifiable. Its elliptic analogue  for second order elliptic operators in divergence form associated with H\"older continuous matrices was obtained in \cite{PPT21}. We generalize  this result as well.
	\begin{corollary}\label{theorem:NTVpubmat_elliptic}
	Let	$A$ be a uniformly elliptic matrix satisfying  $A\in \widetilde \DMO$ and let $E\subset \Rn1$, $n\geq 2$, be a compact set with $\mathcal H^n(E)<\infty$. If $T$ is the associated operator given by \eqref{eq:single layer} and $T_{\mathcal H^n|_E}$ is  bounded  on $L^2(\mathcal H^n|_E)$, then $E$ is $n$-rectifiable.
	\end{corollary}
	
	The main advantage of Theorem \ref{theorem:bound_L2_norm_operators} is that its application gives alternative and more direct proofs  of \cite{CMT19},   \cite{PPT21}, and \cite{Pu19} via \cite{ENV12}, \cite{NToV14},  \cite{NToV14b}, and \cite{GT18}, which readily extend to uniformly elliptic matrices in $\widetilde \DMO$. 
	
	{
		Moreover, if we further assume that the matrix $A$ is H\"older continuous, Corollaries \ref{theorem:DS_elliptic} and \ref{theorem:NTVpubmat_elliptic} are crucial tools in order to prove a rectifiability result for elliptic measure in the context of a non-variational one-phase problem (see \cite[Theorem 1.3]{PPT21}) which generalizes \cite{AHMMMTV16}. For more details we refer to \cite[Section 12]{PPT21}.
	}
	
	\vv
	
	Finally, we also extend the main result of Girela-Sarri\'on and the fourth named author \cite{GT18} as well as its elliptic analogue of the third named author in \cite{Pu19}.
		{Let $\mu$ be a Radon measure on $\Rn1$. For a ball $B\subset \Rn1$ of radius $r(B)$ {and an integer $N>0$}, we denote
		{
		\begin{align}
		\Theta_\mu(B)&=\frac{\mu(B)}{r(B)^n},\notag\\
		\alpha_A(t)&=t+ t^\beta + \omega_A(t),\qquad t>0, \, \beta\in (0,1]\label{eq:def_alpha_A}\\
		P_{\gamma, \mu}(B)&\coloneqq  \sum_{j\geq 0} 2^{-{\gamma j}}\Theta_{\mu}(2^j B), \qquad \gamma\in (0,1]\notag\\
	\mathcal P^N_{\omega,\mu}(B)& \coloneqq \sum_{j \geq N} \alpha_A(2^{-j}) \Theta_\mu(2^j B).\label{eq:N-poisson}
\end{align}
	}
		Given an $n$-dimensional plane $L$ in $\Rn1$ we denote
		\[
		\beta^L_{\mu,1}(B)=\frac{1}{r(B)^n}\int_B\frac{\dist(x,L)}{r(B)}\, d\mu(x)\qquad\text{ and }\qquad \beta_{\mu,1}(B)=\inf_L \beta^L_{\mu,1}(B),
		\]
		where the infimum is taken over all hyperplanes.
		Finally, for a set $E\subset\Rn1$ with $\mu(E)>0$ and $f\in L^1_{\loc}(\mu)$ we write
		\[
		m_{E}(f,\mu)=\frac{1}{\mu(E)}\int_E f\, d\mu.
		\]
		
		Let $M(\Rn1)$ be the space of real Borel measures, endowed with the total variation norm $\|\cdot\|$: for $\mu\in M(\Rn1)$ we indicate by $|\mu|$ its variation and by $\|\mu\|\coloneqq |\mu|(\Rn1)$.
		
		\vv
		For our application, we have to determine whether $T_{\mu,\varepsilon}f$ converges pointwise $\mu$-almost everywhere for $\varepsilon\to 0$. In case it does, we denote the limit as
		\[
			\pv T_\mu f(x)=\lim_{\varepsilon\to 0}T_{\mu,\varepsilon}f(x)
		\]
		and we refer to it as the \textit{principal value} of the integral $T_\mu f(x)$. The existence of principal values for gradients of single layer potentials can be proved in our framework via a minor variant of the arguments of \cite[Theorem 1.1]{Pu19}: one can study separately the case of rectifiable measures and that of measures with zero density, which can be both analyzed via the frozen coefficients method of Lemma \ref{lem:main_pw_estimate}. Ultimately, this implies the following result.
		\begin{proposition}\label{theorem_pv_layer_pot}
			Let $\mu$ be a Radon measure on $\Rn1$, $n\geq 2$, with compact support and with growth of degree $n$.
			Let $A$ be a uniformly elliptic matrix satisfying $A\in \widetilde\DMO$ and assume that its associated gradient of the single layer potential $T_\mu$ is bounded on $L^2(\mu)$. Then the following holds:
			\begin{enumerate}
				\item for $1\leq p<\infty$ and all $f\in L^p(\mu)$, $\pv T_\mu f(x)$ exists for $\mu$-a.e. $x\in\Rn1$.
				\item for all $\nu\in M(\Rn1)$, $\pv T\nu(x)$ exists for $\mu$-a.e. $x\in\Rn1$.
			\end{enumerate}
		\end{proposition}
		\vv
	}
	Finally, we state the local quantitative rectifiability criterion for Radon measures which generalizes \cite{GT18} and \cite{Pu19}. {We refer to the introductions of the aforementioned articles for a detailed discussion of the result and the role of all the hypotheses. } The result is stated in the form of \cite[Corollary 3.2]{AMT17a}.
	
	{
	\begin{corollary}\label{cor:elliptic_GSTo}
		Let	$A$ be a uniformly elliptic matrix satisfying  $A\in \widetilde \DMO$ and let $\mu$ be a Radon measure with compact support in $\Rn1$, $n\geq 2$. Let $B\subseteq \Rn1$ be an open ball with $\mu(B)>0$ and let $C_0,C'_0>0$. Denote by $T_\mu$ the gradient of the single layer potential associated with $L_A$ and $\mu$, and let $\beta$ be as in Lemma \ref{lem:estim_fund_sol}. Suppose that $\mu$ and $B$ are such that, for some positive real numbers $\tau$, $\delta$, and $\lambda$, and a positive integer  $N$, the following properties hold:
		\begin{enumerate}
			\item $r(B)\leq\lambda$.
			\item We have $\mathcal P^0_{\omega,\mu}(B)\leq C_0 \Theta_\mu(B)$, $\mathcal P^N_{\omega,\mu}(B)\leq C_0 \mathfrak I_{\alpha_A}(2^{-N})\Theta_\mu(2^NB)$, and it holds $\Theta_\mu(B(x,r))\leq C_0 \Theta_\mu(2^NB)$ for all $x\in B$ and $0<r\leq 2^Nr(B)$.
			\item $T_{\mu|_{2^{N} B}}$ is bounded on $L^2(\mu|_{2^{N} B})$ and it holds $\|T_{\mu|_{2^{N} B}}\|_{L^2(\mu|_{2^{N} B})\to L^2(\mu|_{2^{N} B})}\leq C'_0\Theta_\mu(2^NB)$.
			\item There exists some $n$-plane $L$ passing through the center of $B$ such that $\beta^L_{\mu,1}(B)\leq \,\delta\Theta_\mu(B).$
			\item We have
			\begin{equation}\label{eq:mean_osc_small_theorem}
				\int_B\big|T_\mu1(x)-m_{B}(T_\mu 1,\mu)\big|^2\, d\mu(x)\leq\,\tau\,\Theta_\mu(2^NB)^2\mu(B).
			\end{equation}
		\end{enumerate}
		Then there exist a choice of $\delta$ and $\tau$ small enough, possibly depending on $n, \Lambda, C_0, C'_0$, and $\diam(\supp \nu)$, a choice of $N=N(\tau, n, \Lambda, \diam(\supp \nu), C_0, C'_0)$ large enough, a choice of 
		$\lambda =\lambda(\tau, N, n, \Lambda, C_0,C'_0, \diam(\supp \nu))$ small enough such that $2^N \lambda$ is also sufficiently small, for which the following holds:  if $\mu$ satisfies (1)--(6), there exist a uniformly $n$-rectifiable set $\Gamma$ and $\theta \in (0,1)$ such that
		\[
		\mu(B\cap\Gamma)\geq \theta\mu(B).
		\]
		The UR constants of $\Gamma$ depend on all the constants above.
	\end{corollary}
	}
	\vv
	{	{In the condition \eqref{eq:mean_osc_small_theorem} we identify $T_\mu 1$ with $\pv T_\mu 1$.}
		Moreover, the well-posedness of the expression on the left hand side of \eqref{eq:mean_osc_small_theorem} can be justified via the existence of principal values and the fact the measure $\mu$ has compact support. For the details we refer to \cite[Section 2.4]{GT18} (see also \cite[Section 3]{Pu19}).
		
		The proof of Corollary \ref{cor:elliptic_GSTo} consists in showing that, for $\lambda\ll 1$, the smallness in the mean oscillation assumption \eqref{eq:mean_osc_small_theorem} implies smallness for the analogous quantity associated to the Riesz transform. As it is more coherent with the notation of the rest of the paper, we will equivalently prove this for cubes in $\Rn1$; we also remark that Girela-Sarri\'on and Tolsa reduced the proof of their theorem to \cite[Main Lemma 3.1]{GT18}, which is formulated for cubes itself.
		
		{
		Finally, we remark that for $0<\beta<1$, $N\geq 1$, $\mu\in M^n_+(\Rn1)$, and a ball $B\subset \Rn1$ such that $2^N B$ satisfies the $P_{\beta,\mu}$-doubling condition
		\begin{equation}\label{eq:P_doubling_Holder}
			P_{\beta,\mu}(2^NB)\lesssim \Theta_\mu(2^NB),
		\end{equation}
		 elementary calculations show that
			\begin{equation*}
				\begin{split}
					\sum_{j\geq N} 2^{-j\beta}\Theta_\mu(2^jB)&=2^{-N\beta}\sum_{j\geq N}2^{-(j-N)\beta}\Theta_\mu \bigl(2^{j-N}(2^NB)\bigr)\\
					&=2^{-N\beta}\sum_{j\geq0} 2^{-j\beta}\Theta_\mu\bigl(2^j(2^NB)\bigr)\lesssim 2^{-N\beta} \Theta_\mu(2^NB).
				\end{split}
			\end{equation*}
		Thus, if $\omega_A(t)\lesssim t^\beta$, the assumption (2) in Corollary \ref{cor:elliptic_GSTo} is satisfied if we can guarantee that both $B$ and $2^NB$ are $P_{\beta,\mu}$-doubling in the sense of \eqref{eq:P_doubling_Holder}. 
		
 Previous versions of Corollary	\ref{cor:elliptic_GSTo} had been applied to solve the non-variational two-phase problem for harmonic measure and elliptic measure associated with an elliptic operator with H\"older continuous coefficients. In particular, those were needed to analyze the measures at the level of points of zero density. One of the main ingredients was \cite[Lemma 6.1]{AMT17a}, which shows that if $\mu$ is an $n$-dimensional measure, for any $\beta \in (0,1)$, there exists $M>0$ depending on $\beta$ and the dimension, such that for $\mu$-a.e. $x\in \Rn1$, there exists a sequence of $M$-$P_{\beta,\mu}$-doubling balls $B(x,r_i)$ with $r_i\to 0$ as $i\to \infty$.
 In fact, if $A$ is a $\beta$-H\"older continuous matrix, Corollary \ref{cor:elliptic_GSTo} is equivalent to \cite[Theorem 1.2]{Pu19} for the study of the  two-phase problem for the elliptic measure as in \cite{Pu19}. 
 Indeed, if $x_0$ is a point for which we may find a sequence of $M$-$P_{\beta,\mu}$-doubling balls, we fix $B_0=B(x_0,r_0)$ to be a ball with radius $r_0\leq \lambda$. Then, we pick the largest $M$-$P_{\beta,\mu}$-doubling ball $B=B(x_0,\tilde r)$ such that $0< \tilde r \leq 2^{-N} r_0$ for which, if $r_0/2 < 2^{N_0} \tilde r \leq r_0$ and $N \leq N_0$, it holds
 \[
 	P_{\beta,\mu}(2^{N_0} B)\leq M 2^{-{N_0}\beta}\Theta_\mu(2^{N_0}B).
 \]
 Hence, under these circumstances, Corollary \ref{cor:elliptic_GSTo} gives the desired property of  big pieces of uniformly $n$-rectifiable measure in $B$, equivalently to the main results of \cite{GT18} and \cite{Pu19}.
 
 Finally, let us mention that generalizing the one-phase and two-phase problems for elliptic measure as in \cite{AHMMMTV16, PPT21} and \cite{AMTV19, Pu19} to elliptic measures associated with $L_A$, $A\in \widetilde{\DMO}$, presents significant difficulties; for instance, the lack of a proper $T(1)$-theorem for suppressed kernels with such general modulus of continuity. Therefore, those problems should be treated separately.
 
 \vvv

	}
	
	\vv
	\subsection*{Structure of the paper.} In \textit{Section \ref{section:preliminaries_and_notation}} we present the general notation which we adopt in the paper and  recall some properties of Dini functions along with their relation to integral operators with proper reproducing kernels (see Lemma \ref{lemma:lemma_int_op_bd}).
	
	\textit{Section \ref{section:PDE_section}} contains the PDE bulk of the paper. In the first part we describe how the elliptic operator $L_A$ and the gradient of the single layer potential transform under a bilipschitz change of variables. Furthermore, in Lemma \ref{lemma:change_matrix_average} we introduce a specific linear map $S$ so that, given a ball $B\subset \Rn1$, the average of the symmetric part of the transformed matrix $\hat A$ equals the identity matrix on $S^{-1}(B)$. This turns out to be crucial for the proof of Main Lemma I, as it allows to compare  $\nabla_1\Gamma_A$ at the level of a given cube with the Riesz kernel effectively.
	In the second part of Section \ref{section:PDE_section} we adapt the arguments of \cite{DK17} in order to prove that $\nabla_1\Gamma_A$ can be interpreted locally as a Calder\'on-Zygmund kernel (see Lemma \ref{lem:estim_fund_sol}), {and we gather other auxiliary PDE lemmas. 
		
	Subsection \ref{subsec:three-steps} contains the previously described three-step perturbation argument: we deal with the frozen coefficients type estimate in Lemma \ref{lem:main_pw_estimate}, in Lemma \ref{lem:lem_1_prep_main_lemma} and Lemma \ref{lem:lem_2_prep_main_lemma} we prove pointwise bounds which allow us to compare $\nabla_1\Gamma_{\bar A_{x,|x-y|/2}}$ and $\nabla_1\Gamma_{\bar A_{x,\delta/2}}$ for  $\delta>0$, and finally in Lemma \ref{lem:estimate_norm_K3} we implement the techniques based on spherical harmonic decomposition in order to estimate the difference of $\nabla\Theta(\cdot;\bar A_{x,\delta/2})$ and $\nabla\Theta(\cdot;\bar A_{\Omega_Q})$, where $\bar A_{\Omega_Q}$ denotes the integral average of $A$ on a set at the level of al cube $Q$.
}
	
	\textit{Section \ref{section:main_lemma}} covers the proof of the two main lemmas. {
		In Main Lemma I we gather all the results of the previous section and estimate the $L^2(\mu)$-norm of the difference of $T_{\mu,\delta}$ and a proper normalization of $\mathcal R_{\mu,\delta}$. Main Lemma II is the main tool for the proof of Corollary \ref{cor:elliptic_GSTo}: under suitable hypotheses on the measure $\mu$, a duality argument and the local Calder\'on-Zygmund character of $\nabla_1\Gamma_A(x,y)$ allow us to transfer the smallness of the $L^2(\mu)$-mean oscillation of $T_\mu$ to the Riesz transform, which is the crucial step in order to invoke \cite{GT18}.}

	In \textit{Section \ref{section:approxmeasure}} we introduce and discuss the properties of an auxiliary measure $\nu_\varepsilon$, which we obtain as the convolution of the measure $\nu$ supported on a cube with a proper cut-off function. This guarantees that the measure $\nu_\varepsilon$ is absolutely continuous with respect to $\mathcal L^{n+1}$, which entails that the $L^2(\nu_\varepsilon)$-norm of the Riesz transform associated with $\nu_\varepsilon$ applied to a Lipschitz function with compact support is (qualitatively) finite. {This is needed in the last section, as it allows to absorb the norm of the Riesz transform in \eqref{eq:lem71a} in the left hand side of that expression.}
	
	The final \textit{Section \ref{section:final_section}} contains the proof of Theorem \ref{theorem:bound_L2_norm_operators} and its corollaries. In particular, we show how to prove those results combining the lemmas of Section \ref{section:main_lemma} and Section \ref{section:approxmeasure} via the change of variables introduced in Lemma \ref{lemma:change_matrix_average}. 
	
	\vv

\section{Preliminaries and notation}\label{section:preliminaries_and_notation}

\subsection*{General notation}

\begin{itemize}
\item For $\lambda>0$ and an open ball $B=B(x,r)$, we define its dilation $\lambda B\coloneqq B(x,\lambda r)$. Analogously, given a Euclidean cube $Q$ in $\Rn1$ with center $x_Q$ and side-length $\ell(Q)$, we denote by $\lambda Q$ the cube with center $x_Q$ and side-length $\lambda \ell(Q)$.
\item For $0<r\leq R<\infty$, we indicate 
\[
A(x,r,R)\coloneqq B(x,R)\setminus \overline{B(x,r)}= \{y\in \R^{n+1}: r<|x-y|<R\}.
\]
\item We denote by $\mathbb S^n=\partial B(0,1)$ the unit sphere in $\Rn1$, by $\sigma$ its surface measure, and we define $\omega_n\coloneqq \sigma(\mathbb S^n)$. 
\item Given $A\subset \R^d$, we denote by $\chi_A$ its characteristic function. 
\item We endow  the space of matrices $\mathbb R^{n_1\times n_2}$ with the norm $|A|\coloneqq \max_{i,j}|a_{ij}|,$ for $A=(a_{ij})_{i,j}\in \mathbb R^{n_1\times n_2}$.
\item We write $a\lesssim b$ if there is $C>0$ so that $a\leq Cb$, and $a\lesssim_t b$ to specify that the constant $C$ depends on
the parameter $t$. We write $a\approx b$ to mean $a\lesssim b\lesssim a$, and define $a\approx_t b$ similarly. 
\end{itemize}

\vv

\subsection*{Dini functions and integral operators}

\vv

Let $\theta$ be a $\kappa$-doubling function in the sense of \eqref{eq:dini1_new} for $\kappa>0$.
{For $\eta \in \left(0, \frac{1}{2}\right)$ denote by $N_\eta$ the positive integer such that $2^{-N_\eta-1} \leq  \eta < 2^{-N_\eta}$. Hence, if $r>0$ and $\eta r\leq t\leq r$ then
\begin{align*}
 \int_{\eta r}^r \theta(t) \,\frac{dt}{t} & \leq \sum_{\ell=0}^{N_\eta} \int_{2^{-\ell-1}r}^{2^{-\ell}r} \theta(t) \,\frac{dt}{t} \leq \kappa \sum_{\ell=0}^{N_\eta }  \theta(2^{-\ell-1}r) \leq  \kappa \sum_{\ell=0}^{N_\eta } \kappa^{N_\eta-\ell-1} \theta(2^{-N_\eta} r) \\
 &= \frac{\kappa^{N_\eta+1} -1}{\kappa-1}\theta(2^{-N_\eta} r) \leq \kappa  \frac{\kappa^{N_\eta+1} -1}{\kappa-1}\theta(\eta r)\eqqcolon C(\kappa, \eta)\, \theta(\eta r),
\end{align*}
where we used  that $\theta$ is  $\kappa$-doubling. Therefore,
\begin{equation*}
	\int_0^R \theta(t)\, \frac{dt}{t}=\sum_{j=0}^\infty \int_{\eta^{j+1} R}^{\eta^{j}R} \theta(t)\, \frac{dt}{t}\leq C(\kappa, \eta)\, \sum_{j=0}^\infty \theta (\eta^{j+1}R).
\end{equation*}
Moreover, if $r>0$ and $\eta r\leq t\leq r$ it holds
\begin{equation}\label{eq:mod_cont_sum_iteration}
	\begin{split}
		\theta(r) &= \theta(r)\, \avint_{2^{-N_\eta} r}^{r} dt \leq \kappa \frac{\theta(r)}{(1-2^{-N_\eta})r}\,\sum_{\ell=0}^{N_\eta-1} \frac{2^{-\ell}r}{\theta(2^{-\ell}r)} \int_{2^{-\ell-1}r}^{2^{-\ell }r} \theta(t) \frac{dt}{t} \\
		&\leq \frac{\kappa }{(1-2^{-N_\eta})}\,\sum_{\ell=0}^{N_\eta-1} \kappa^{\ell} 2^{-\ell}\int_{2^{-\ell-1}r}^{2^{-\ell }r} \theta(t) \frac{dt}{t} \lesssim \max(1,(\kappa/2)^{N_\eta}) \int_{\eta r}^r \theta(t) \,\frac{dt}{t}.
	\end{split}
\end{equation}
Thus, for $R>0$,
\begin{align}\label{eq:mod_cont_sum}
 \sum_{j=0}^\infty \theta(\eta^{j}R) &\lesssim  \sum_{j=0}^\infty  \max(1,(\kappa/2)^{N_\eta}) \int_{\eta^{j+1} R }^{\eta^j R} \theta(t) \,\frac{dt}{t}\\
 & = \max(1,(\kappa/2)^{N_\eta}) \int_0^R \theta(t)\, \frac{dt}{t}.\notag
\end{align}
}

In particular, $\theta$ belongs to $\DS(\kappa)$ if and only if the doubling property \eqref{eq:dini1_new} holds and $\sum_{j=0}^\infty \theta(2^{-j})<+\infty$. 
One can analogously show that, if $\theta$ verifies  \eqref{eq:dini1_new}, and $0<d\leq n$, we have
\begin{equation}\label{eq:mod_cont_sum_2}
	\sum_{k=1}^\infty \frac{\theta(2^kR)}{(2^kR)^{d}}\lesssim \int_R^\infty \theta(t)\, \frac{dt}{t^{d+1}}, \qquad \qquad R>0.
\end{equation}
{Moreover, by the doubling property of $\theta$,
\begin{equation}\label{eq:omega<dini}
\theta(r) \leq \kappa \int_{r/2}^r \theta(t) \, \frac{dt}{t} \leq \kappa\, \mathfrak I_{\theta}(r),\qquad r>0.
\end{equation}}
\vvv

{
\begin{lemma}\label{lem:mod_cont_large_DS1}
	Assuming that for  fixed $d>0$
	\[
			\mathfrak L^d_{\theta}(t)= t^{d}\int_t^\infty \theta(s)\, \frac{ds}{s^{d+1}} < +\infty \quad \textup{for any}\,\,t>0,
	\]
then $\mathfrak L^d_\theta$ is  a $2^d$-doubling  function. Moreover, it holds that 
\begin{enumerate}
\item If $\mathfrak I_\theta(1)<\infty$ and $\mathfrak L^d_{\theta}(1)<\infty$, then $\mathfrak L^d_\theta\in \DS\left( 2^d \right)$. 
\item If $\mathfrak I_{\mathfrak I_\theta}(1) <\infty$ and $\mathfrak I_{\mathfrak L^d_{\theta}}(1) <\infty$, then $\mathfrak L^d_{\mathfrak I_{\theta}} \in \DS \left( 2^d \right)$.
\end{enumerate}
\end{lemma}

\begin{proof}
  If $t/2 \leq s \leq t$, then 
	\begin{align*}
\mathfrak L^d_{\theta}(t)& = t^d \int_t^\infty\theta(r)\, \frac{dr}{r^{d+1}} \leq 2^d s^d \int_s^\infty\theta(r)\, \frac{dr}{r^{d+1}}\leq 2^d \mathfrak L^d_\theta(s).
	\end{align*}
	which proves that $\mathfrak L^d_\theta$ is $2^d$-doubling. Moreover, to show $(1)$, we use Fubini's theorem to get
	\begin{equation*}
		\begin{split}
			\int_0^1 \mathfrak L^d_\theta(t)\, \frac{dt}{t}&=\int_0^1 t^{d}\int_t^1\theta(s)\, \frac{ds}{s^{d+1}}\, \frac{dt}{t} + \int_0^1 t^{d}\int_1^\infty\theta(s)\, \frac{ds}{s^{d+1}}\, \frac{dt}{t}\\
			&=\frac{1}{d}\int_0^1 \theta(s)\, \frac{ds}{s} + \frac{1}{d}\int_1^\infty \theta(s)\, \frac{ds}{s^{d+1}}<\infty.
		\end{split}
	\end{equation*}
	To prove $(2)$, we apply Fubini's theorem and for $r>0$, it holds
\begin{equation}\label{eq:L_I_theta}
	\begin{split}
	\mathfrak L^d_{\mathfrak I_{\theta}}(r)=	r^d\int_r^\infty \mathfrak I_{\theta}(t)\, \frac{dt}{t^{d+1}}&=\frac{1}{d}\int_0^r \theta(t)\, \frac{dt}{t} + \frac{r^d}{d}\int_r^\infty \theta(t)\, \frac{dt}{t^{d+1}}\\
		&=\frac{1}{d}\bigl(\mathfrak I_\theta(r) + \mathfrak L^d_\theta(r)\bigr)<\infty.
	\end{split}
\end{equation}
	\end{proof}}
\vv

\begin{remark}\label{rem:rem_right_cont}
	If we assume that $\theta \in \DS(\kappa)$, the condition \eqref{eq:mod_cont_sum} implies that $\theta(\eta^{j}R)\to 0$ as $j\to \infty$ for all $R>0$. In particular, the previous lemma implies that if $\theta \in \DS(\kappa)\cap \DL_d(\kappa)$, then $\mathfrak L^d_\theta(\eta^{j}R)\to 0$ as $j\to \infty$ for all $R>0$.
\end{remark}

\vv
\vv

\begin{definition}\label{def:theta_d_kernel}
	Let $\theta$ be a $\kappa$-doubling function and $0<d\leq n+1$.
	We say that a function $K\colon \Rn1\times \Rn1\setminus\{(0,0)\}\to \R$ is a \textit{$(\theta,d)$-kernel} if it is continuous  and there exists $C>0$ such that
	\[
		|K(x,y)|\leq C \frac{\theta(|x-y|)}{|x-y|^{d}} \quad \text{ for }x\neq y.
	\]
\end{definition}

The latter estimate for {$(\theta,d)$-kernels} is directly connected with the Dini integral.
\begin{lemma}\label{lem:lem_estim_dini_integr_growth}
	 {Let $\theta\in \DS(\kappa)$.}
	 Let $0<d\leq n+1$, and assume that $\mu\in M^d_+(\Rn1)$ with $d$-growth constant $c_0>0$. For $\rho>0$ we have
	\begin{equation}\label{eq:lem_estim_dini_integr}
		\int_{B(x,\rho)}\frac{\theta(|x-z|)}{|x-z|^{d}}\, d\mu(z) \lesssim_{c_0,\kappa} \int^\rho_0 \theta(t)\, \frac{dt}{t},
	\end{equation} 
	and the right hand side of \eqref{eq:lem_estim_dini_integr} tends to $0$ as $\rho\to 0$.
\end{lemma}
\begin{proof}
	The proof of \eqref{eq:lem_estim_dini_integr} follows from a standard estimate of the integral on dyadic annuli, the $d$-growth of $\mu$, and \eqref{eq:mod_cont_sum}. Indeed,
	\begin{equation*}
		\begin{split}
			&\int_{B(x,\rho)}\frac{\theta(|x-z|)}{|x-z|^{d}}\, d\mu(z)=\sum_{j=0}^\infty \int_{A(x,2^{-j-1}\rho, 2^{-j}\rho)}\frac{\theta(|x-z|)}{|x-z|^{d}}\, d\mu(z)\\
			&\qquad \qquad \lesssim \sum_{j=0}^\infty \frac{\theta(2^{-j}\rho)}{2^{-jd}\rho^{d}}\mu\bigl(B(x,2^{-j}\rho)\bigr)\lesssim \sum_{j=0}^\infty \theta (2^{-j}\rho)\approx \int^\rho_0 \theta(t)\, \frac{dt}{t}.\qedhere
		\end{split}
	\end{equation*}
\end{proof}

\begin{lemma}\label{lemma:lemma_int_op_bd}
	Let $0<d\leq n+1$ and let  $\mu$ be a measure with compact support in $\Rn1$ and $d$-growth with constant $c_0>0$. If $K$ is a $(\theta,d)$-kernel for {$\theta\in \DS(\kappa)$}, then its associated integral operator
	\[
	Tf(x)\coloneqq \int K(x,y)f(y)\, d\mu(y), \qquad x\in \Rn1
	\]
	is bounded on $L^2(\mu)$. More specifically, if $C>0$ is as in Definition \ref{def:theta_d_kernel} and $R\coloneqq \diam(\supp \mu)$, we have
	\begin{equation}\label{eq:bound_Schurs_Test}
		\|T\|_{L^2(\mu)\to L^2(\mu)}\lesssim_{\kappa, C}c_0 \int_0^R \theta(t)\, \frac{dt}{t}.
	\end{equation}
\end{lemma}
\begin{proof}
	The lemma is a direct consequence of Schur's Test (see for instance \cite[Theorem 6.18]{Folland}) and Lemma \ref{lem:lem_estim_dini_integr_growth}. Indeed, for all $x\in \Rn1$, we have
	\[
	\int |K(x,y)|\, d\mu(y)\lesssim \int \frac{\theta(|x-y|)}{|x-y|^{d}}\, d\mu(y)\lesssim \int_0^{R} \theta(t)\, \frac{dt}{t}<+\infty
	\]
	and, analogously,
	\[
	\int |K(x,y)|\, d\mu(x)\lesssim \int_0^{R} \theta(t)\, \frac{dt}{t} \qquad \text{ for all }y\in\Rn1.\qedhere
	\]
\end{proof}

\vvv

\section{Change of variables and pointwise estimates for the gradient of the fundamental solution}\label{section:PDE_section}

\subsection{Change of variables}
Our arguments involve a change of variables with respect to a particular bilipschitz map, which will be specified later on. For this reason, we state some result concerning how the elliptic operator, its fundamental solution, and its associated gradient of the single layer potential change under such a transformation.

\begin{lemma}[see \cite{AGMT17},  Lemma 2.4 ]\label{lemma:cov_general}
	Let $A$ be a uniformly elliptic matrix with real entries and let $\phi\colon \Rn1\to \Rn1$ be a bilipschitz map. If we denote
	\[
	A_\phi\coloneqq |\det D(\phi)|D(\phi^{-1})(A\circ \phi)D(\phi^{-1})^T,
	\]
	where $D(\cdot)$ denotes the differential matrix, then $A_\phi$ is a uniformly elliptic matrix in $\Rn1$. 
	Furthermore, $u\colon \Rn1\to \R$ is a weak solution of $L_A u=0$ if and only if $\tilde u\coloneqq u\circ \phi$ is a weak solution of $L_{A_\phi}\tilde u=0$ in $\Rn1$.
\end{lemma}

\vv

\begin{lemma}[see \cite{PPT21}, Lemma 5.3]\label{lemma:change_var_fund_sol} Let $\phi\colon\Rn1\to\Rn1$ be a  bilipschitz map and let $A(\cdot)$ be a uniformly elliptic matrix with coefficients in $L^\infty(\Rn1)$. Let $\Gamma_A$ be the fundamental solution of $L_A=-\div(A\nabla\cdot).$ Set
	${A}_\phi\coloneqq  |\det\phi|D({\phi}^{-1})(A\circ\phi)D({\phi}^{-1})^T.$ Then
	\begin{equation}
		\Gamma_{{A}_\phi}(x,y)=\Gamma_A(\phi(x),\phi(y)) \qquad \text{ for }x,y\in\Rn1,
	\end{equation}
	and
	\begin{equation}
		\grad_1 \Gamma_{A_\phi}(x,y)=D(\phi)^T(x)\grad_1\Gamma_A(\phi(x),\phi(y)) \qquad \text{ for } x,y\in\Rn1.
	\end{equation}
\end{lemma}

\vv

\begin{lemma}[see \cite{PPT21}, Lemma 5.4]\label{lemma:potential_changed} Let $\phi\colon\Rn1\to\Rn1$ be a  bilipschitz map. {Let $\mu$ be a Radon measure on $\Rn1$ and $\phi\sharp\mu$ be its image measure.} For every $x\in\Rn1$ we have
	\begin{equation}
		T_\phi\mu(x)=D(\phi)^T(x)T{\phi\sharp\mu}(\phi(x)).
	\end{equation}
\end{lemma}

\vv

 If $\phi$ and $\mu$ are as in the previous lemma, $\nu\coloneqq\phi^{-1}\sharp \mu$, and {for $\delta>0$} we define 
	\[
		\widetilde T_{\phi,\delta}\nu(x)\coloneqq \int_{|\phi(x)-\phi(y)|>\delta} \nabla_1 \Gamma_{A_\phi}(x,y)\, d\nu(y)
	\]
	then the arguments of \cite[Lemma 5.4]{PPT21} show that
	\[
		\widetilde T_{\phi,\delta}\nu(x)= D(\phi)^T(x)T_{\delta}{\mu}(\phi(x)),
	\]
	For $f\in L^1_{\loc}(\nu)$ we also denote \(\widetilde T_{\phi,\nu,\delta}f(x)\coloneqq \widetilde T_{\phi,\delta}(f\nu)(x).\)
	In particular, by \cite[Section 6, p. 740]{PPT21}, if $T_\mu$ is bounded on $L^2(\mu)$ then the operators $\widetilde T_{\phi,\nu,\delta}$ are bounded on $L^2(\nu)$ uniformly on $\delta>0$ and 
	\begin{equation}\label{eq:T_phi_delta_sim_T_delta}
		\|\widetilde T_{\textup{Id},\nu,\delta}\|_{L^2(\nu)\to L^2(\nu)}\approx \|T_{\phi^{-1}, \mu, \delta}\|_{L^2(\mu)\to L^2(\mu)},
	\end{equation}
	where the implicit constant depends on the bilipschitz constant of $\phi$. 	Moreover, if we set  
\[
	\widehat T_{ \phi, \mu,\delta}{f(x)}\coloneqq \int_{| \phi(x)-\phi(y)|> \delta} \nabla_1 \Gamma_A(x,y){f(y)}\,d\mu(y),\, \qquad { f\in L^1_{\loc}(\mu), x\in\Rn1},
\]
 we can prove the following lemma.

\begin{lemma}\label{lemma:change_of_truncations}
Let $A$, $T$, and $\mu$ be as in Theorem \ref{theorem:bound_L2_norm_operators}, and $L \colon\Rn1\to\Rn1$ be an invertible linear map. Then for any  $\delta>0$ 
	\begin{equation}\label{eq:change_of_truncations}
	1+	\|\widehat T_{L, \mu,\delta} \|_{L^2(\mu)\to L^2(\mu)} \approx 1+ \|T_{\mu,\delta}\|_{L^2(\mu)\to L^2(\mu)},
	\end{equation}
	where the implicit constant depends on $\|L\|_{\textup{op} }$, $n, \Lambda, c_0,$ and $\diam(\supp\mu)$.
	\end{lemma}

\begin{proof}
{Let $f\in L^2(\mu)$.} Let $\psi \in C^\infty_c(\Rn1 \setminus B(0,1) )$ be a  function such that $0 \leq \psi \leq 1$ in $\Rn1$, $\psi=1$ in $\Rn1 \setminus B(0,2)$, and $\| \nabla \psi\|_\infty \lesssim 1$, and we  define $\psi_\delta(\cdot)\coloneqq \psi(\frac{\cdot}{\delta})$. By standard estimates, we can prove that
\begin{equation*}
\begin{split}
\Big| \widehat T^\psi_{L,\mu,\delta}f(x)  &- \widehat T_{L,\mu,\delta}f(x)  \Big| \\
&\coloneqq\left| \int \psi_\delta\left(L(x)-L(y) \right) K(x,y) f(y)\, d\mu(y)-\widehat T_{L,\nu,\delta}f(x) \right|  \lesssim \mathcal{M}_\mu f(x),
\end{split}
\end{equation*}
where $\mathcal{M}_\mu$ is the centered Hardy-Littlewood maximal function with respect to $\mu$. Apparently the same estimate holds for $L=\textup{Id}$. Moreover, by the mean value theorem,  if $M_1\coloneqq\min\{\|L\|_{\textup{op}},1\}$, $M_2\coloneqq 2 \max\{\|L\|_{\textup{op}}^{-1},1\}$, it holds  that 
\begin{align*}
\left| \widehat T^\psi_{L,\mu,\delta}f(x) -\widehat T^\psi_{\textup{Id}, \mu,\delta}f(x)  \right| &\lesssim  \|\nabla \psi \|_{L^\infty} \|L-\textup{Id}\|_{\textup{op}} \int_{A(x, M_1\delta, M_2\delta)} \frac{|x-y| }{\delta} \frac{|f(y)|}{|x-y|^n}\,d\mu(y)\\
&\lesssim_{\|L\|_{\textup{op} }} \mathcal{M}_\mu f(x),
\end{align*}
and our result readily follows by triangle inequality, the $L^2(\mu)$-boundedness of $\mathcal{M}_\mu$, and the fact that $\widehat T_{\textup{Id}, \mu,\delta}f= T_{\mu,\delta}f$.
\end{proof}

We remark that, if $A$ is a uniformly elliptic matrix with real entries {and $A_s\coloneqq (A+A^T)/2$ is its symmetric part}, for every $x\in \Rn1$ and $r>0$ the matrix $(\bar A_s)_{x,r}$ is a symmetric and uniformly elliptic matrix with real entries. In particular, it admits a unique square root $\sqrt{(\bar A_s)_{x,r}}$, which is symmetric and uniformly elliptic, too.

\vv
A particularly useful change of variables is the one that transforms the symmetric part of the matrix at a given point into the identity (see \cite[Lemma 2.5]{AGMT17}).
A standard application of Lemma \ref{lemma:cov_general} and {change of variables} allows us to state the following adaptation to the context of the present paper. 
\begin{lemma}\label{lemma:change_matrix_average}
	Let $A$ be a uniformly elliptic matrix with real entries {and let $A_s$ be its symmetric part}. For a fixed point $x\in \Rn1$ and $r>0$, define $S =\sqrt{(\bar A_s)_{x,r}}$. If
	\begin{equation}\label{eq:change_var_matrices_oscillation}
		\hat{A}(\cdot)\coloneqq\frac{|\det S||S^{-1}(B(x,r))|}{|B(x,r)|}S^{-1}(A\circ S)(\cdot)S^{-1},
	\end{equation}
	then $\hat{A}$ is uniformly elliptic, $\avint_{S^{-1}(B(x,r))} \hat A_s=Id$ and $u$ is a weak solution of
	$L_A u = 0$ in $\Rn1$ if and only if $\tilde{u} = u \circ S$ is a weak solution of $L_{\hat{A}}\tilde{u}=0$ in $\Rn1.$
\end{lemma}
\begin{proof}
	In light of Lemma \ref{lemma:cov_general}, we only have to verify that $\avint_{S^{-1}(B(x,r))} \hat A_s=Id$. This follows from a change of variables and the definition of $S$. In particular, we have
	\begin{equation*}
		\begin{split}
			&\avint_{S^{-1}(B(x,r))} \hat A_s(y)\, dy=\frac{|\det S||S^{-1}(B(x,r))|}{|B(x,r)|}S^{-1}\Bigl(\avint_{S^{-1}(B(x,r))}A_s(Sy)\, dy \Bigr) S^{-1}\\
			&\qquad \qquad =\frac{|\det S||S^{-1}(B(x,r))|}{|B(x,r)|}S^{-1}\Bigl(\frac{1}{|S^{-1}(B(x,r))|}\int_{B(x,r)}A_s(z)|\det S^{-1}|\, dz \Bigr) S^{-1}\\
			&\qquad \qquad=S^{-1}\Bigl(\frac{1}{|B(x,r)|}\int_{B(x,r)}A_s(z)\, dz \Bigr) S^{-1}=Id.
		\end{split}
	\end{equation*}
	This concludes the proof of the lemma.
\end{proof}
\vv

The change of variables $S$ defined in the previous lemma is a bilipschitz function with bi-Lipschitz constant $\Lambda^{1/2}$, and it maps balls to ellipsoids. In particular, we have that $\Lambda^{-1/2}\leq |S|\leq \Lambda^{1/2}.$ For more details we refer to \cite[Section 2]{AGMT17}.
{We also remark that in the setting of Lemma \ref{lemma:change_matrix_average} we have that, {for $\zeta\in \Rn1$, $\rho>0$, and }denoting $\tilde \zeta\coloneqq S\zeta$, it holds
	\begin{equation}\label{eq:inclusion_balls_cov}
		B\bigl(\tilde \zeta,\Lambda^{-1/2}\rho/2\bigr)\subset S\bigl(B(\zeta,\rho)\bigr)\subset B\bigl(\tilde \zeta,2\Lambda^{1/2}\rho\bigr).
	\end{equation}
	Furthermore, we {equivalently} have
	\begin{equation}\label{eq:inclusion_B_SB2}
		B(\zeta,\rho)\subset S^{-1}\bigl(B(S\zeta, 2\Lambda^{1/2}\rho)\bigr)\subset B(\zeta,4\Lambda \rho).
	\end{equation}
	
}

The mean oscillation of the transformed matrix under the change of variables in \eqref{eq:change_var_matrices_oscillation} can be controlled according to the following lemma.
\begin{lemma}\label{lemma:comparison_moduli_change_of_variables}
	Let $A$ be a uniformly elliptic matrix with real entries, {$S$ and $\hat A$ be as in Lemma \ref{lemma:change_matrix_average}}, and define 
	\[
	\hat \omega_{\hat A}(\rho)\coloneqq \sup_{z\in \Rn1}\avint_{S^{-1}(B(\zeta,\rho))}\Bigl|\hat A(y)-\avint_{S^{-1}(B(\zeta,\rho))}\hat A\Bigr|\, dy\qquad \text{ for }\rho>0.
	\]
	Then we have that
	\begin{equation}\label{eq:ineq_hatomega_omega}
		\hat \omega_{\hat A}(\rho)\approx_{\Lambda}\omega_A(\rho)\approx_{\Lambda, \kappa}\omega_{\hat A}(\rho) \qquad \text{ for all }\qquad \rho>0,
	\end{equation}
	 {where $\Lambda$ is the uniform ellipticity constant of $A$.} In particular, if $A\in \DMO_s$  then ${\hat A} \in \DMO_s$ as well.
\end{lemma}
\begin{proof}
{The upper bound in the first equality of \eqref{eq:ineq_hatomega_omega} is an easy consequence of the definition of $\hat A$, the change of variables formula, the uniform ellipticity of $A$ and the fact that $|S^{-1}|\approx_\Lambda 1\approx_\Lambda |\det S|$. In particular, for any ball $B\subset \Rn1$ of radius $r(B)$ we have
\begin{equation*}
	\begin{split}
		&\avint_{S^{-1}(B)}\Bigl|\hat A(y)- \avint_{S^{-1}(B)}\hat A\Bigr|\, dy \\
		&\qquad \approx_{\Lambda}\frac{1}{|S^{-1}(B)|^2}\int_{S^{-1}(B)}\Bigl|\int_{S^{-1}(B)}\bigl(S^{-1}A(Sz)S^{-1}-S^{-1}A(Sw)S^{-1}\bigr)\, dw\Bigr|\, dz\\
		&\qquad \lesssim_\Lambda \frac{1}{|B|^2}\int_{B}|S^{-1}|\Bigl|\int_{B}\bigl(A(z)-A(w)\bigr)\, dw\Bigr||S^{-1}|\, dz\lesssim_\Lambda\frac{1}{|B|}\int_B \Bigl|A(z)-\avint_B A\Bigr|\, dz.
	\end{split}
\end{equation*}
The proof of the lower bound is analogous.
}

	In order to prove the second estimate in \eqref{eq:ineq_hatomega_omega}, we observe that, given a ball $B\coloneqq B(\zeta,\rho)$ and denoting $ \tilde B\coloneqq B(S\zeta, 2\Lambda^{1/2}\rho)$, the inclusion \eqref{eq:inclusion_B_SB2}, the first estimate in \eqref{eq:ineq_hatomega_omega} and the doubling assumption yield
	\begin{equation*}
		\begin{split}
			\avint_{B}\Bigl|\hat A(y)-\avint_B\hat A\Bigr|\, dy&\leq \avint_{B}\Bigl|\hat A(y)-\avint_{S^{-1}(\tilde B)}\hat A\Bigr|\, dy + \Bigl|\avint_{S^{-1}(\tilde B)}\hat A - \avint_{B}\hat A\Bigr|\\
			&\lesssim_{\Lambda}\avint_{S^{-1}(\tilde B)}\Bigl|\hat A(y)-\avint_{S^{-1}(\tilde B)}\hat A\Bigr|\, dy \leq \hat \omega_{\hat A}\bigl(2\Lambda^{1/2}\rho\bigr)\lesssim_{\Lambda,\kappa}\omega_A(\rho).
		\end{split}
	\end{equation*}
	The converse inequality can be proved analogously.
\end{proof}
\vv
{	
We recall that, if $A_0$ is a uniformly elliptic matrix with constant coefficients, then 
\begin{equation}\label{eq:theta_sym}
	\Gamma_{A_0}(x,y)=\Gamma_{(A_0)_s} (x,y)\qquad \text{for all } x,y\in \Rn1, x\neq y.
\end{equation}}

Let $A\in \DMO_s$. The quantity $\mathfrak I_{ \omega_A}(r)$ defined in \eqref{eq:def_mathfrak_I} satisfies  $\mathfrak I_{ \omega_A}(2^{-j}r)\to 0$ for every $0<r<1$ by Remark \ref{rem:rem_right_cont}.  {We remark that it is not necessary to assume that $ \omega_A(r)$ vanishes as $r\to 0^+$ for this property to hold.}

\vv

\subsection{Estimates for the gradient of the fundamental solutions.}

The following theorem is an easy adaptation of one of the main results of \cite{DK17}.

\begin{theorem}\label{theorem:dong_kim}
	Let  $A$ be a uniformly elliptic matrix satisfying $A\in \DMO_s$. 
	{Let $0<\eta<1/2$, and set $N\coloneqq 3 (\tfrac{4}{3})^{N_\eta}$ for $N_\eta$ such that $2^{-N_\eta-1}\leq \eta< 2^{-N_\eta}.$} Assume that $g\colon B(0,N+1)\to \R^m$ is a function that satisfies the Dini mean oscillation condition
	\begin{equation}\label{eq:dinig_auxlem}
		\int_0^1\mathring\omega^{0,N}_g(t)\, \frac{dt}{t}<+\infty, \quad \text{ where } \quad \mathring\omega^{0,k}_g(t)\coloneqq  \  \sup_{w\in B(0,k)} \avint_{B(w,t)}|g(x)-\bar g_{w,t}|\, dx, \qquad k\geq 1.
	\end{equation}
	Let $u$ be a weak solution of
	\[
	\div (A(x)\nabla u)=\div g,\qquad \text{ in } B(0,N+1).
	\]
	{There exists $\eta$ such that} $u\in C^1\bigl(\overline{B(0,1)};\R^m\bigr)$.  Furthermore, it holds that
	\begin{equation}\label{eq:DK217}
		\|\nabla u\|_{L^\infty(B(0,2))}\lesssim \|\nabla u\|_{L^1(B(0,4))} + \int_0^1 \mathring\omega^{0,N}_g(t)\, \frac{dt}{t},
	\end{equation}
	and, for $x,y\in B(0,1)$ such that $|x-y|<1/2$,
	\begin{equation}\label{eq:DK219}
		\begin{split}
			&|\nabla u(x)-\nabla u(y)|\lesssim \|\nabla u\|_{L^1(B(0,4))}|x-y|^\beta \\
			&\qquad + \Biggl(\|\nabla u\|_{L^1(B(0,4))} + \int_0^1\mathring\omega^{0,N}_g(t)\, \frac{dt}{t}\Biggr)\int_0^{|x-y|}\mathring\omega^{0,N}_A(t)\, \frac{dt}{t} + \int_0^{|x-y|}\mathring\omega^{0,N}_g(t)\, \frac{dt}{t},
		\end{split}
	\end{equation}
	where $\beta>0$ and the implicit constants depend on $n$.
\end{theorem}
\begin{proof}
	For the proof of the fact that $u$ belongs to $C^1\bigl(\overline{B(0,1)}\bigr)$, we refer to \cite[Theorem 1.5]{DK17}. 	The inequality \eqref{eq:DK217} is a variant of \cite[(2.17)]{DK17}, which is formulated in terms of a slightly different modulus of oscillation {(see \cite[(2.15)]{DK17}).} In order to prove \eqref{eq:DK217}, 	we fix an exponent $0<p<1$ and  define
	\[
	\phi(\bar x, r)\coloneqq \inf_{c\in \Rn1}\Biggl(\avint_{B(\bar x, r)}|\nabla u(z)-c|^p\, dz\Biggr)^{1/p}, \qquad \bar x\in \Rn1, \, 0<r<1/3.
	\]
	
	It was proved in \cite[p. 424]{DK17} that, for $0<\eta<1/2$ sufficiently small (depending on $p$ and an absolute constant) and $j=1,2, \ldots$,
	\begin{equation}\label{eq:estim_phi_omega_g}
		\phi\bigl(\bar x, \eta^j r\bigr)\lesssim 2^{-j}\phi(\bar x, r) + \|\nabla u\|_{L^\infty(B(x,r))} \sum_{i=1}^j 2^{1-i}\mathring\omega^{0,3}_A\bigl(\eta^{j-i}r\bigr) + \sum_{i=1}^j 2^{1-i}\mathring\omega^{0,3}_g\bigl(\eta^{j-i}r\bigr).
	\end{equation}
	Now notice that	
		\begin{align}\label{eq:doublesumdk}
			\sum_{j=1}^\infty	\sum_{i=1}^j 2^{1-i}\mathring\omega^{0,3}_g\bigl(\eta^{j-i}r\bigr) &= \sum_{i=1}^\infty	\sum_{j=i}^\infty 2^{1-i}\mathring\omega^{0,3}_g\bigl(\eta^{j-i}r\bigr)=2 \sum_{j=0}^\infty\mathring\omega^{0,3}_g\bigl(\eta^{j}r\bigr).
		\end{align}

		{By \cite[p. 495]{Li17}, there exists a constant $\kappa>0$ depending only on  $n$ such that
		\begin{equation}
			\mathring \omega^{0,3}_g(s)\leq \kappa \mathring \omega^{0,4}_g(t), \qquad \text{ for }\frac{s}{2}\leq t<s.
		\end{equation}	
		Moreover, by definition of $N$ and a minor variant of \eqref{eq:mod_cont_sum_iteration} and \eqref{eq:mod_cont_sum}, we have
		\begin{equation}\label{eq:doubling_localized}
			\sum_{j=0}^\infty\mathring{\omega}^{0,3}_g(\eta^jr)\lesssim \int^r_0 \mathring \omega^{0,N}_g(t)\, \frac{dt}{t}.
		\end{equation}
		Analogous properties hold for $\mathring{\omega}^{0,3}_A.$}
		Thus, 
		\begin{equation*}
			\begin{split}
				\sum_{j=1}^\infty\sum_{i=1}^j 2^{1-i}\mathring\omega^{0,3}_g\bigl(\eta^{j-i}r\bigr)\overset{\eqref{eq:doublesumdk}}{\leq} 2 \sum_{j=0}^\infty\mathring \omega^{0,3}_g\bigl(\eta^{j}r\bigr)
				\overset{\eqref{eq:doubling_localized}}{\lesssim} \int_0^r \mathring\omega^{0,N}_g(t)\, \frac{dt}{t}.
			\end{split}
		\end{equation*}
	
	The same computations can be repeated to handle the second summand on the right hand side of \eqref{eq:estim_phi_omega_g}. Finally, the inequalities \eqref{eq:DK217} and \eqref{eq:DK219} follow as in \cite[(2.17) and (2.19)]{DK17}.
\end{proof}

\vv

\begin{remark}\label{rem:pwbound1}
Let $0<r<R_0$, $x_0\in \Rn1$,  and $g\colon B(x_0, (N +1) R_0)\to \R$, where {$N\coloneqq 3 (\tfrac{4}{3})^{N_\eta}$ for $\eta$ as in \eqref{eq:estim_phi_omega_g} and $N_\eta$ such that $2^{-N_\eta-1}\leq \eta< 2^{-N_\eta}$}. For $0<t<r$ and $1\leq k\leq N$, we define  
	\begin{equation*}
		\begin{split}
			 \mathring\omega^{x_0, k r}_g(t) &= \sup_{w\in B(x_0, k  r)} \avint_{B(w, t r)}\bigl|g(x) - \bar g_{w,tr}\bigr|\, dx.
		\end{split}
	\end{equation*}

	Let  $A$ be a uniformly elliptic such that  $A\in \DMO_s$ and assume that, for $N$ as above,
	\[
		\int_0^1 \mathring\omega^{x_0, N R_0}_g(t)\, \frac{dt}{t}<\infty.
	\] 
	Then by the proof of Theorem \ref{theorem:dong_kim}, it holds that  if  $u$ is a weak solution to $-\div (A(x)\nabla u)=-\div g$ in $B(x_0, (N+1) R_0)$, we obtain that $u\in C^1(\overline{B(x_0, r)})$ and satisfies the following estimates:
		\begin{equation}\label{eq:Linftyest-r}
		\|\nabla u\|_{L^\infty(B(x_0,2r))}\lesssim_{R_0} \avint_{B(x_0, 4 r)} |\nabla u(x)|\, dx+ \int_0^1 \mathring\omega^{x_0, N r}_g(t)\, \frac{dt}{t},
\end{equation}
	and, for $x,y\in B(x_0,r)$ such that $|x-y|<r/2$,
	\begin{equation}\label{eq:ModulusCont-r}
		\begin{split}
			&|\nabla u(x)-\nabla u(y)|\lesssim_{R_0}  \frac{|x-y|^\beta}{r^\beta}\avint_{B(x_0,4 r)}|\nabla u(z)|\, dz \\
			&\qquad\qquad \qquad  + \Biggl(\avint_{B(x_0,4r)}|\nabla u(z)|\, dz  + \int_0^1\mathring\omega^{x_0,  Nr}_g(t)\, \frac{dt}{t}\Biggr)\int_0^{\frac{|x-y|}{r}}\mathring\omega^{x_0,  N r}_A(t)\, \frac{dt}{t} \\
			&\qquad\qquad \qquad + \int_0^{\frac{|x-y|}{r}}\mathring\omega^{x_0,  N r}_g(t)\, \frac{dt}{t}.
		\end{split}
	\end{equation}
	{Furthermore, we have that the} implicit constants  blow up {logarithmically} as $R_0\to \infty$. If $0<R_0<1$, they only depend on ellipticity, dimension, and the Dini Mean Oscillation condition.
\end{remark}

	\vv

\vv

An important consequence of the pointwise bounds of Theorem \ref{theorem:dong_kim}, are the following estimates for the fundamental solution and its derivatives.
\begin{lemma}\label{lem:estim_fund_sol}
	Let $A(\cdot)=(a_{ij})$	be a uniformly elliptic matrix in $\Rn1$, $n\geq 2$, satisfying $A\in \DMO_s$. Let $R>0$, and let $\beta>0$ be as in Theorem \ref{theorem:dong_kim}.
	Then there exists $C=C(n,\Lambda,R)>0$ such that the fundamental solution $\Gamma_A$ satisfies the following pointwise bounds:
	\begin{enumerate}
		\item \(|\Gamma_A(x,y)|\leq C |x-y|^{-(n-1)}\) for $x,y\in \Rn1$, $0<|x-y|<R$.
		\item \(|\nabla_1\Gamma_A(x,y)| + |\nabla_2\Gamma_A(x,y)|\leq C {|x-y|^{-n}}\)  for $x,y\in \Rn1$, $0<|x-y|<R$.
		\item \(|\nabla_1\nabla_2\Gamma_A(x,y)|\leq C |x-y|^{-(n+1)}\)  for $x,y\in \Rn1$, $0<|x-y|<R$.
		\item {We have
		\begin{align}\label{eq:continuityGamma}
		|\nabla_1 \Gamma_A(x,y) - \nabla_1\Gamma_A(x,z)| &+ |\nabla_1\Gamma_A(y,x) - \nabla_1\Gamma_A(z,x)| \notag\\
		& \leq C \left( \frac{|y-z|^\beta}{|x-y|^\beta}+ \int_0^{\frac{|y-z|}{|x-y|}} \omega_A(t)\,\frac{dt}{t}\right)  |x-y|^{-n},
		\end{align}
		for  $2|y-z|\leq |x-y|<R$.}
	\end{enumerate}
\end{lemma}

\begin{proof}
{For the proof of (1) we refer to \cite[Section 5]{HK07}. The bounds (2) and (4) follow directly from (1), the fact that the function $u_y(\cdot)\coloneqq \Gamma_A(\cdot,y)$ satisfies $L_A u_y=0$ in $B(x,|x-y|/8)$, and Remark \ref{rem:pwbound1}. The bound for $\nabla_1\nabla_2\Gamma_A$ can be proved analogously, observing that $w_y(\cdot)\coloneqq \nabla_2 \Gamma_A(\cdot,y)$ satisfies $L_A w_y=0$ in $B(x,|x-y|/8)$.	
}
\end{proof}

\vvv

{	 
		
		If $\Omega\subset \Rn1$, $n\geq 2$, is an open set and $2^*\coloneqq \frac{2(n+1)}{n-1}$, we define $Y^{1,2}(\Omega)$ as the space of all weakly differentiable functions $u\in L^{2^*}(\Omega)$,  whose weak derivatives belong to $L^2(\Omega)$. We endow $Y^{1,2}(\Omega)$ with the norm
	\[
	\|u\|_{Y^{1,2}(\Omega)}\coloneqq \|u\|_{L^{2^*}(\Omega)} + \|\nabla u\|_{L^2(\Omega)},\qquad u\in Y^{1,2}(\Omega).
	\]
	We denote by $Y^{1,2}_0(\Omega)$ the closure of $C^\infty_c(\Omega)$ in $Y^{1,2}(\Omega)$ and remark that $Y^{1,2}(\Rn1)=Y^{1,2}_0(\Rn1)$ (see for instance \cite[p. 46]{MZ97}).
	
	Let $A$ be a uniformly elliptic matrix in $\Rn1$, $n\geq 2$, with coefficients in $L^\infty(\Rn1)$. For $g \in L^{2}(\Rn1;\Rn1)$, we denote by $L_A^{-1}\nabla \cdot  g$ the unique solution $u$ of the variational Dirichlet problem $L_A u= -\dv g$ and $ u \in Y_0^{1,2}(\Rn1)$. By a modification of the argument that proves \cite[(3.47)]{HK07} (for the details see, for instance, the proof of \cite[(6.3)]{Mo19}), for $g \in L^2(\Rn1)$, we have
	\begin{equation}\label{eq:def_L_A_inverse}
		L_A^{-1} \nabla \cdot g(x)=\int \nabla_2 \Gamma_A(x,y) \cdot g(y)\, dy,
	\end{equation}
	and as in \cite[(3.10)-(3.11)]{HK07}, one can prove  that 
	\begin{equation}\label{eq:L2_boundedness_nabla_L_inv}
		\|\nabla L_A^{-1} \nabla \cdot g \|_{L^2(\R^{n+1})} \lesssim \|g \|_{L^2(\R^{n+1})}.
	\end{equation}

	\vv
	
	The proof of the next lemma is a standard adaptation of the one of \cite[Corollary 3.5]{HK07}, which we present below for the reader's convenience.
	\begin{lemma}\label{lem:G-tildeG}
		If $A$ and $\tilde A$	are uniformly elliptic matrices in $\Rn1$, $n\geq 2$, so that  $A,\tilde A\in \DMO_s$, then for $R>0$ and all $x, y \in \Rn1$ such tha $0<|x-y|<R$, it holds that
		\begin{equation}\label{eq:representation-perturbation}
				\Gamma_{\tilde A}(x,y)-\Gamma_{A}(x,y) =\int \nabla_2\Gamma_A(x,z)\cdot\bigl(A(z)-\tilde   A(z)\bigr)\nabla_1\Gamma_{\tilde A}(z,y)\, dz.
		\end{equation}
	\end{lemma}	
	\begin{proof}
		Set $\Gamma\coloneqq \Gamma_A$, $\tilde \Gamma\coloneqq \Gamma_{\tilde A}$, $\Gamma_*\coloneqq \Gamma_{A^T}$, $\tilde \Gamma_*\coloneqq \Gamma_{\tilde A^T}$,  and $r\coloneqq |x-y|/4$. For $0<\rho<r$ we denote as $\Gamma^\rho$ the \textit{averaged fundamental solution} (see \cite[Section 3.1]{HK07}) which can be defined, via Lax-Milgram theorem, as the unique function in $Y^{1,2}_0(\Rn1)$ such that
		\begin{equation}\label{eq:av_matrix_def}
			\int A\nabla_1 \Gamma^\rho(\cdot,y)\cdot \nabla u = \avint_{B(y,\rho)} u = \int f_{\rho,y} u  \qquad \text{for all } u\in Y^{1,2}_0(\Rn1),
		\end{equation}
		for $f_{\rho,y}(z)\coloneqq |B(y,\rho)|^{-1}\chi_{B(y,\rho)}(z)$. Note that, if $B$ is a ball such that $B \cap B(y,\rho)=\varnothing$ and $u\in  C^\infty_c(B)$, then the right hand side of \eqref{eq:av_matrix_def} vanishes, i.e.,  $L_A\Gamma^\rho(\cdot, y)=0$ in $B$. Thus, the De Giorgi-Nash-Moser theorem implies that $\Gamma^\rho(\cdot, y)$ is locally H\"older continuous in $\Rn1\setminus \overline{B(y,\rho)}$. Moreover, by \cite[(3.45)]{HK07}, $\Gamma^\rho$ admits the representation
		\begin{equation}\label{eq:def_gamma_rho_av}
			\Gamma^\rho(x,y)\coloneqq \avint_{B(y,\rho)}\Gamma (x,z)\, dz.
		\end{equation}
		We  define $\tilde \Gamma^\rho ,  \Gamma_*^\rho , \tilde \Gamma_*^\rho \in  Y^{1,2}_0(\Rn1)$ analogously and it holds 
		\begin{equation}\label{eq:av_matrix_tilde_transp}
			\int \nabla_1 \tilde \Gamma_*^\rho(\cdot,x)\cdot \tilde A \nabla u = \avint_{B(x,\rho)} u \qquad \text{for all } u\in Y^{1,2}_0(\Rn1).
		\end{equation}
	
		We claim that 
		\begin{equation}\label{eq:claim_wt_Gamma_rho}
			 \Gamma^\rho(x,y)=\int \nabla_2 \tilde \Gamma(x,\cdot)\cdot \tilde A(\cdot)\nabla_1  \Gamma^\rho(\cdot, y).
		\end{equation}
		Let  $0<\rho'<\rho<r$. If we apply \eqref{eq:av_matrix_tilde_transp} for $\tilde \Gamma^{\rho'}_*$ and $u=\Gamma^\rho(\cdot, y) \in Y^{1,2}_0(\Rn1)$, then it holds that
		\begin{equation}\label{eq:avint_rho_prime_Gamma_rho}
			\int \nabla_1 \tilde \Gamma^{\rho'}_*(z, x)\cdot  \tilde A\nabla_1 \Gamma^{\rho}(z, y)\, dz = \avint_{B(x,\rho')}\Gamma^\rho(z, y)\, dz.	
		\end{equation}
		Since $\Gamma^\rho(\cdot, y)$ is continuous away from $y$, by Lebesgue's differentiation theorem,
		\begin{equation}\label{eq:ldt_Gamma_rho}
			\lim_{\rho'\to 0}\avint_{B(x,\rho')}\Gamma^\rho(z,y)\, dz=\Gamma^\rho(x,y). 
		\end{equation}
		Moreover, by \cite[(3.22) and (3.24)]{HK07}, there exist $C_1, C_2 >0$ such that
		\begin{equation}\label{eq:bound_L2_away}
			\int_{\Rn1\setminus B(x,r)}\bigl|\nabla_1 \Gamma^{\rho'}_{\tilde A^T}(z,x)\bigr|^2\, dz\leq C_1 r^{1-n}
		\end{equation}
		and
		\begin{equation}\label{eq:bound_Lp_ball}
			\bigl\|\nabla_1 \Gamma^{\rho'}_{\tilde A^T}(\cdot, x)\bigr\|_{L^p(B(x,r))}\leq C_2 r^{\frac{n+1}{p}-n}\qquad \text{ for }p\in [1,(n+1)/n),
		\end{equation}
		where $C_1$ and $C_2$ do not depend on $\rho'$.

	We now split 
		\begin{equation}\label{eq:lim_rho_1}
			\begin{split}
			\int \nabla_1 & \tilde \Gamma^{\rho'}_*(z, x)\cdot  \tilde A\nabla_1 \Gamma^{\rho}(z, y)\, dz \\				&=\int_{B(x,r)} + \int_{\Rn1\setminus B(x,r)} \nabla_1\tilde \Gamma^{\rho'}_*(z, x)\cdot  \tilde A\nabla_1 \Gamma^{\rho}(z, y)\, dz\eqqcolon I_{\rho'} + II_{\rho''}.
			\end{split}
		\end{equation}
	Since $A\in \DMO_s$ and $L_A\Gamma^\rho(\cdot, y)=0$ in $B(x,r)$, by Lemma \ref{lem:estim_fund_sol} and \eqref{eq:def_gamma_rho_av}, for any $z\in B(x,r)$, it holds that 
		\[
			|\nabla_1 \Gamma^\rho_A(z,y)| \leq  \avint_{B(y,\rho)}| \nabla_1 \Gamma (z,w)|\, dw \lesssim \avint_{B(y,\rho)}\frac{1}{|w-z|^{n}}\, dw\approx r^{-n}\approx |x-y|^{-n},
		\]
		where in the penultimate inequality we used that $|w-z|\approx r$ for $w\in B(y,\rho)$, and so  $\nabla_1 \Gamma^\rho(\cdot, y)\in L^\infty(B(x,r)).$
		
		Fix  $p\in \left(1,\frac{n+1}{n}\right)$ and  consider a sequence $\rho'_j\to 0$ such that $\rho_j'< \rho <r$ for all $j$. As in the proof of \cite[Theorem 3.1]{HK07}, by \eqref{eq:bound_Lp_ball} and weak compactness of $W^{1,p}( B(x,r))$, we may pass to a subsequence, which we still denote by $\rho'_{j}$, such that $\nabla_1 \tilde \Gamma_*^{\rho'_j}(\cdot, x)\rightharpoonup \nabla_1 \tilde \Gamma_*(\cdot, x) $ in $L^p(B(x,r))$. Hence
		\begin{equation}\label{eq:limit_I_rho}
			\lim_{j\to \infty} I_{\rho'_j}=\int_{B(x,r)}\nabla_1 \tilde \Gamma_*(z, x)\cdot \tilde A \nabla_1 \Gamma^{\rho}(z, y)\,dz.
		\end{equation}
	Furthermore, once again as in the proof of \cite[Theorem 3.1]{HK07}, the bound \eqref{eq:bound_L2_away} implies that,  by passing to another subsequence if necessary, $\nabla_1 \tilde \Gamma_*^{\rho'_j}(\cdot, x)\rightharpoonup \nabla_1 \tilde \Gamma_*(\cdot, x) $ in $L^2(\Rn1 \setminus B(x,r))$. Thus, since $\Gamma^{\rho}(\cdot, y)\in Y_0^{1,2}(\Rn1),$
		\begin{equation}\label{eq:limit_II_rho}
			\lim_{j\to \infty} II_{\rho'_{j}}=\int_{\Rn1\setminus B(x,r)}\nabla_1 \tilde \Gamma_*(z, x)\cdot \tilde A \nabla_1 \Gamma^{\rho}(z, y)\,dz.
		\end{equation}
	Therefore,  \eqref{eq:limit_I_rho}, \eqref{eq:limit_II_rho}, and \eqref{eq:ldt_Gamma_rho}, imply
	\begin{equation}\label{eq:3.29}
			\Gamma^\rho(x,y)=\int \nabla_1  \tilde \Gamma_*(z, x)\cdot \tilde A \nabla_1 \Gamma^{\rho}(z, y)\,dz= \int \nabla_2  \tilde \Gamma(x, z)\cdot \tilde A \nabla_1 \Gamma^{\rho}(z, y)\,dz,
		\end{equation}
		where in the last equality we used that $\Gamma_{\tilde A^T}(w,z)=\Gamma_{\tilde A} (z,w)$ for all $z,w\in \Rn1$, $z\neq w$ (see \cite[(3.43)]{HK07}). This concludes the proof of  \eqref{eq:claim_wt_Gamma_rho}.
		
Let us now  split
		\[
			\int \nabla_2 \tilde \Gamma(x,\cdot )\cdot A(\cdot)\nabla_1  \Gamma^{\rho_j}(\cdot, y)=\int_{B(x,r)} + \int_{B(y,r)} + \int_{\Rn1\setminus (B(x,r)\cup B(y,r))}\eqqcolon I^1_{\rho_j}+I^2_{\rho_j}+I^3_{\rho_j}.
		\]
Since  $\tilde \Gamma(x, \cdot )\in Y^{1,2}(\Rn1 \setminus B(x,r))$, by the weak convergence $\nabla_1  \Gamma^{\rho_j}(\cdot, y) \rightharpoonup \nabla_1  \Gamma(\cdot, y)$ in $L^2(\Rn1 \setminus B(y,r))$, we have that 
$$
\lim_{j \to \infty} I^3_{\rho_j} = \int_{\Rn1\setminus (B(x,r)\cup B(y,r))} \nabla_2 \tilde \Gamma(x, z )\cdot A(z)\nabla_1  \Gamma(z, y)\,dz.
$$ 
By Lemma \ref{lem:estim_fund_sol}, it holds that $|\nabla_2 \tilde \Gamma(x, z)| \lesssim |x-y|^{-n}$, for all $z \in B(y,r)$, and since $\nabla_1  \Gamma^{\rho_j}(\cdot, y) \rightharpoonup \nabla_1  \Gamma(\cdot, y) $ in $L^p(B(y,r))$, we get that 
$$
\lim_{j \to \infty} I^2_{\rho_j} = \int_{ B(y,r)} \nabla_2 \tilde \Gamma(x, z )\cdot A(z)\nabla_1  \Gamma(z, y)\,dz.
$$ 
Lastly,  by Lemma \ref{lem:estim_fund_sol}, $\nabla_1  \Gamma^{\rho_j}(\cdot,y)$ is a uniformly bounded equicontinuous family of functions  in  $B(x,r)$, and so, after passing to a subsequence, we get that $\nabla_1  \Gamma^{\rho_j}(\cdot,y) \to \nabla_1  \Gamma(\cdot,y)$ uniformly in $B(x,r)$. By \cite[(3.52)]{HK07}, it holds that 
$$
\|\nabla_2 \tilde \Gamma(x, \cdot)\|_{L^1(B(x,r))} \lesssim r,
$$ 
and thus, we deduce that 
$$
\lim_{j \to \infty} I^1_{\rho_j} = \int_{ B(x,r)} \nabla_2 \tilde \Gamma(x, z )\cdot A(z)\nabla_1  \Gamma(z, y)\,dz.
$$ 
Hence, as $\lim_{j \to \infty}  \Gamma^{\rho_j}(x,y) =   \Gamma(x,y)$, by \eqref{eq:3.29}, we conclude that
\begin{equation} \label{eq:tilde_gamma_limit_1}
			 \Gamma(x,y)=\int \nabla_2 \tilde  \Gamma(x, \cdot)\cdot  A\nabla_1  \Gamma(\cdot, y).
\end{equation}
		Similarly, we can be prove  that
		\begin{align*}
			\tilde \Gamma(x,y)=		\tilde \Gamma_*(y,x)&=\int \nabla_2   \Gamma_*(y, \cdot)\cdot  \tilde A^T\nabla_1 \tilde \Gamma_*(\cdot,x)= \int \nabla_2 \tilde \Gamma(x,\cdot)\cdot \wt A\nabla_1  \Gamma(\cdot, y),
		\end{align*}
		which, together with \eqref{eq:tilde_gamma_limit_1}, concludes the proof of the lemma.
	\end{proof}
	
}

\vv

 For a real elliptic $(n+1)\times (n+1)$-matrix $A_0$ with constant coefficients, we use the notation $\Theta(x,y; A_0)=\Gamma_{A_0}(x,y)$ to indicate the fundamental solution of $L_{A_0}$. In particular, we recall that $\Theta(x,y;A_0)=\Theta(x-y,0;A_0)$ and
	\begin{equation}\label{eq:fund_sol_const_matrix}
		\Theta(z,0;A_0)=\Theta(z,0;A_{0,s})=
		\begin{cases}\displaystyle
			\frac{-1}{(n-1)\omega_n\sqrt{\det A_{0,s}}}\frac{1}{\langle A_{0,s}^{-1}z, z\rangle^{(n-1)/2}} \;\;\;\text{ for }n\geq 2, \\\\
			\displaystyle
			\frac{1}{4\pi\sqrt{\det A_{0,s}}}\log\big(\langle A_{0,s}^{-1}z, z\rangle\big)\;\;\text{ for }n=1,
		\end{cases}
	\end{equation}
where $A_{0,s}\coloneqq\frac12(A_0+A_0^T)$ is the symmetric part of $A_0$. Moreover it holds
	\begin{equation}\label{eq:gradient_fund_sol_const_matrix}
		\nabla_1 \Theta(z,0;A_0)=\frac{\omega_n^{-1}}{\sqrt{\det A_{0,s}}}\frac{A_{0,s}^{-1}\,z}{\langle A_{0,s}^{-1}z, z\rangle^{(n+1)/2}}, \qquad \text{ for }z\in \mathbb R^{n+1}\setminus\{0\}.
	\end{equation}

Finally we observe that, for any integer $k\geq 0$, we have
\begin{equation}\label{eq:estim_deriv_fund_sol}
	\bigl|\nabla^k_1 \Theta (z,0;A_0)\bigr|\lesssim \frac{1}{|z|^{n+k-1}}\qquad \text{ for } z\in \mathbb R^{n+1}\setminus \{0\},
\end{equation}
where the implicit constant depends on dimension, the ellipticity constants of $A_{0,s}$,   and the order of differentiation $k$.

\begin{lemma}
	Let $A=(a_{ij})_{1\leq i,j\leq n+1}$ be a  matrix  such that $a_{ij} \in L^\infty(\Rn1)$ for $1\leq i,j \leq n+1$ and $A\in \DMO_s$. Then, for $r>0$ and $p\geq 1$,
		\begin{equation}\label{eq:sup_aux_1}
			\sup_{0< \rho\leq r}\sup_{x\in \Rn1}	\biggl(\avint_{B(x,\rho)}|A(z)-\bar A_{x,\rho}|^p\, dz\biggr)^{1/p}\lesssim p\,\mathfrak I_{\omega_A}(r).
		\end{equation}
		Moreover, if  $\mathcal C_{j}\coloneqq A(x,2^jr, 2^{j+1}r),   j \in \mathbb N$,  we have
		\begin{equation}\label{eq:sup_aux_2}
			\biggl(\avint_{\mathcal C_j}|A(z)-\bar A_{x,r}|^p\, dz\biggr)^{1/p}\lesssim p \,\mathfrak I_{\omega_A}(2^jr).
		\end{equation}
\end{lemma}
\begin{proof}
	By the John-Nirenberg inequality (see for instance \cite[p. 144]{St93}), it holds
	\[
		\sup_{0< \rho\leq r}\sup_{x\in \Rn1}	\biggl(\avint_{B(x,\rho)}|A(z)-\bar A_{x,\rho}|^p\, dz\biggr)^{1/p}\lesssim p \sup_{0< \rho\leq r}\sup_{x\in \Rn1}\avint_{B(x,\rho)}|A(z)-\bar A_{x,\rho}|\, dz.
	\]
	Let $x\in \Rn1$, $r>0$, and $\rho\in [0,r]$. We denote by $N$ the positive integer such that $2^N\rho\leq r<2^{N+1}\rho$.
	Then
	\begin{equation*}
		\begin{split}
			&\avint_{B(x,\rho)}|A(z)-\bar A_{x,\rho}|\, dz\leq 2 \,\avint_{B(x,\rho)}|A(z)-\bar A_{x,r}|\, dz\\
			&\qquad\qquad\lesssim 2 \,\avint_{B(x,\rho)}|A(z)-\bar A_{x,2\rho}|\, dz + \sum_{j=2}^{N+1} |\bar A_{x,2^j\rho}-\bar A_{x,2^{j-1}\rho}| + |\bar A_{x,2^{N+1}\rho}-\bar A_{x,r}|\\
			&\qquad\qquad\lesssim \sum_{j=1}^{N+1}\omega_A(2^j\rho)	\lesssim \int_0^{2^{N+1}\rho}\omega_A(t)\, \frac{dt}{t}\leq \int_0^{2r}\omega_A(t)\, \frac{dt}{t}\lesssim \mathfrak I_{\omega_A}(r),
		\end{split}
	\end{equation*}
	where the last bound is a consequence of the doubling property of $\omega_A$.
	Thus, taking the supremum for $\rho\in [0,r)$, we obtain \eqref{eq:sup_aux_1}.
	
	Let us now prove \eqref{eq:sup_aux_2}. Let $B^j_k$, $k=1,\ldots, M_n$, be a collection of balls of radius $\tfrac{5}{4}2^j r$ which covers $\mathcal C_j$, for a dimensional constant $M_n>1$. Fix a ball $\widetilde B$ in this family, denote by $\tilde x$ its center, and define $L\coloneqq \bigl\{t x+ (1-t)\tilde x:t\in [0,1]\bigr\}$.
	There exists a sequence of balls $B_0,\ldots, B_j$ centered at $L$ such that $B_0=B(x,r)$, $B_j=\widetilde B$, $r(B_k)\approx 2^k r$, and $ B_{k+1}\cap  B_k\neq \varnothing$.
	
	Moreover, for every $k=1,\ldots, M_n$, there exists a ball $B'_k$ centered at $L\cap  B_k\cap  B_{k+1}$ such that $r(B'_k)\approx 2^kr$ and $B_k\cup B_{k+1}\subset B'_k$. Hence, if we denote $\bar A_{B_k}\coloneqq \avint_{B_k}A$, it holds
	\begin{equation}\label{eq:bound_chain}
		\begin{split}
			\bigl|\bar A_{B_k}- \bar A_{B_{k+1}}\bigr|& \leq \bigl|\bar A_{B_k}- \bar A_{B'_{k}}\bigr| +  \bigl|\bar A_{B_{k+1}}- \bar A_{B'_{k}}\bigr| \\
			&\leq \avint_{B_k}|A(z)-\bar A_{B'_{k}}|\, dz + \avint_{B_{k+1}}|A(z)-\bar A_{B'_{k}}|\, dz\\
			&\lesssim \avint_{B'_k}|A(z)-\bar A_{B'_k}|\, dz \lesssim \omega_A(r(B'_k))\lesssim \omega_A(2^kr),
		\end{split}
	\end{equation}
	where the last bound follows from the doubling property of $\omega_A$.
	Thus, 
	\begin{equation}\label{eq:diff_av_triangle_chain}
		\bigl|\bar A_{B(x,r)}-\bar A_{\widetilde B}\bigr|\leq \sum_{k=0}^j\bigl|\bar A_{B_{k+1}}-\bar A_{B_k}\bigr|\overset{\eqref{eq:bound_chain}}{\lesssim} \sum_{k=0}^j\omega_A(2^kr)\overset{\eqref{eq:mod_cont_sum}}{\lesssim} \mathfrak I_{\omega_A}(2^j r),
	\end{equation}
and so
	\begin{equation}\label{eq:lemma_estim_avint_annuli}
		\begin{split}
			&\biggl(\avint_{\mathcal C_j}|A(z)-\bar A_{x,r}|^p\, dz\biggr)^{1/p}\\
			&\qquad \leq	 \sum_{k=0}^{M_n}\biggl(\avint_{B^j_k}|A(z)-\bar A_{B^j_k}|^p\, dz\biggr)^{1/p} +  \sum_{k=0}^{M_n} \bigl|\bar A_{B(x,r)}-\bar A_{B^j_k}\bigr|\\
			&\qquad \overset{\eqref{eq:diff_av_triangle_chain}}{\lesssim_n} \sum_{k=0}^{M_n}\biggl(\avint_{B^j_k}|A(z)-\bar A_{B^j_k}|^p\, dz\biggr)^{1/p} +  \mathfrak I_{\omega_A}(2^j r).
		\end{split}
	\end{equation}
	Observe that, for $k=1,\ldots, M_n$, if we apply \eqref{eq:sup_aux_1} we have
	\begin{align}
		\biggl(\avint_{B^j_k}\bigl|A(z)-\bar A_{B^j_k}\bigr|^p\, dz\biggr)^{1/p} &\lesssim p\,\mathfrak I_{\omega_A}\bigl( r(B^j_k)\bigr) \lesssim p\, \mathfrak I_{\omega_A}(2^j r),\label{eq:lemma_estim_avint_annuli2}
	\end{align}
where in the last inequality we used the doubling property of $\mathfrak I_{\omega_A}$.	Thus, by \eqref{eq:lemma_estim_avint_annuli} and \eqref{eq:lemma_estim_avint_annuli2}  we obtain \eqref{eq:sup_aux_2} and conclude the proof of the lemma.
\end{proof}
\vvv

\subsection{The three-step perturbations}\label{subsec:three-steps}
{An important component of our method  is the comparison of $\nabla_1\Gamma_A$ to the gradient of the fundamental solution associated with the averaged matrix. This is what we  prove in the next lemma.}
\begin{lemma}\label{lem:main_pw_estimate}
Let  $A$ be a uniformly elliptic matrix in $\Rn1$, $n\geq 2$, satisfying $A\in\DMO_s\cap \DMO_\ell$.
For $R_0>0$,  there exists $C=C(n, \Lambda, R_0)>0$ such that, for $x,  y \in \Rn1$ such that $0<|x-y|< R<R_0 $ , and
\[
	\mathcal K^1_{\Theta}(x,y) \coloneqq \nabla_1\Gamma_A(x,y)- \nabla_1\Theta\bigl(x,y; \bar A_{x,|x-y|/2}\bigr),
\]
we have
\begin{equation}\label{eq:estimate_fund_sol_average}
	|\mathcal K^1_\Theta(x,y)|\leq C \, \frac{\tau_A(r)}{r^{n}} + C \frac{\widehat \tau_A(R)}{R^n}\qquad \text{ for } r\coloneqq|x-y|/2,
\end{equation}
where  
\begin{equation*}
	\tau_A(r) \coloneqq  \mathfrak I_{\omega_A}(r) + \mathfrak L^n_{\omega_A}(r)=\int_0^r \omega_A(t)\, \frac{dt}{t} +  r^{n}\int_r^\infty \omega_A(t)\, \frac{dt}{t^{n+1}}
\end{equation*}
and
\[
	\widehat \tau_A(R)= \mathfrak I_{\omega_A}(R) + \mathfrak L^{n-1}_{\omega_A}(R) =  \int_0^R\omega_A(t)\, \frac{dt}{t} + R^{n-1}\int_R^\infty \omega_A(t) \, \frac{dt}{t^{n}}.
\]
{ In particular, if $R_0=1$, the constant $C$ only depends on ellipticity, dimension, and the Dini Mean Oscillation condition.}
\end{lemma}

\begin{proof}
	Let $x,y\in B(0,R)$. Let us denote by $\bar a_{ij}$ the coefficients of $\bar A_{x,r}$ and, for brevity, write $\Theta(z,y)\coloneqq \Theta(z,y;\bar A_{x,r}).$
	{ By \eqref{eq:representation-perturbation}, it holds}
	\begin{equation}\label{eq:lem_pointwise_111}
		\begin{split}
			&\nabla_1\Gamma(\cdot ,y)-\nabla_1\Theta(\cdot ,y) =\int \nabla_1\nabla_2\Gamma(\cdot ,z)\bigl(\bar A_{x,r}-  A(z)\bigr)\nabla_1\Theta(z,y)\, dz\\
			&\qquad \qquad  \eqqcolon \int \Phi(\cdot,y,z)\, dz\\
			& \qquad \qquad = \int_{B(x,r)} \Phi(\cdot,y,z)\, dz + \int_{B(y,r)}\Phi(\cdot,y,z)\, dz \\
			&\qquad \qquad \qquad+ \int_{\Rn1\setminus(B(x,r)\cup B(y,r))} \Phi(\cdot ,y,z)\, dz \eqqcolon I_r (\cdot)+ II_r (\cdot )+ III_r(\cdot).
		\end{split}
	\end{equation}	
	First, let us estimate $I_r$. Setting $\varepsilon_{x,r}(z)\coloneqq (\bar A_{x,r}- A (z))\,\chi_{B(x,r)}(z),$
we have
	\begin{equation*}
		\begin{split}
	 		I_r(w)&=\int \nabla_1\nabla_2 \Gamma (w,z)\cdot\varepsilon_{x,r}(z) \nabla_1 \Theta (z,y)\, dz \\
	 		&=	\int \nabla_1\nabla_2 \Gamma (w,z)\cdot\varepsilon_{x,r}(z) \bigl(\nabla_1 \Theta(z, y)- \nabla_1\Theta(x,y)\bigr)\bigr)\, dz\\
	 		&\qquad\quad + \int \nabla_1\nabla_2 \Gamma (w,z)\cdot\varepsilon_{x,r}(z) \nabla_1\Theta(x,y)\, dz\\
	 		&\eqqcolon I_{r,1}(w)+I_{r,2}(w).
		\end{split}
	\end{equation*}
	For $w=x$, in order to estimate $I_{r,1}(x)$, we use Lemma \ref{lem:estim_fund_sol} and write
	\begin{equation*}
		\begin{split}			
		|I_{r,1}(x)|&\lesssim \int|\varepsilon_{x,r}(z)|\frac{\bigl|\nabla_1\Theta(z,y)- \nabla_1\Theta(x,y)|}{|x-z|^{n+1}}\, dz\\
		&\lesssim \int  \frac{|\varepsilon_{x,r}(z)|}{|x-z|^{n} |x-y|^{n+1}}\, dz\approx \frac{1}{r^{n+1}}\int_{B(x,r)}\frac{|A(z)-\bar A_{x,r}|}{|x-z|^{n}}\, dz.
			\end{split}
	\end{equation*}
	Then, we estimate the last integral by splitting the domain of integration into dyadic annuli $A(x,2^{-j-1}r, 2^{-j}r)$:
	\begin{equation}\label{eq:estimIr1}
		\begin{split}
			|I_{r,1}(x)|&\lesssim \frac{1}{r^{n+1}}\sum_{j=0}^\infty \int_{A(x,2^{-j-1}r, 2^{-j}r)} \frac{\bigl|A(z)-\bar A_{x,r}\bigr|}{|x-z|^{n}}\, dz\\
			&\lesssim \frac{1}{r^{n+1}}\sum_{j=0}^\infty \frac{r}{2^{j}}\avint_{B(x,2^{-j}r)}\bigl|A(z)-\bar A_{x,r}\bigr|\, dz \\
			&\leq \frac{1}{r^{n}}\sum_{j=0}^\infty \frac{1}{2^j} \avint_{B(x,2^{-j} r)}\Bigl(\bigl|A(z)-\bar A_{x,2^{-j}r}\bigr| + |\bar A_{x,2^{-j}r} - \bar A_{x,r}\bigr|\Bigr)\, dz \\
			&\lesssim \frac{1}{r^{n}}\sum_{j=0}^\infty \frac{1}{2^j} \Biggl(\omega_A(2^{-j} r) + \sum_{k=0}^{j-1}\avint_{B(x,2^{-k}r)}\bigl|A(z)-\bar A_{x,2^{-k}r}\bigr|\, dz\Biggr)\\
			&\lesssim \frac{1}{r^{n}}\sum_{j=0}^\infty \frac{1}{2^j} \Biggl(\omega_A(2^{-j} r) + \sum_{k=0}^{j-1}\omega_A(2^{-k}r)\Biggr)
			\\
			&\lesssim \frac{1}{r^{n}}\sum_{j=0}^\infty \frac{j+1}{2^j}\omega_A(2^{-j} r)\lesssim \frac{1}{r^n}\int_0^r \omega_A(t) \frac{dt}{t},
		\end{split}
	\end{equation}
	{ where we used that $j+1\leq 2^j$ for any $j\geq 0$ integer.}
	
	The estimate of $I_{r,2}(x)$ is slightly more delicate.
	Let us denote $g\coloneqq \varepsilon_{x,r} (\cdot)\nabla_1 \Theta(x,y)$
	and observe that $L^{-1}_{A}\div g$ is a weak solution to
	\(
		L_A \bigl(L^{-1}_{A}\div g\bigr)= \div g
	\)
	in $B(x,r/3)$.
	We claim that, {for $N$ as in Remark \ref{rem:pwbound1} and  $\rho\coloneqq r/(3N+3)$,}
	\begin{equation}\label{eq:claim_lem_pw_123}
		\mathring\omega^{x,N\rho}_g(t)\lesssim \frac{\omega_A(t \rho)}{r^{n}}, \qquad 0<t<1,
	\end{equation}
	with the implicit constant depending on $n$ and the doubling parameter $C_1$.
	In order to prove \eqref{eq:claim_lem_pw_123} we first observe that, for $0<t<1$ and $\varepsilon(\cdot)\coloneqq\bar A_{x,r}-A(\cdot)$, we have
	\begin{equation*}
		\begin{split}
			&\mathring\omega^{x,N\rho}_g(t)=\sup_{w\in B(x,N\rho)}\avint_{B(w,t\rho)} \Bigl|g(z)-\avint_{B(w,t\rho)}g(u)\, du\Bigr|\, dz\\
			&=\sup_{w\in B(x,N\rho)}    \avint_{B(w,t\rho)} \left| \left(\chi_{B(x,r)}(z)\varepsilon(z) - \avint_{B(w,t\rho)}\varepsilon(u)\chi_{B(x,r)}(u) \, du \right)\cdot \nabla_1\Theta(x,y)\right|\, dz\\
			&\overset{\eqref{eq:estim_deriv_fund_sol}}{\lesssim} r^{-n}\sup_{w\in B(x,N\rho)} \avint_{B(w,t\rho)} \Bigl|\chi_{B(x,r)}(z)\varepsilon(z) -  \avint_{B(w,t\rho)}\varepsilon(u)\chi_{B(x,r)}(u)\, du\Bigr|\, dz\\
			&\eqqcolon r^{-n}\sup_{w\in B(x,N\rho)}\mathcal I(w,t\rho).
		\end{split}
	\end{equation*}
	For $w\in B(x,N\rho)$ and $t <1$ by triangle inequality we have $ B(w,t \rho)\subset B(x,r)$.
	Thus, for such $w$ and $t$ we can write
	\[
		\mathcal I(w,t\rho)= \avint_{B(w,t\rho)} \Bigl|\varepsilon(z) -  \avint_{B(w,t\rho)}\varepsilon\Bigl|\, dz=\avint_{B(w,t\rho)} \bigl|A(z) -  \bar A_{w,t\rho}\bigr|\, dz\leq \omega_A(t\rho),
	\]
	which implies \eqref{eq:claim_lem_pw_123}. Hence, by Theorem \ref{theorem:dong_kim} and {\eqref{eq:Linftyest-r} (see the end of Remark \ref{rem:pwbound1})},{
	\begin{equation}\label{eq_estim_I2r}
		\begin{split}
			|I_{r,2}(x)|&\leq \sup_{w\in B(x,2\rho)}\Bigl|\int\nabla_1\nabla_2 \Gamma(w,z)\varepsilon_{x,r}(z)\nabla_1 \Theta(x,y)\, dz\Bigr|\\
			&\lesssim \biggl(\frac{1}{\rho^{n+1}}\int_{B(x,4\rho)}\bigl|\nabla L^{-1}_{A}\div g(w)\bigr|^2 \, dw\biggr)^{1/2} + \int_0^{1}\mathring\omega^{x,N\rho}_g(t)\, \frac{dt}{t}\\
			&\overset{\eqref{eq:claim_lem_pw_123}}{\lesssim}_N \biggl(\frac{1}{\rho^{n+1}}\int_{B(x,4\rho)}\bigl|\nabla L^{-1}_{A}\div g(w)\bigr|^2 \, dw\biggr)^{1/2} + \frac{1}{r^{n}}\int_0^{r}\omega_A(t)\, \frac{dt}{t}.
		\end{split}
	\end{equation}}
	
The operator $\nabla L^{-1}_{A}\div$ is bounded from  $L^2(\mathcal L^{n+1})$ to $L^2(\mathcal L^{n+1})$ and satisfies
	\begin{equation}\label{eq:lempw21}
		\begin{split}
			\bigl\|\nabla L^{-1}_{A}\div g\bigr\|_{L^2(\mathcal L^{n+1})}&\overset{\eqref{eq:L2_boundedness_nabla_L_inv}}{\lesssim}\|g\|_{L^2(\mathcal L^{n+1})}.
		\end{split}	
	\end{equation}
By Lemma \ref{lem:estim_fund_sol} and the fact that $|x-y|=2 r$ we have that
	\begin{align}
 			\|g\|_{L^2(\mathcal L^{n+1})}&\leq\biggl(\int_{B(x,r)} |\varepsilon(z)|^2|\nabla_1 \Theta(x,y)|^2\, dz\biggr)^{1/2} \notag\\
 			&= |\nabla_1 \Theta(x,y)|\biggl(\int_{B(x,r)}\bigl|A(z)-\bar A_{x,r}\bigr|^2\, dz\biggr)^{1/2} \overset{\eqref{eq:sup_aux_1}}{\lesssim}r^{-n} r^{\frac{n+1}{2}}\, \mathfrak I_{\omega_A}(r).\label{eq:estim:F_x_L2}
 	\end{align}
 Hence,  combining \eqref{eq:estimIr1}, \eqref{eq_estim_I2r}, \eqref{eq:lempw21}, and \eqref{eq:estim:F_x_L2}, we obtain
 	\[
 		|I_{r}(x)|\lesssim r^{-n} \mathfrak I_{ \omega_A}(r).
 	\]

	\vv
	
	Let us bound $II_r(x)$. For $z\in B(y,r)$, we have $|x-z|\geq|x-y|-|z-y|>r.$ Hence, by Lemma \ref{lem:estim_fund_sol}, triangle inequality and analogous calculations to those that proved \eqref{eq:estimIr1}, we have that
	\begin{equation*}
		\begin{split}
			|II_r(x)|&\lesssim \int_{B(y,r)}\frac{|A(z)-\bar A_{x,r}|}{|x-z|^{n+1}|y-z|^{n}}\, dz\leq \frac{1}{r^{n+1}}\int_{B(y,r)}\frac{|A(z)-\bar A_{x,r}|}{|y-z|^{n}}\, dz\\
			&\leq {\frac{1}{r^{n+1}}\int_{B(y,r)}\frac{|\bar A_{x,r}-\bar A_{y,4r}|}{|y-z|^{n}} \, dz}+ \frac{1}{r^{n+1}}\int_{B(y,4r)}\frac{|A(z)-\bar A_{y,4r}|}{|y-z|^{n}}\, dz\\
				&\overset{\eqref{eq:estimIr1}}{\lesssim} \frac{1}{r^{n+1}}\Biggl(\int_{B(y,r)}\frac{1}{|z-y|^{n}}\,\avint_{B(x,r)}|A(w)-\bar A_{y,4r}|\,  dw\, dz \Biggr) + \frac{1}{r^{n}}\int_0^{4r}\omega_A(t)\, \frac{dt}{t},\\
				&\lesssim  \frac{1}{r^{n}}\int_0^{4r}\omega_A(t)\, \frac{dt}{t} + \frac{\omega_A(4r)}{r^{n}}\lesssim  \frac{1}{r^{n}}\int_0^{r}\omega_A(t)\, \frac{dt}{t} + \frac{\omega_A(r)}{r^{n}} \overset{\eqref{eq:omega<dini}}{\lesssim} r^{-n}\, \mathfrak I_{\omega_A}(r),
		\end{split}
	\end{equation*}
	where in the penultimate  inequality we used the doubling property of $\omega_A$ and $ \mathfrak I_{\omega_A}$.

	\vv
	{
	We are left with the estimate of $III_r$\footnote{This is exactly the part we mentioned in the introduction that was missing  in the justification of \cite[Lemma 2.2]{KS11} when the coefficients are not periodic (even in the H\"older continuous case).}. Let us observe that $B(x,r)\cup B(y,r)\subset B(x,4r)$. Given $j\geq 0$, we denote $\mathcal C_j\coloneqq A(x,2^j r, 2^{j+1}r)$ and we split
	\begin{equation}\label{eq:3rlem}
		\begin{split}
			III_r (x)=\int_{B(x,8r)\setminus \bigl(B(x,r)\cup B(y,r)\bigr)} \Phi(x,y,z)\, dz + \sum_{j\geq 3}\int_{\mathcal C_j}\Phi(x,y,z)\, dz
				\eqqcolon \mathcal I^0_r + \mathcal I^1_r,
		\end{split}
	\end{equation}
	where $\Phi$ is the function defined in \eqref{eq:lem_pointwise_111}.
	The term $\mathcal I^0_r$ can be readily estimated using Lemma \ref{lem:estim_fund_sol} and the fact that, for $z\in B(x,4r)\setminus \bigl(B(x,r)\cup B(y,r)\bigr)$, $|x-z|>r$ and $|y-z|>r$. In particular,
	\begin{equation}\label{eq:estim_I_0_final}
		\begin{split}
			|\mathcal I^0_r|&\lesssim \int_{B(x,8r)\setminus \bigl(B(x,r)\cup B(y,r)\bigr)}\frac{\bigl|A(z)- \bar A_{x,r}\bigr|}{|x-z|^{n+1}|y-z|^{n}}\, dz\\
		  &\leq r^{-2n-1}\int_{B(x,8r)}\bigl|A(z)- \bar A_{x,r}\bigr|\, dz\\
			&\lesssim \frac{1}{r^{n}}\avint_{B(x,8r)}\bigl|A(z)- \bar A_{x,8r}\bigr|\, dz + \frac{1}{r^{n}}\bigl|\bar A_{x,8r}- \bar A_{x,r}\bigr|\, dz\\
			&\lesssim \frac{\omega_A(8r)}{r^{n}}\lesssim \frac{\omega_A(r)}{r^{n}}\overset{\eqref{eq:omega<dini}}{\lesssim} r^{-n}\, \mathfrak I_{\omega_A}(r).
		\end{split}
	\end{equation}
	
	{
		For $w\in \Rn1$ we denote
		\[
			v_j(w)\coloneqq \int_{\mathcal C_j} \nabla_2\Gamma(w,z)\bigl(\bar A_{x,r}-A(z)\bigr)\nabla_1 \Theta(z,y)\, dz
		\]
		so that
		\[
			\mathcal I^1_r=\sum_{j\geq 3}\nabla v_j(x)=\sum_{j=3}^{j_0}\nabla v_j + \sum_{j\geq j_0+1}\nabla v_j\eqqcolon \mathcal I^{1,1}_r + \mathcal I^{1,2}_r,
		\]
		where $j_0$ is such that $2^{j_0-3}r\leq R< 2^{j_0-2}r$. Remark that $v_j$ is a weak solution to $L_A v_j=0$ in $B(x,2^jr)$ for all $j\geq 3.$
		
		Let us estimate $\mathcal I^{1,1}_r$. For $j\in \{3,\ldots,j_0\}$\footnote{The method for $j\in \{3,\ldots,j_0 \}$ could be significantly  simplified using  the pointwise estimates for $\nabla_1 \nabla_2 \Gamma_A$, but as we will repeat this argument to handle the case $j >j_0$, we decided to use it  for both. Here we only use the $\DMO_s$ condition to bound $|\nabla v_j(x)|$ by its $L^2$-average on the ball, while	{the last inequality of \eqref{eq:nabla_v_j_1}} holds for solutions of elliptic equations with $L^\infty$ coefficients without any additional regularity assumption.},
		 by {\eqref{eq:Linftyest-r} we have}
		\begin{equation}\label{eq:nabla_v_j_1}
			\begin{split}
				|\nabla v_j(x)|&\leq \sup_{w\in B(x,2^{j-4}r)}|\nabla v_j(w)|\lesssim_R \Bigl(\avint_{B(x,2^{j-3}r)} |\nabla v_j|^2\Bigr)^{1/2}\\
				&\lesssim \frac{1}{2^jr}\Bigl(\avint_{B(x,2^{j-2}r)} | v_j|^2\Bigr)^{1/2},
			\end{split}
		\end{equation}
		where the last bound follows from Caccioppoli inequality.
		

		For $z\in \mathcal C_j$, we have $|x-z|\approx|y-z|\approx 2^j r$ and so \eqref{eq:estim_deriv_fund_sol} implies that $|\nabla_1\Theta(z,y)|\lesssim |z-y|^{-n}\approx (2^jr)^{-n}$. If $w\in B(x,2^{j-2}r)$, by {Cauchy-Schwarz} inequality and the latter estimate, we have
		\begin{equation}\label{eq:vj1}
			\begin{split}
				|v_j(w)|&=\Bigl|\int_{\mathcal C_j}\nabla_2 \Gamma (w,z)\bigl(\bar A_{x,r}-A(z)\bigr)\nabla_1\Theta (z,y)\, dz\Bigr|\\
					&\lesssim \frac{1}{(2^j r)^n}\Bigl(\int_{\mathcal C_j}|\nabla_2\Gamma(w,z)|^2\, dz\Bigr)^{1/2}\Bigl(\int_{\mathcal C_j}\bigl|\bar A(z)-A_{x,r}\bigr|^{2}\, dz\Bigr)^{1/2}.
			\end{split}
		\end{equation}

		
		Thus, since $|z-w|<2^{j+2}r$ for $z\in \mathcal C_j$, {by Caccioppoli inequality and Lemma \ref{lem:estim_fund_sol} we have}
		\begin{equation}\label{eq:vj2}
			\begin{split}
				\Bigl(\int_{\mathcal C_j}|\nabla_2\Gamma(w,z)|^2\, dz\Bigr)^{1/2}&\leq\Bigl(\int_{B(w,2^{j+2}r)}|\nabla_2\Gamma(w,z)|^2\, dz\Bigr)^{1/2}\\
				&\lesssim \frac{1}{2^j r}\Bigl(\int_{A(w,2^{j+2}r,2^{j+3}r)}|\Gamma(w,z)|^2\, dz\Bigr)^{1/2}\\
				&\overset{}{\lesssim} \frac{|B(w,2^{j+3}r)|^{1/2}}{(2^j r)^n}\lesssim(2^{j}r)^{\frac{(1-n)}{2}}.
			\end{split}
		\end{equation}
{	Inequality \eqref{eq:sup_aux_2} yields
		\begin{equation}\label{eq:telescoping_an_1}
			\begin{split}
				\Bigl(\int_{\mathcal C_j}\bigl|A(z)-\bar A_{x,r}\bigr|^{2}\, dz\Bigr)^{1/2}&=|\mathcal C_j|^{1/2}\Bigl(\avint_{\mathcal C_j}\bigl|A(z)-\bar A_{x,r}\bigr|^{2}\, dz\Bigr)^{1/2}\\
				& \\
				&\lesssim (2^{j+1}r)^{\frac{(n+1)}{2}}\mathfrak I_{\omega_A}(2^jr).
			\end{split}
		\end{equation}
		}
	{ In view of \eqref{eq:vj1}, \eqref{eq:vj2}, and \eqref{eq:telescoping_an_1} we obtain
		\begin{equation}\label{eq:vj}
			|v_j(x)|\lesssim_n(2^j r)^{1-n}\mathfrak I_{\omega_A}(2^jr),
		\end{equation}
		which, by \eqref{eq:nabla_v_j_1}, implies
		\begin{equation}
			|\nabla v_j(x)|\lesssim_{R, n}(2^j r)^{-n}\mathfrak I_{\omega_A}(2^jr),
		\end{equation}
		and hence
		\begin{equation}\label{eq:estim_I_1_final}
			\begin{split}
				|\mathcal I^{1,1}_r|&\leq \sum_{j=3}^{j_0} |\nabla v_j(x)|\lesssim_{R, n} \sum_{j=3}^{j_0} (2^j r)^{-n}\mathfrak I_{\omega_A}(2^jr)\leq \sum_{j=1}^{\infty} (2^j r)^{-n}\mathfrak I_{\omega_A}(2^jr)\\
				&\overset{\eqref{eq:mod_cont_sum_2}}{\lesssim} \int_r^{\infty} \mathfrak I_{\omega_A}(t)\, \frac{dt}{t^{n+1}}\overset{\eqref{eq:L_I_theta}}{=}\frac{1}{nr^n}\int_0^r \omega_A(t)\, \frac{dt}{t} + \frac{1}{n}\int_r^\infty \omega_A(t)\, \frac{dt}{t^{n+1}}\\
				&= \frac{1}{n r^n} \, \left( \mathfrak I_{\omega_A}(r)+\mathfrak L^{n}_{\omega_A}(r) \right).
			\end{split}
		\end{equation}
	}
	
		We are left with the estimate of $\mathcal I^{1,2}_r$. For  $j\geq j_0 + 1$, the rescaled version of \eqref{eq:DK217} and Cacciopoli inequality give
		\begin{equation}\label{eq:estim_nabla_v_j_second_case}
			\begin{split}
				|\nabla v_j(x)|&\lesssim_R \Bigl(\avint_{B(x,R)}|\nabla v_j|^2\Bigr)^{1/2}\lesssim \frac{1}{R}\Bigl(\avint_{B(x,2R)}|v_j(w)|^2\, dw\Bigr)^{1/2}.
			\end{split}
		\end{equation}
		If $z\in \mathcal C_j$ and $w\in B(x,2R)$, since $j\geq j_0+1$ and $R \leq 2^{j_0-2}r$, it holds that $|z-w|\approx 2^jr$, and, 
		  arguing as above, we can prove \eqref{eq:vj}. Therefore,
		\begin{equation*}
			\begin{split}
				|\nabla v_j(x)|\overset{\eqref{eq:vj}}{\lesssim} \frac{(2^j r)^{1-n}}{R}\mathfrak I_{\omega_A}(2^jr),
			\end{split}
		\end{equation*}
		which infers that
		{
		\begin{equation}\label{eq:estim_mathcal_J_0}
			\begin{split}
				\sum_{j\geq j_0+1}|\nabla v_j(x)|&\lesssim \sum_{j\geq j_0+1}\frac{(2^j r)^{1-n}}{R} \mathfrak I_{\omega_A}(2^jr)\overset{\eqref{eq:mod_cont_sum_2}}{\lesssim} \frac{1}{R}\int_R^{\infty}\mathfrak I_{\omega_A}(t)\, \frac{dt}{t^n}\\
				&\overset{\eqref{eq:L_I_theta}}{\approx_n}\frac{1}{R^n}\int_0^R\omega_A(t)\, \frac{dt}{t} + \frac{1}{R}\int^\infty_R \omega_A(t)\, \frac{dt}{t^n}\\
				&\approx_n R^{-n}\bigl(\mathfrak I_{\omega_A}(R) + \mathfrak L^{n-1}_{\omega_A}(R)\bigr).
			\end{split}
		\end{equation}
	}
}
	\vv
Therefore, combining \eqref{eq:estim_I_0_final}, \eqref{eq:estim_I_1_final}, and \eqref{eq:estim_mathcal_J_0}, we obtain 
	\begin{equation*}
		\begin{split} 
			|III_r(x)|\leq |\mathcal I^0_r| +  |\mathcal I^1_r| &
			\lesssim r^{-n}\, \left( \mathfrak I_{\omega_A}(r)+\mathfrak L^{n}_{\omega_A}(r) \right)+ R^{-n} \left(\mathfrak I_{\omega_A}(R) + \mathfrak L^{n-1}_{\omega_A}(R) \right).
		\end{split}
	\end{equation*}
Gathering the bounds for  $I_r(x)$, $II_r(x)$, and $III_r(x)$, we conclude \eqref{eq:estimate_fund_sol_average}.
}
\end{proof}

\vv

\begin{lemma}\label{lem:lem_1_prep_main_lemma}
	Let $A$ be a uniformly elliptic matrix in $\Rn1$, $n\geq 2$, satisfying $A\in \DMO_s$. Let $0<\delta< r<1$ and $x \in \Rn1$, and assume that  $\Omega_{x,\delta} \subset \Rn1$ is a Borel set such that for some constant $M\geq 1$,
	$$
	B(x,\delta) \subset \Omega_{x,\delta} \subset B(x,M \delta).
	$$ If we denote  $\bar A_{\Omega_{x,\delta}}\coloneqq\avint_{\Omega_{x,\delta}} A$, then
	\begin{equation}\label{eq:estim_diff_averages}
		\bigl|\bar A_{x,r/2}-\bar A_{\Omega_{x,\delta}}\bigr|\lesssim_{M}  \int_{\delta}^{r}\omega_A(t) \,\frac{dt}{t}\leq \tau_A(r)
	\end{equation}
	{
	and 
	\begin{equation}\label{eq:estim_diff_averages_2}
		\bigl|(\bar A_{x,r/2})_s-(\bar A_{\Omega_{x,\delta}})_s\bigr|\lesssim_{M}  \int_{\delta}^{r}\omega_A(t) \,\frac{dt}{t}\leq \tau_A(r).
	\end{equation}
	}
	Moreover, for 
	\begin{align*}
		\mathcal K^2_{\Theta}(x,y) &\coloneqq \nabla_1 \Theta (x-y,0;\bar A_{x,r/2})-\nabla_1 \Theta(x-y,0;\bar A_{x,\delta/2} )
	\end{align*}
	and all $z\in \mathbb R^{n+1}\setminus\{0\}$, it holds
	\begin{equation}\label{eq:estim_diff_kernels}
		\bigl|\mathcal K^2_{\Theta}(z,0)\bigr|\lesssim_{n,\Lambda, M} \frac{1}{|z|^{n}}\int_{\delta}^{r} \omega_A(t) \,\frac{dt}{t}.
	\end{equation}
\end{lemma}

\begin{proof}
	Let $N_0 \geq 1$ be such that $2^{-N_0} r \leq 	\delta < 2^{-N_0+1}r$ and  $2^{N_1-1} \leq M < 2^{N_1}$. Therefore, if $N=N_0-N_1$, we have that 
	\begin{equation*}
		\begin{split}
			\bigl|\bar A_{x,r/2} &-\bar A_{\Omega_{x,\delta}}\bigr|\leq \bigl|\bar A_{x,r}-\bar A_{x,2^{-N+1} r}\bigr| + \bigl|\bar A_{\Omega_{x,\delta}}-\bar A_{x,2^{-N+1} r} \bigr|\\
			&\lesssim_M  \sum_{j=1}^{N-2} \bigl|\bar A_{x,2^{-j}r}-\bar A_{x,2^{-j-1}r}\bigr| +\omega_A(2^{-N+1}r) \lesssim \sum_{j=1}^{N-2} \omega_A(2^{-j}r) + \omega_A(2^{-N+1}r)\\
			&= \sum_{j=1}^{N-1} \omega_A(2^{-j}r)\lesssim \int_{\delta}^{r} \omega_A(t)\, \frac{dt}{t}.
		\end{split}
	\end{equation*}
	{The bound \eqref{eq:estim_diff_averages_2} follows directly from \eqref{eq:estim_diff_averages}.}
	
	Since $A$ is uniformly elliptic with constant $\Lambda$, it is invertible and its inverse is uniformly elliptic as well with the same constant.	Moreover,  as all its eigenvalues are bounded from above by $\Lambda$ and below by $\Lambda^{-1}$, and so is its determinant $\det(A)$ (as the product of its $n+1$ eigenvalues).  The same considerations apply to $\mathcal A_r\coloneqq (\bar A_{x,r/2})_s$ and $\mathcal A_\delta\coloneqq (\bar A_{\Omega_{x,\delta}})_s$. By standard calculations we can write
	\begin{equation}\label{eq:first_estim_diff_kernels}
		\begin{split}
			&{\bigl|\nabla_1 \Theta(z,0;\bar A_{x,r/2})-\nabla_1 \Theta(z,0;{\bar A_{\Omega_{x,\delta}}})\bigr|\overset{\eqref{eq:fund_sol_const_matrix}}{=} }\bigl|\nabla_1 \Theta(z,0;\mathcal A_r)-\nabla_1 \Theta(z,0;\mathcal A_\delta)\bigr|\\
			&=\frac{1}{\omega_n}\Bigl|\frac{\mathcal A_r^{-1}z}{\sqrt{\det \mathcal A_r}\langle\mathcal A_r^{-1}z,z\rangle^{(n+1)/2}}-\frac{\mathcal A_\delta^{-1}z}{\sqrt{\det \mathcal A_\delta}\langle\mathcal A_\delta^{-1}z,z\rangle^{(n+1)/2}}\Bigr|\\
			&\lesssim_{\Lambda,n} \frac{1}{|z|^{2n+2}}\Bigl|\sqrt{\det \mathcal A_\delta}\langle\mathcal A_\delta^{-1}z,z\rangle^{(n+1)/2}\mathcal A_r^{-1}z - \sqrt{\det \mathcal A_r}\langle\mathcal A_r^{-1}z,z\rangle^{(n+1)/2}\mathcal A_\delta^{-1}z\Bigr|.
		\end{split}
	\end{equation}
	We remark that by elementary calculations and using the ellipticity of $\mathcal A_\delta$ and $\mathcal A_r$ we have
	\begin{equation}\label{eq:diff_inv_rQ}
		\begin{split}
			|\mathcal A^{-1}_\delta z-\mathcal A^{-1}_r z|& =|\mathcal A_\delta^{-1}\mathcal A_r\mathcal A_r^{-1}z-\mathcal A_\delta^{-1}\mathcal A_\delta\mathcal A_r^{-1}z|\\
			&=\bigl|\mathcal A_\delta^{-1}(\mathcal A_r-\mathcal A_\delta)\mathcal A_r^{-1}z\bigr|\lesssim_\Lambda |\mathcal A_r-\mathcal A_\delta||z|.
		\end{split}
	\end{equation}
	The mean value theorem implies that, for $0<a<b$  we have
	\begin{equation*}
		\begin{split}
			|a^{(n+1)/2}-b^{(n+1)/2}|&\leq \frac{n+1}{2}\max_{t\in [a,b]} t^{(n-1)/2} (b-a)= \frac{n+1}{2} b^{(n-1)/2} (b-a).
		\end{split}
	\end{equation*}
	The symmetric inequality holds for $0<b\leq a$ and, for the choices $a= \langle\mathcal A_r^{-1}z,z\rangle^{(n+1)/2}$ and $b= \langle\mathcal A_\delta^{-1}z,z\rangle^{(n+1)/2}$, gives
	\begin{equation}\label{eq:mvt_bound_diff_lem1}
		\begin{split}
			&\bigl| \langle\mathcal A_r^{-1}z,z\rangle^{(n+1)/2} - \langle\mathcal A_\delta^{-1}z,z\rangle^{(n+1)/2}\bigr|\\
			&\qquad\lesssim_n \bigl|\langle\mathcal A_r^{-1}z,z\rangle- \langle\mathcal A_\delta^{-1}z,z\rangle\bigr|\Bigl|\langle\mathcal A_r^{-1}z,z\rangle^{(n-1)/2} + \langle\mathcal A_\delta^{-1}z,z\rangle^{(n-1)/2}\Bigr|	\\
			&\qquad\overset{\eqref{eq:diff_inv_rQ}}{\lesssim}|\mathcal A_r-\mathcal A_\delta|\,|z|^2 \, |z|^{n-1}= |z|^{n+1}|\mathcal A_r-\mathcal A_\delta|,
		\end{split}
	\end{equation}
	by the ellipticity  of $\mathcal A_r^{-1}$ and $\mathcal A_\delta^{-1}$.
	
	Let us recall that the map $\det\colon \mathbb R^{(n+1)\times (n+1)}\to \R$ is a polynomial in the entries of the matrix and, more specifically, that Jacobi's formula gives
	\[
	\frac{\partial}{\partial \tilde a_{ij}}\det(\tilde{\mathcal A})=(\adj \tilde{\mathcal A})_{ij} \qquad\text{ for }\quad  \tilde{\mathcal A}=(\tilde a_{ij})_{i,j}\in\mathbb R^{(n+1)\times (n+1)},
	\]
	where $\adj \tilde{\mathcal A}=\det(\tilde{\mathcal A}) \tilde{\mathcal A}^{-1}$ is the adjugate matrix of $\tilde{\mathcal A}$. 
	In particular, the map $\det(\cdot)$ is locally Lipschitz continuous and, for $|\tilde{\mathcal A}|\leq \Lambda$, its Lipschitz constant depends only on $\Lambda$ and $n$. Moreover, 
{	\[
	|\sqrt a - \sqrt b|=\bigl(a^{1/2} + b^{1/2}\bigr)^{-1}|a-b| \qquad \text{ for }a,b>0,
	\]}
	which implies
	\begin{equation}\label{eq:diff_sq_rt_r_Q}
		\begin{split}
			|\sqrt{\det \mathcal A_r} - \sqrt{\det \mathcal A_\delta}|&= \bigl((\det \mathcal A_r)^{1/2} + (\det \mathcal A_\delta)^{1/2}\bigr)^{-1}|\det \mathcal A_r - \det \mathcal A_\delta|\\
			&\lesssim_{\Lambda, n} |\det \mathcal A_r - \det \mathcal A_\delta|\lesssim_{\Lambda,n} |\mathcal A_r - \mathcal A_\delta|.	
		\end{split}
	\end{equation}
	Finally, \eqref{eq:first_estim_diff_kernels}, triangle inequality, the bounds \eqref{eq:mvt_bound_diff_lem1}, \eqref{eq:diff_sq_rt_r_Q}, \eqref{eq:diff_inv_rQ}, and the uniform ellipticity of $\mathcal A_r$ and $\mathcal A_\delta$ yield
	\begin{equation*}
		\begin{split}
			&\bigl|\nabla_1 \Theta(z,0;\mathcal A_r)-\nabla_1 \Theta(z,0;\mathcal A_\delta)\bigr|\\
			&\qquad {\lesssim}_{\Lambda, n} \frac{1}{|z|^{2n+2}}\Bigl(|\sqrt{\det \mathcal A_\delta} - \sqrt{\det\mathcal A_r}|\bigl|\langle\mathcal A_\delta^{-1}z,z \rangle^{(n+1)/2}\mathcal A_r^{-1}z\bigr|\\
			&\qquad\qquad + |\sqrt{\det \mathcal A_r}|\bigl|\langle \mathcal A_\delta^{-1}z,z\rangle^{(n+1)/2}(\mathcal A_r^{-1}z-\mathcal A_\delta^{-1}z)\bigr|\\
			&\qquad\qquad + |\sqrt{\det\mathcal A_r}|\bigl|\langle \mathcal A_\delta^{-1}z,z\rangle^{(n+1)/2} - \langle \mathcal A_r^{-1}z,z\rangle^{(n+1)/2}\bigr| |\mathcal A_\delta^{-1}z|\Bigr)\\
			&\qquad\lesssim_{n,\Lambda} \frac{|\mathcal A_r-\mathcal A_\delta|}{|z|^{n}}\overset{\eqref{eq:estim_diff_averages_2}}{\lesssim} \frac{1}{|z|^n}\int_{\delta}^{r} \omega_A(t) \,\frac{dt}{t}.
		\end{split}
	\end{equation*}
	This concludes the proof of the lemma.
\end{proof}

We can also demonstrate the following lemma.
\begin{lemma}\label{lem:lem_2_prep_main_lemma}
	Let $A$ be a uniformly elliptic matrix in $\Rn1$, $n\geq 2$, satisfying $A\in \DMO_s$. Assume that  $Q \subset \Rn1$ is a dyadic cube and $\Omega_Q \subset \Rn1$ is a Borel set such that for some constant $M\geq 1$,
	$$
	B(x_Q,\ell(Q)) \subset \Omega_Q \subset B(x_Q, M\ell(Q)).
	$$ If we denote  $\bar A_{\Omega_Q}\coloneqq\avint_{\Omega_Q} A$, then if $x \in Q$,
	\begin{equation}\label{eq:estim_diff_averages2}
		\bigl|\bar A_{\Omega_{Q}} - \bar A_{x, \delta/2}\bigr|\lesssim_{M}  \int_{0}^{\ell(Q)} \omega_A(t) \,\frac{dt}{t}\leq \tau_A(\ell(Q)), \quad \text{for}\,\, \delta< \sqrt{n+1} \ell(Q),
	\end{equation}
	and {for all } $z\in \mathbb R^{n+1}\setminus\{0\}$ {and $\delta< \sqrt{n+1} \ell(Q)$},
	\begin{equation}\label{eq:estim_diff_kernels2}
		\bigl|\nabla_1 \Theta(z,0;\bar A_{\Omega_Q})-\nabla_1 \Theta(z,0;\bar A_{x,\delta/2})\bigr|\lesssim_{n,\Lambda, M} \frac{1}{|z|^{n}} \int_{0}^{\ell(Q)}\omega_A(t) \,\frac{dt}{t}.
	\end{equation}
\end{lemma}

\begin{proof}
The proof is a routine adaptation of the one of Lemma \ref{lem:lem_1_prep_main_lemma} and is left as an exercise to the interested reader.
\end{proof}

\vv

Let $\delta>0$ and, for an affine function $L$, define the truncated integral operator
\[
	\widetilde T^{L,j}_{\mu, \delta}f(x)\coloneqq \int_{|L x- Ly|>\delta} \mathcal K^j_{\Theta}(x,y)f(y)\,d\mu(y), \quad j=1,2.
\]

\begin{lemma}\label{lem:truncatedL2}
Let $A$ be a uniformly elliptic matrix in $\Rn1$, $n\geq 2$, satisfying  $A\in \widetilde \DMO$.  Let $R_0>0$ and $\mu \in M^n_+(\Rn1)$ with compact support so that $\diam (\supp \mu) =R\leq R_0$. Then $\widetilde T^{L,j}_{\mu,\delta}\colon L^2(\mu) \to L^2(\mu)$,  $j=1,2$, satisfying
\begin{align*}
	\sup_{\delta>0}\| \widetilde T^{L,1}_{\mu, \delta} f \|_{{L^2(\mu)}\to{L^2(\mu)} }& \lesssim \mathfrak I_{\tau_A}(R) +\widehat \tau_A(R)\\
\sup_{\delta>0}\| \widetilde T^{L,2}_{\mu, \delta} f \|_{{{L^2(\mu)}\to{L^2(\mu)} }} 	& \lesssim \mathfrak I_{\tau_A}(R),
\end{align*}
where the implicit constants depend on $\|L\|_{\textup{op}}$, $\Lambda$, $n$, and $R_0$. {In particular, if $R_0=1$, the implicit constants only depend on ellipticity, dimension, and the Dini Mean Oscillation condition.}
\end{lemma}

\begin{proof}
In view of Lemmas \ref{lem:main_pw_estimate} and  \ref{lem:lem_1_prep_main_lemma}, we can apply Lemma \ref{lemma:lemma_int_op_bd} to the integral operators with kernel $\mathcal K^j_{\Theta}(x,y) \chi_{B(L x,\delta/2)^c}(Ly)$  and deduce the result.
\end{proof}
\vv

{
	
	\begin{lemma}\label{lem:estimate_norm_K3}
		Let $A$ be a uniformly elliptic matrix in $\Rn1$, $n\geq 2$, satisfying  $A\in \widetilde \DMO$.
		Let $Q$ be a cube in $\Rn1$ with center $x_Q$ and side-length $\ell(Q) \lesssim 1$,
		and let $\nu\in M^n_+(\Rn1)$ be supported on $Q$ and have $n$-growth constant $c_0>0$.
		Assume also that there exist a Borel set $\Omega_Q$ and a constant $M\geq 1$ such that \[B(x_Q,\ell(Q))\subset \Omega_Q\subset B(x_Q,M\ell(Q)).\]
			 For $\delta>0$ let us define
			 \begin{equation}\label{eq:definition_K_3}
			 	\mathcal K^3_\Theta(x,z)\coloneqq\nabla_1 \Theta (z,0;\bar A_{x,\delta/2}) - \nabla_1 \Theta\bigl(z,0;\bar A_{\Omega_Q}\bigr)
			 \end{equation}
			 and
			 \[
			 	T_{\mathcal K^3_\Theta, \nu,\delta}f(x)\coloneqq \int_{|x-y|>\delta} \mathcal K^3_\Theta(x,x-y)f(y)\, d\nu(y),\qquad f\in L^1_{\loc}(\nu).
			 \]
		Then, there exist a positive constant $C''=C''(n,\Lambda, M,c_0)$ such that we have
		\begin{align}\label{eq:main_lemma_estim2}
			\|T_{\mathcal K^3_\Theta, \nu,\delta}\|_{L^2(\nu)\to L^2(\nu)}&\leq  C'' \mathfrak I_{\omega_A}(\ell(Q))^{1/2}\|\mathcal R_{\nu,\delta}\|_{L^2(\nu)\to L^2(\nu)}.
		\end{align}	
		
		\end{lemma}
	\begin{proof}
		
		For brevity, let us denote $\mathcal K\coloneqq \mathcal K^3_\Theta$.
		Observe that the function $\mathcal K(x,\cdot)$ is homogeneous of degree $-n$ for any $x\in \Rn1$, namely
		\[
		\mathcal K(x,z)=\frac{1}{|z|^{n}}\mathcal K\Bigl(x,\frac{z}{|z|}\Bigr) \quad \text{ for all }z\in \Rn1,
		\]
		and satisfies $\|\mathcal K(x,\cdot)\|_{L^2(\mathbb S^{n})}\lesssim_{n,\Lambda} 1$. Indeed, by \eqref{eq:gradient_fund_sol_const_matrix} and ellipticity of $(\bar A_{x,\delta/2})_s$ we have
		\begin{equation*}
			\begin{split}
				\|\mathcal K(x, \cdot)\|^2_{L^2(\mathbb S^{n})}&\lesssim \int_{\mathbb S^n}\biggl|\frac{\omega_{n}^{-1}}{\sqrt{\det (\bar A_{x,\delta/2})_s}}\frac{(\bar A_{x,\delta/2})_s^{-1}\zeta}{\langle(\bar A_{x,\delta/2})^{-1}_s\zeta,\zeta\rangle^{(n+1)/2}}\biggr|^2\, d\sigma(\zeta)\\
				&\qquad \qquad \qquad +\int_{\mathbb S^n}\Bigl|\omega_n^{-1}\frac{\zeta}{|\zeta|^{n+1}}\Bigr|^2\, d\sigma(\zeta) \lesssim_{n,\Lambda}\int_{\mathbb S^n}|\zeta|^{-2n}\, d\sigma(\zeta) = \omega_{n}.
			\end{split}
		\end{equation*}	
		
		Let $\{\varphi_{j,\ell}\}_{j\geq 1, 1\leq \ell\leq N_j}$ be an orthonormal basis of $L^2(\mathbb S^{n})$ of spherical harmonics of degree $j$. In particular $N_j$ satisfies the asymptotic estimate
		\begin{equation}\label{eq:bound_Nj}
			N_j=O(j^{n-1}) \qquad \text{ for }j\gg 1,
		\end{equation}
		for which we refer, for instance, to \cite[display (2.12)]{AH12}.
		Hence, we decompose $\mathcal K$ into spherical harmonics in the $L^2$-sense and write
		\begin{equation*}
			\begin{split}
				\mathcal K(x,z)&=\frac{1}{|z|^{n}}\mathcal K\Bigl(x,\frac{z}{|z|}\Bigr)=\frac{1}{|z|^{n}}\sum_{j\geq 1}\sum_{\ell=1}^{N_j}\bigl\langle\mathcal K(x,\cdot), \varphi_{j,\ell}\bigr\rangle_{L^2(\mathbb S^{n})}\varphi_{j,\ell}\Bigl(\frac{z}{|z|}\Bigr)\\
				&\eqqcolon \frac{1}{|z|^{n}}\sum_{j\geq 1}\sum_{\ell=1}^{N_j} k_{j,\ell}(x)\varphi_{j,\ell}\Bigl(\frac{z}{|z|}\Bigr).
			\end{split}
		\end{equation*}
		We observe that, as $\mathcal K(x,\cdot)$ is an odd function, $k_{j,\ell}(x)=0$ if $j$ is even.
		Furthermore, since $\mathcal K(x,\cdot)$ is smooth on $\mathbb S^{n}$, then by \cite[Thorem 2.36]{AH12}, the series $\sum_{j\geq 1}\sum_{\ell=1}^{N_j} k_{j,\ell}(x)\varphi_{j,\ell}(\cdot)$ converges uniformly on $\mathbb S^n$. 
		
		We claim that
		\begin{equation}\label{eq:bound_spherical_coefficients_I}
			|k_{j,\ell}(x)|\lesssim_{n,\Lambda} \mathfrak I_{\omega_A}(\ell(Q)) \qquad\text{ for all }x\in \Rn1.
		\end{equation}
		Indeed, the bound \eqref{eq:estim_diff_kernels2} readily implies
		\begin{equation*}
			\begin{split}
				|k_{j,\ell}(x)|&\leq \int_{\mathbb S^{n}}|\mathcal K(x,\zeta)||\varphi_{j,\ell}(\zeta)|\, d\sigma(\zeta)
				\lesssim_{n,\Lambda} \int_{0}^{\ell(Q)}\omega_A(t) \,\frac{dt}{t} \int_{\mathbb S^{n}} \frac{|\varphi_{j,\ell}(\zeta)|}{|\zeta|^{n}}\,d\sigma(\zeta)\\
				&\lesssim_n  \int_{0}^{\ell(Q)} \omega_A(t) \,\frac{dt}{t},
			\end{split}
		\end{equation*}
		where the last inequality holds because of the normalization $\|\varphi_{j,\ell}\|_{L^2(\mathbb S^{n})}=1$ and Cauchy-Schwarz inequality.
		On the other hand we observe that, for $m\geq 1$, the bound \eqref{eq:estim_deriv_fund_sol} and the definition of $\mathcal K$ yield
		\[
		\bigl|\nabla^{2m}_z\mathcal K\bigl(x,z\bigr)\bigr|\lesssim_{n,\Lambda, m} \frac{1}{|z|^{n+2m}}, \qquad \text{ for } z\in \Rn1\setminus \{0\}.
		\]
		Thus, denoting by $\Delta_{\mathbb S^{n}}^m$ the $m$-th iteration of the Laplace-Beltrami operator on $\mathbb S^{n}$, we have
		\begin{equation}\label{eq:bound_laplace_beltrami}
			\bigl|\Delta^m_{\mathbb S^{n},\zeta}\mathcal K\bigl(x,\zeta\bigr)\bigr|\lesssim_{n,\Lambda,m}1,
		\end{equation}
		where the subscript in $\Delta^m_{\mathbb S^{n},\zeta}$ denotes that the operator is applied with respect to the $\zeta$-variable.
		Moreover, as $\mathcal K(x,\cdot)$ is infinitely differentiable on $\mathbb S^{n}$, \cite[III 3.1.5]{St70} gives
		\[
		\int_{\mathbb S^n} \Delta^m_{\mathbb S^{n}}\mathcal K(x,\zeta) \,\varphi_{j,\ell}(\zeta)\, d\sigma(\zeta)= k_{j,\ell}(x)[-j(j+n-1)]^m	\qquad \text{ for all }m\geq 1.
		\]
		Hence, Cauchy-Schwarz inequality implies that
		\begin{equation}
			\begin{split}
				[j(j+n-1)]^m\bigl|k_{j,\ell}(x)\bigr|&\lesssim \Bigl(\int_{\mathbb S^{n}}|\Delta^m_{\mathbb S^{n}}\mathcal K(x,\cdot)|^2\, d\sigma \Bigr)^{\frac{1}{2}}\Bigl(\int_{\mathbb S^{n}}|\varphi_{j,\ell}|^2\, d\sigma \Bigr)^{\frac{1}{2}}\\
				&= \Bigl(\int_{\mathbb S^{n}}|\Delta^m_{\mathbb S^{n}}\mathcal K(x,\cdot)|^2\, d\sigma \Bigr)^{\frac{1}{2}}\overset{\eqref{eq:bound_laplace_beltrami}}{\lesssim}_{n,\Lambda,m} 1,
			\end{split}
		\end{equation}
		where we remark that the bound is uniform on $x$ and $y$.
		Hence, for every $m>0$ we have
		\begin{equation*}
			\bigl|k_{j,\ell}(x)\bigr|\lesssim_{n,\Lambda, m} j^{-2m}.
		\end{equation*}
		In particular, the choice {$m= (n+5)(n-1)/2$} yields
		\begin{equation}\label{eq:bound_spherical_coefficients_II}
			\bigl|k_{j,\ell}(x)\bigr|\lesssim_{n,\Lambda} j^{-(n+5)(n-1)}.
		\end{equation}
		Taking the geometric mean of \eqref{eq:bound_spherical_coefficients_I} and \eqref{eq:bound_spherical_coefficients_II}, we  obtain
		\begin{equation}\label{eq:bound_spherical_coefficients_final}
			\bigl| k_{j,\ell}(x)\bigr|\lesssim_{n,\Lambda} \frac{\bigl(\mathfrak I_{\omega_A}(\ell(Q))\bigr)^{1/2}}{j^{(n+5)(n-1)/2}}.
		\end{equation}
		Now, let us define the kernel
		\[
		K_{j,\ell}(z)\coloneqq \frac{1}{|z|^{n }}\varphi_{j,\ell}\Bigl(\frac{z}{|z|}\Bigr), \qquad \text{ for } z\in \Rn1\setminus \{0\}.
		\]
		By \cite[p. 276]{St70} we have
		\[
		\sup_{|x|=1}\Bigl|\frac{\partial^\alpha\varphi_{j,\ell}(x)}{\partial x^\alpha}\Bigr|\lesssim_{\alpha} j^{\frac{n+1}{2}+ |\alpha|}, \quad \text{ for }\alpha\in \mathbb N^{n+1}.
		\]
		so $K_{j,\ell}$ satisfies the estimate
		\begin{equation}                                                        
			|\partial^\alpha_zK_{j,\ell}(z)|\lesssim_{\alpha} \frac{j^{\frac{n+1}{2}+|\alpha|}}{|z|^{n+|\alpha|}}\quad \text{ for } \alpha\in \mathbb N^{n+1}, z\in \Rn1.
		\end{equation}
		Thus, for $j$ odd and by \cite[Corollary 1.4]{To21}, its associated singular integral operator
		\[
		T_{K_{j,\ell},\nu}f(x)\coloneqq \int K_{j,\ell}(x-y)f(y)\, d\nu(y)
		\]
		is bounded on $L^2(\nu)$ with norm
		\begin{equation}
			\|T_{K_{j,\ell},\nu}\|_{L^2(\nu)\to L^2(\nu)}\lesssim j^{\frac{n+5}{2}} \|\mathcal R_\nu\|_{L^2(\nu)\to L^2(\nu)}.
		\end{equation}
		
		Let $T_{K_{j,\ell},\nu,\delta}$ be the associated $\delta$-truncated operator.
		For $\tilde N>1$ and  $\psi \in  L^\infty(\nu)$, we have
		\begin{equation}\label{eq:dct_ml}
			\begin{split}
				&\int \biggl|\int_{|x-y|>\delta}\sum^{\tilde N}_{j\geq 1}\sum_{\ell=1}^{N_j}k_{j,\ell}(x) K_{j,\ell}(x,x-y)\psi(y)\, d\nu(y)\biggr|^2\, d\nu(x)\\
				&\qquad=\int \biggl|\sum^{\tilde N}_{j\geq 1}\sum_{\ell=1}^{N_j}\int_{|x-y|>\delta}k_{j,\ell}(x) K_{j,\ell}(x,x-y)\psi(y)\, d\nu(y)\biggr|^2\, d\nu(x)\\
				&\qquad=\int \biggl|\sum^{\tilde N}_{j\geq 1}\sum_{\ell=1}^{N_j}k_{j,\ell}(x)T_{K_{j,\ell},\nu,\delta}\psi(x)\biggr|^2\, d\nu(x).
			\end{split}
		\end{equation}

		Hence by \eqref{eq:dct_ml}, \eqref{eq:bound_spherical_coefficients_final} and \eqref{eq:bound_Nj} we  get
		\begin{equation*}
			\begin{split}
				&\biggl(\int \biggl|\sum^{\tilde N}_{j\geq 1}\sum_{\ell=1}^{N_j}k_{j,\ell}(x)T_{K_{j,\ell},\nu,\delta}\psi(x)\biggr|^2\, d\nu(x)\biggr)^{1/2}\\
				&\qquad\qquad\leq \sum^{\tilde N}_{j\geq 1}\sum_{\ell=1}^{N_j}\|k_{j,\ell}\|_{L^\infty(\Rn1\times\Rn1)}\|T_{K_{j,\ell},\nu,\delta}\psi\|_{L^2(\nu)}\\
				&\qquad\qquad\lesssim \biggl(\int_{0}^{\ell(Q)}\omega_A(t) \,\frac{dt}{t} \biggr)^{1/2}\|\mathcal R_\nu\|_{L^2(\nu)\to L^2(\nu)} \sum_{j\geq 1} \frac{N_j}{j^{n+1}}\|\psi\|_{L^2(\nu)}\\
				&\qquad\qquad\lesssim \biggl(\int_{0}^{\ell(Q)}\omega_A(t) \,\frac{dt}{t} \biggr)^{1/2} \|\mathcal R_\nu\|_{L^2(\nu)\to L^2(\nu)}\|\psi\|_{L^2(\nu)}.
			\end{split}
		\end{equation*}
		By the uniform convergence of the decomposition in spherical harmonics and dominated convergence theorem, the estimate above implies that 
		\[
		\|T_{\mathcal K,\nu,\delta}\psi\|_{L^2(\nu)}\lesssim \biggl(\int_{0}^{\ell(Q)}\omega_A(t) \,\frac{dt}{t} \biggr)^{1/2} \|\mathcal R_\nu\|_{L^2(\nu)\to L^2(\nu)}\|\psi\|_{L^2(\nu)}.
		\]
		Arguing by density, this proves the lemma.
	\end{proof}
}

\vv

\begin{lemma}\label{lem:T_bounded_polynomial_growth}
If $A$ is a uniformly elliptic matrix in $\Rn1$, $n\geq 2$, satisfying  $A\in \widetilde \DMO$ and if  $T_{\mu}\colon L^2(\mu) \to L^2(\mu)$ for some $\mu \in M_+(\Rn1)$ without atoms {is bounded}, then there exist $r_0={r_0(n, \Lambda, \diam (\supp\mu))}\in (0,\diam(\supp \mu))$ small enough and $c_0>0$ so  that $\mu(B(x,r)) \leq c_0 r^n$, for all $x \in \Rn1$ and $r<r_0$.
\end{lemma}

\begin{proof}
Let $K(x,y)\coloneqq \nabla_1 \Gamma_A(x,y)$ for $x,y\in\Rn1$, $x\neq y$. The lemma follows from  \cite[Proposition 1.4, p.~56]{Da92} once we show that there exists $r_0 \in (0,\diam(\supp \mu))$ such that for any fixed cube $Q$ of side length  $\ell(Q)< r_0$, it holds $|K(x,y)| \gtrsim |x-y|^{-n}$ for all $x \neq y \in Q$.  To this end, fix a cube  $Q$, and for $x\neq y\in Q$ we have that the matrix $\mathcal A\coloneqq\bar A_{x,|x-y|/2}$ is elliptic with constant $\Lambda$, which yields
\begin{equation}\label{eq:bd_below_nabla_theta}
	|\nabla_1 \Theta (x-y,0;\mathcal A)|\overset{\eqref{eq:gradient_fund_sol_const_matrix}}{=}\frac{\omega^{-1}_n}{\sqrt{\det \mathcal A_s}}\frac{|\mathcal A_s^{-1}(x-y)|}{\langle\mathcal A_s^{-1}(x-y),x-y\rangle^{(n+1)/2}}\geq\frac{C(n,\Lambda)}{|x-y|^n}
\end{equation}
for some $C(n,\Lambda)>0$. Let $C>0$ be as Lemma \ref{lem:main_pw_estimate} for $R=\diam(\supp\mu)$ {and $R_0=2R$}. Then, {for $|x-y|<R$,}
\begin{equation}\label{eq:bd_above_diff_kernels}
	\begin{split}
			|K(x,y)-\nabla_1 \Theta(x-y,0;\mathcal A)|&\leq C\frac{\tau_A(r)}{r^n} + C\frac{\widehat \tau_A(R)}{R^n}\\
			&\leq \frac{C\tau_A(r) + C \widehat \tau_A(R)R^{-n} r^n}{r^n}, \qquad r=\frac{|x-y|}{2}.
	\end{split}
\end{equation}
Thus,  \eqref{eq:bd_below_nabla_theta}, \eqref{eq:bd_above_diff_kernels}, and the triangle inequality imply
\begin{equation*}
	\begin{split}
			|K(x,y)|&\geq |\nabla_1 \Theta(x-y,0;\mathcal A)| - |K(x,y)-\nabla_1 \Theta(x-y,0;\mathcal A)|\\
			&\geq \frac{C(n,\Lambda)-C\tau_A(r) - C \widehat \tau_A(R)R^{-n}r^n}{r^n}.
	\end{split}
\end{equation*}
Since $\tau_A(2^{-j})\to 0$ as $j \to \infty$ and also $\tau_A$ is $c_{db}$-doubling for some constant $c_{db}> 1$, there exists $j_0 \in \mathbb{N}$ such that  for every $j > j_0$,
\[
	C\tau_A(2^{-j}) + C \widehat \tau_A(R)R^{-n} 2^{-jn}   < \frac{C(n,\Lambda)}{2\max(c_{db},{ 2^n})}.
\]
 Therefore, if  $2^{-N} \leq |x-y|< 2^{-N+1}$ for some $N \in \mathbb{N}$ so that $N>j_0$,  it holds that 
 $$
 C\tau_A(|x-y|/2) + C \widehat \tau_A(R)R^{-n} {\frac{|x-y|^n}{2^n}}< \max(c_{db},2^n)\frac{C(n,\Lambda)}{2\max(c_{db},2^n)}=C(n,\Lambda)/2.
 $$
 Therefore, for $|x-y|/2< 2^{-j_0}\eqqcolon r_0$,
\[
	|K(x,y)| \gtrsim |x-y|^{-n},
\]
which proves  the lemma.
\end{proof}

\vvv

\section{The Main Lemmas}\label{section:main_lemma}

Let $A$ be a uniformly elliptic matrix as in Lemma \ref{lem:main_pw_estimate}. In particular, we recall that we introduced the function $\tau_A$ in \eqref{eq:defin_tau_A}.

\begin{lemma}[Main Lemma $\I$]\label{lem:main_lemma}
	Let $A$ be a uniformly elliptic matrix in $\Rn1$, $n\geq 2$, satisfying  $A\in \widetilde \DMO$.
	Let $Q$ be a cube in $\Rn1$ with center $x_Q$ and side-length $\ell(Q) \lesssim 1$,
	and let $\nu\in M^n_+(\Rn1)$ be supported on $Q$ and have $n$-growth constant $c_0>0$.
	{Assume also that there exist a Borel set $\Omega_Q$ and a constant $M\geq 1$ such that \[B(x_Q,\ell(Q))\subset \Omega_Q\subset B(x_Q,M\ell(Q))\] and $({\bar{A}}_{\Omega_Q})_s\coloneqq\avint_{\Omega_Q} A_s=Id$.} Let $T_\nu$ denote the gradient of the single layer potential associated with the matrix $A$.
	Then, there exist positive constants $C'=C'(n,c_0)$ and $C''=C''(n,\Lambda, M,c_0)$ such that for $\delta>0$ we have
	\begin{align}\label{eq:main lemma}
		\|T_{\nu,\delta}-\omega_n^{-1}\mathcal R_{\nu,\delta}\|_{L^2(\nu)\to L^2(\nu)}&\leq C' \mathfrak I_{\tau_A}(\ell(Q)) + C' \widehat \tau_A(\ell(Q))\\
		&\qquad\qquad +  C'' \mathfrak I_{\omega_A}(\ell(Q))^{1/2}\|\mathcal R_{\nu,\delta}\|_{L^2(\nu)\to L^2(\nu)}.\notag
	\end{align}
\end{lemma}

\begin{proof}
	{ Let $\delta>0$ and }
	\[
	\bar T_{\nu,\delta} f(x)\coloneqq \int_{|x-y|>\delta} \nabla_1 \Theta(x-y,0; \bar A_{x,\delta/2})f(y)\, d\nu(y).
	\]
By Lemma \ref{lem:truncatedL2} we have
	\begin{equation}\label{eq:main_lemma_estim1}
		\|T_{\nu,\delta} - \bar T_{\nu,\delta}\|_{L^2(\nu)\to L^2(\nu)}\lesssim_{n,\Lambda} \mathfrak I_{\tau_A}(\ell(Q)) +  \widehat \tau_A(\ell(Q)).
	\end{equation}
	Moreover,
	\[
		\mathcal K^3_\Theta(x,z)\overset{\eqref{eq:definition_K_3}}{=} \nabla_1 \Theta (z,0;\bar A_{x,\delta/2}) - \nabla_1 \Theta\bigl(z,0;\bar A_{\Omega_Q}\bigr)=\nabla_1 \Theta (z,0;\bar A_{x,\delta/2}) - \omega^{-1}_n\frac{z}{|z|^{n+1}},
	\]
	where the second equality holds because of the assumption $(\bar A_{\Omega_Q})_s=Id$. Hence $\bar T_{\nu,\delta}-\omega_n^{-1}\mathcal R_{\nu,\delta}=T_{\mathcal K^3_\Theta, \nu, \delta},$ so Lemma \ref{lem:estimate_norm_K3} concludes the proof of \eqref{eq:main lemma}.
\end{proof}

\vv

In the next lemma we denote by $\mathcal R_{\mu}$ and $T_{\mu}$ the principal values of the corresponding singular integral operators.

{
	\begin{lemma}[Main Lemma $\II$]\label{lem:main2}
		Let $A$ be a uniformly elliptic matrix in $\Rn1$, $n\geq 2$, satisfying  $A\in \widetilde \DMO$, and let $Q$ be a cube in $\Rn1$ with center $x_Q$ and side-length $\ell(Q) \lesssim 1$. 	Let also  $\Omega_Q$  be a Borel set and $M\geq 1$ a constant  such that \[B(x_Q,\ell(Q))\subseteq \Omega_Q\subseteq B(x_Q,M\ell(Q))\] and ${\bar{A}}_{\Omega_Q}\coloneqq\avint_{\Omega_Q} A_s=Id$. {Let $\mu$ be a non-negative Radon measure on $\Rn1$ with compact support.} Assume  that for some integer $N>0$ we have $2^N\ell(Q) \leq \diam(\supp \mu)$  and for a constant $C_0>0$, the measure $\mu$ is such that
		\begin{align}\label{eq:main2-upperADR}
			&\Theta_\mu\bigl(B(x,r)\cap 2^N Q \bigr)\leq C_0 \Theta_\mu(2^NQ),\,\,\qquad\quad\textup{ for all }\, x\in 2^N Q, \,\,0< r\leq 2^N \ell(Q),\\
			&   \mathcal P^N_{\omega,\mu}(Q)  \leq C_0 \mathfrak I_{\alpha_A}\bigl( 2^{-N}\bigr)\Theta_\mu(2^NQ)\label{eq:main2-Poissondoubling}.
		\end{align}
		If $T_\mu$ denotes the gradient of the single layer potential associated with the matrix $A$ and there exists $\tau \in(0, 1)$ such that
		\begin{equation}\label{eq:small_av_T_1}
			\Bigl(\avint_Q \Bigl| T_{\mu}1(x)  -  \avint_QT_{\mu}1 \Bigr|^2 \,d\mu(x)\Bigr)^{1/2} \leq \tau^{1/2}\Theta_\mu(2^NQ),
		\end{equation}
		then it holds that
		\begin{equation*}
			\begin{split}
				\Big(\avint_Q &\Bigl| \mathcal R_{\mu}1(x)  - \avint_Q \mathcal R_{\mu}1 \Bigr|^2 \,d\mu(x)\Big)^{1/2} \leq C_1 \Theta_\mu(2^NQ)\Bigl(\tau^{1/2} + \mathfrak I_{\alpha_A}\bigl(2^{-N}\bigr)\Bigr)\\
				&+  C_1\Theta_\mu(2^NQ)\Bigl(\vartheta\bigl(2^N \ell(Q)\bigr)  + \mathfrak I_{\omega_A}(2^N\ell(Q))^{1/2} \|\mathcal R_{\mu}\|_{L^2(\mu|_{2^NQ})\to L^2(\mu|_{2^N Q})  }\Bigr),
			\end{split}
		\end{equation*}
		where $\vartheta(8 \ell(Q) )\coloneqq\mathfrak I_{\tau_A}(8 \ell(Q)) + \widehat \tau_A(8\ell(Q))$, and $C_1$ depends on $n, \Lambda, c_0, C_0, M$, and $\diam(\supp \mu)$.
	\end{lemma}
\begin{proof}
	Note that  \eqref{eq:main lemma} still holds if  we replace the truncated singular integrals on its left hand-side by  their principal values. This is an easy application of Fatou's lemma and the existence of principal values {given by Proposition \ref{theorem_pv_layer_pot}}.

	For brevity, we write 
	\[
	m_Q(f, \mu)=\avint_Q f\,d\mu. 
	\]
	We define
	\[
		L^2_0(\mu,Q)\coloneqq \bigl\{f\in L^2(\mu):\supp f\subset Q,\, \quad m_Q(f,\mu)=0\bigr\}
	\]
	and by $L^2_0(\mu,Q;\Rn1)$ its vector-valued analogue.
	The space $L^2_0(\mu,Q)$ endowed with the norm $\|\cdot\|_{L^2(\mu)}$ is a Hilbert space whose Banach dual is the space of functions in $L^2(\mu)$ 	modulo an additive constant and equipped with the norm $\|\cdot\|_{L^2(\mu)}$
	(see e.g. \cite[1.2.2, p. 143]{St93}). Moreover, for $f\in L^2(\mu,Q)$ it holds
	\begin{equation}\label{eq:duality_L_0}
		\begin{split}
			\|f-m_Q(f,\mu)\|_{L^2(\mu,Q)}&\approx \sup_{g\in L^2_0(\mu,Q),\atop \|g\|_{L^2(\mu)}=1}\int (f-m_Q(f,\mu))g\, d\mu\\
			&= 	\sup_{g\in L^2_0(\mu,Q),\atop\|g\|_{L^2(\mu)}=1}\int fg\, d\mu,		
		\end{split}
	\end{equation}
	where the second identity follows from $m_Q(g,\mu)=0$.
	
	Then we have that
	\begin{equation*}
		\begin{split}
			\biggl(\int_Q \bigl|\mathcal R_\mu 1(x)- m_Q(\mathcal R_\mu1, \mu(x))\bigr|^2\, d\mu\biggr)^{1/2}&\approx \sup_{\vec g\in L^2_0(\mu,Q;\Rn1),\atop\|\vec g\|_{L^2(\mu;\Rn1)}=1} \biggl|\int \mathcal R_\mu 1\cdot \vec{g}\, d\mu\biggr|\\
			&=\sup_{\vec g\in L^2_0(\mu,Q;\Rn1),\atop\|\vec g\|_{L^2(\mu;\Rn1)}=1} \biggl|\int \mathcal R_\mu^*\vec g\, d\mu\biggr|,
		\end{split}
	\end{equation*}
	where $\mathcal R^*_\mu \vec g(x)\coloneqq \int \tfrac{y-x}{|y-x|^{n+1}}\cdot \vec g(y)\, d\mu(y)$. We also denote
	$\mathcal R_\mu\cdot  \vec g(x)\coloneqq \int \tfrac{x-y}{|x-y|^{n+1}}\cdot \vec g(y)\, d\mu(y)$, so that $\mathcal R^*_\mu g(x)=-\mathcal R_\mu \cdot \vec g(x)$ for all $x\in \Rn1$.

	Then, for $\vec g \in L^2(\mu,Q;\Rn1)$ with $\|\vec g\|_{L^2(\mu,Q;\Rn1)}=1$ and $\mathcal S_\mu \cdot \vec g\coloneqq \omega_n T_\mu\cdot \vec g - \mathcal R_\mu\cdot \vec g$, triangle inequality yields
	\begin{equation*}
		\begin{split}
			\biggl|\int \mathcal R_\mu^*\vec g\, d\mu\biggr|&=\biggl|\int \mathcal R_\mu\cdot \vec g\, d\mu\biggr|\\
			&\lesssim \biggl|\int  T_\mu\cdot \vec g\, d\mu\biggr| + \biggl|\int \mathcal S_\mu\cdot \vec g\, d\mu\biggr| \eqqcolon I + II.
		\end{split}
	\end{equation*}
	By \eqref{eq:duality_L_0} and the hypothesis \eqref{eq:small_av_T_1} we have
	\begin{equation}\label{eq:estim_lemma42_I}
		\begin{split}
			I \lesssim \|T_\mu1 - m_{Q}(T_\mu1, \mu)\|_{L^2(\mu)}\leq \tau^{1/2} \Theta_\mu(2^NQ)\mu(Q)^{1/2}.
		\end{split}
	\end{equation}
	We denote
	\[
		\mathfrak K(x,y)\coloneqq \omega_n \nabla_1\Gamma_A(x,y) - \frac{x-y}{|x-y|^{n+1}}, \qquad x,y\in \Rn1, x\neq y.
	\]
	In order to estimate $I$, we first observe that Lemma \ref{lem:estim_fund_sol} and the standard Calder\'on-Zygmund properties of the Riesz kernel imply that
	\begin{equation}\label{eq:CZ_mathfrak_K}
		|\mathfrak K(x,y)-\mathfrak K(x,z)|\lesssim_{n,\Lambda, R} \alpha_A\Bigl(\frac{|y-z|}{|x-y|}\Bigr)|x-y|^{-n}
	\end{equation}
	for $2|y-z|\leq |x-y|\leq R$, where
	\begin{equation}
		 \alpha_A(t)\coloneqq t^\beta + t + \omega_A(t), \qquad t>0.
	\end{equation}
	Moreover, $A\in \widetilde\DMO$ implies $\alpha_A \in \DS(\kappa)$.
	
	Now, we write
	\begin{equation*}
		\begin{split}
			II\leq  \biggl|\int_{2^NQ} \mathcal S_\mu\cdot \vec g\, d\mu \biggr|+ \biggl|\int_{\Rn1\setminus 2^NQ} \mathcal S_\mu\cdot \vec g\, d\mu \biggr|\eqqcolon II_1 + II_2.
		\end{split}
	\end{equation*}
	In order to estimate $II_1$ we apply Lemma \ref{lem:main_lemma}, Cauchy-Schwarz inequality and the assumption $\|\vec g\|_{L^2(\mu;\Rn1)}=1$, which give
	\begin{equation*}
		\begin{split}
			II_1 &\leq  C' \Theta_\mu(2^NQ)\vartheta\bigl(2^N\ell(Q)\bigr) \mu(Q)^{1/2} \\
			&\qquad + C''\Theta_\mu(2^NQ)\mathfrak I_{\omega_A}\bigl(2^N\ell(Q)\bigr)^{1/2}\|\mathcal R_\mu\|_{L^2(\mu|_{2^NQ})\to L^2(\mu|_{2^NQ})}\mu(Q)^{1/2},
		\end{split}
	\end{equation*}
	where the multiplicative factor $\Theta_\mu(2^NQ)$ on the right hand side is a consequence of \eqref{eq:main2-upperADR}.
	
	Denote by $x_Q$ the center of the cube $Q$. To estimate $II_2$, observe that $\vec g\in L^2_0(\mu,Q;\Rn1)$ and \eqref{eq:CZ_mathfrak_K} imply that
	\begin{equation}\label{eq:estimate_lemma42_II2}
		\begin{split}
			II_2 &\leq  \int_{\Rn1 \setminus 2^NQ} \int_Q |\mathfrak K(x,y) - \mathfrak K(x,x_Q)||\vec g(y)|\, d\mu(y)\, d\mu(x)\\
				&\lesssim \int_{\Rn1 \setminus 2^NQ} \int_Q \alpha_A\Bigl(\frac{|y-x_Q|}{|x-y|}\Bigr)\frac{1}{|x-y|^n}|\vec g(y)|\, d\mu(y)\, d\mu(x)\\
				&\leq \sum_{j\geq N} \int_{2^{j+1}Q\setminus 2^j Q}\int_Q \alpha_A\Bigl(\frac{|y-x_Q|}{|x-y|}\Bigr)\frac{1}{|x-y|^n}|\vec g(y)|\, d\mu(y)\, d\mu(x)\lesssim \mathcal P^N_{\omega,\mu}(Q)\mu(Q)^{1/2},
		\end{split}
	\end{equation}
	where the last inequality follows from the definition of $\mathcal P^N_{\omega,\mu}(Q)$, the doubling property of $\alpha_A$, and the assumption $\|\vec g\|_{L^2(\mu;\Rn1)}= 1$.
	
	The bounds \eqref{eq:estim_lemma42_I}, \eqref{eq:estimate_lemma42_II2}, \eqref{eq:estimate_lemma42_II2}, and the assumption \eqref{eq:main2-Poissondoubling} conclude the proof of the lemma.
\end{proof}
}

\vvv

\section{The approximating measures}\label{section:approxmeasure}
Let $Q$ be a cube in $\Rn1$ and let $\nu\in M^n_+(Q)$. We fix a function $\varphi \in C^\infty_c(\Rn1)$ such that $\supp \varphi\subset B(0,1)$, {$0\leq \varphi\leq 2$}, and $\|\varphi\|_{L^1(\Rn1)}=1$. Given $\varepsilon>0$, we denote
\[
\varphi_{\varepsilon}(z)\coloneqq \frac{1}{\varepsilon^{n+1}}\varphi\Bigl(\frac{z}{\varepsilon}\Bigr)\qquad \text{ for }z\in \Rn1,
\]
and we define
\begin{equation}\label{eq:definition_measure_nu_varepsilon}
	\nu_\varepsilon\coloneqq \nu*\varphi_\varepsilon.
\end{equation}

The measure $\nu_\varepsilon$ here introduced is absolutely continuous with respect to the Lebesgue measure $\mathcal L^{n+1}$, its support is contained in the  set $\{x\in\Rn1:\dist(x,\supp\nu)\leq \varepsilon\}$ and it satisfies $\|\nu_\varepsilon\|=\|\nu\|$ for all $\varepsilon>0$. 
The following lemma shows that, under our hypotheses on the matrix $A$, the $L^2(\nu)$-boundedness of $T_\nu$ controls the $L^2(\nu_\varepsilon)$-boundedness of $T_{\nu_\varepsilon}$. 
\begin{lemma}\label{lemma:approx1}
	Let $A$ be a uniformly elliptic matrix in $\Rn1$, $n\geq 2$, satisfying  $A\in \widetilde \DMO$. Let $\nu\in M^n_+(Q)$ with growth constant $c_0>0$, and let $\nu_\varepsilon$ be as in \eqref{eq:definition_measure_nu_varepsilon} for $\varepsilon>0$. Then
	\begin{equation}\label{eq:L2nu_eps_toL2_nu}
		\|T_{\nu_\varepsilon}\|_{L^2(\nu_\varepsilon)\to L^2(\nu_\varepsilon)}\lesssim 1 + \|T_\nu\|_{L^2(\nu)\to L^2(\nu)},
	\end{equation}
	where the implicit constant depends on $n$, $c_0$, $\Lambda$ and $\diam(\supp \nu)$.
\end{lemma}

\begin{proof}
	Given $f\in L^2(\nu_\varepsilon)$ and $\varepsilon>0$, we define $\sigma_\varepsilon\coloneqq f\nu_\varepsilon$. Let $N\in \mathbb N$ be such that $ 2^{-N-1}\ell(Q)\leq \varepsilon<2^{-N}\ell(Q)$. Let $\{Q_i\}_i$ be a family of $\widetilde N\coloneqq 3^{n+1}2^{N(n+1)}$ cubes with disjoint interior and side-length $2^{-N}\ell(Q)$ that cover $3Q$. We denote by $x_{Q_i}$ the center of the cube $Q_i$.	For any $i=1,\ldots,\widetilde N$ we define the set
	\(
	v(Q_i)\coloneqq 3Q_i\cap 3Q,
	\)
	which consists of the union of at most $3^{n+1}$ cubes of side-length $2^{-N}\ell(Q)$.
	We also define
	\begin{equation}\label{eq:def_sigma_ep_i}
		\tilde \sigma_{\varepsilon, i}\coloneqq \frac{\sigma_{\varepsilon}(Q_i)}{\nu(v(Q_i))}\nu|_{v(Q_i)}=\Bigl(\frac{\sigma_{\varepsilon}(Q_i)}{\nu(v(Q_i))}\chi|_{v(Q_i)}\Bigr)\nu\eqqcolon \tilde f_{\varepsilon,i}\,\nu,
	\end{equation}
	and also
	\begin{equation}\label{eq:def_sigma_ep}
		\tilde \sigma_\varepsilon\coloneqq \sum^{\widetilde{N}}_{i=1} \tilde \sigma_{\varepsilon,i} = \sum^{\widetilde{N}}_{i=1} \tilde f_{\varepsilon,i} \,\nu\eqqcolon\tilde f_\varepsilon\,\nu.
	\end{equation}
	We claim that $\tilde f_\varepsilon \in L^2(\nu)$ satisfying
	\begin{equation}\label{eq:claim_norm_f_varepsilon}
		\|\tilde f_\varepsilon\|_{L^2(\nu)}\lesssim_n\|f\|_{L^2(\nu_\varepsilon)}.
	\end{equation}
	Observe that the choices of $\varepsilon$ and $N$ yield \[\{z\in\Rn1:\dist(z,Q_i)<\varepsilon\}\subset v(Q_i).\]
	Thus, an application of Fubini's theorem and the choice of the cut-off function $\varphi$ give
	\begin{equation}\label{eq:nuep_leqnuv}
		\begin{split}
			\nu_\varepsilon(Q_i)&=\int_{Q_i}\int \varphi_\varepsilon(x-y)\, d\nu(y)\, dx\\&=\int_{\{z\in\Rn1:\dist(z,Q_i)<\varepsilon\}}\int_{Q_i}\varphi_\varepsilon(x-y)\,dx \,d\nu(y) \\
			&\leq \nu\bigl(\{z\in\Rn1:\dist(z,Q_i)<\varepsilon\}\bigr)\leq \nu(v(Q_i)).
		\end{split}
	\end{equation}
	
	Note that there exists a dimensional constant $c_n>0$ such that, for $i=1,\ldots,\tilde N$, the set $v(Q_i)$ has non-empty intersection with at most $c_n$ different sets of the form $v(Q_j)$ for $j=1,\ldots,\tilde N$. Analogously, every cube $Q_j$ can be contained in at most $c_n$ different sets of the form $v(Q_i)$. For $i=1,\ldots,\tilde N$ we also define the cube $Q^*_i$ as a cube such that $v(Q_i^*)\cap v(Q_i)\neq \varnothing$ and, for all $j$ such that $v(Q_j)\cap v(Q_i)\neq \varnothing$, it holds
	\begin{equation}\label{eq:definition_Q_j_star}
		\frac{|\sigma_\varepsilon(Q_j)|}{\nu(v(Q_j))}\leq \frac{|\sigma_\varepsilon(Q^*_i)|}{\nu(v(Q^*_i))}.
	\end{equation}
	Hence, by elementary inequalities  and the definitions above we obtain
	\begin{equation}\label{eq:L2boundtildefepsilon}
		\begin{split}
			\|\tilde f_\varepsilon\|^2_{L^2(\nu)}&=\int \biggl|\sum^{\widetilde{N}}_{i=1} \frac{\sigma_\varepsilon(Q_i)}{\nu(v(Q_i))}\chi_{v(Q_i)}(x)	\biggr|^2\, d\nu(x)\\
			&= \sum^{\widetilde{N}}_{i=1} \frac{\sigma_\varepsilon(Q_i)^2}{\nu(v(Q_i))}  + \sum_{i, j =1, i\neq j}^{\widetilde{N}} \int \frac{\sigma_\varepsilon(Q_i)}{\nu(v(Q_i))}\frac{\sigma_\varepsilon(Q_j)}{\nu(v(Q_j))}\chi_{v(Q_i)\cap v(Q_j)}(x)\, d\nu(x)\\
			&\overset{\eqref{eq:definition_Q_j_star}}{\leq} \sum^{\widetilde{N}}_{i=1} \frac{\sigma_\varepsilon(Q_i)^2}{\nu(v(Q_i))}  + c_n\sum_{i=1}^{\widetilde{N}}\int \frac{\sigma_\varepsilon(Q^*_i)^2}{\nu(v(Q^*_i))^2}\chi_{v(Q_i)}(x)\, d\nu(x).
		\end{split}
	\end{equation}
	
	The first sum on the right hand side of \eqref{eq:L2boundtildefepsilon} satisfies
	\begin{equation}\label{eq:L2boundtildefepsilon2}
		\begin{split}
			\sum^{\widetilde{N}}_{i=1} \frac{(\sigma_\varepsilon(Q_i))^2}{\nu(v(Q_i))}&=\sum^{\widetilde{N}}_{i=1} \frac{1}{\nu(v(Q_i))}\biggl(\int_{Q_i}f\, d\nu_\varepsilon\biggr)^2\\
			&\leq \sum^{\widetilde{N}}_{i=1} \frac{\nu_\varepsilon(Q_i)}{\nu(v(Q_i))}\|f\chi_{Q_i}\|^2_{L^2(\nu_\varepsilon)}\overset{\eqref{eq:nuep_leqnuv}}{\leq} \|f\|^2_{L^2(\nu_\varepsilon)}
		\end{split}
	\end{equation}
	and, analogously, we have
	\begin{equation}\label{eq:L2boundtildefepsilon3}
		\sum_{i=1}^{\widetilde{N}}\int \frac{\sigma_\varepsilon(Q^*_i)^2}{\nu(v(Q^*_i))^2}\chi_{v(Q_i)}(x)\, d\nu(x)\lesssim_n \|f\|^2_{L^2(\nu_\varepsilon)}.
	\end{equation}
	Hence, by  \eqref{eq:L2boundtildefepsilon2} and \eqref{eq:L2boundtildefepsilon3}, we get  \eqref{eq:claim_norm_f_varepsilon}.
	
	\vv
	Let $K(x,y)\coloneqq \nabla_1\Gamma(x,y;A)$, for $x,y\in\Rn1$ with $x\neq y$ and, for $\delta>0$, we define $K_\delta(x,y)\coloneqq K(x,y) \chi_{B(0,\delta)^c}(x-y).$
	For $\delta\in (0,\varepsilon/2)$ we write
	\begin{align}
		T_\delta \sigma_\varepsilon(x)&= \int_{|x-y|< \varepsilon}K_\delta(x,y)\, d\sigma_\varepsilon(y) + \int_{|x-y|\geq \varepsilon}K_\delta(x,y)\, d\bigl(\sigma_\varepsilon - \tilde\sigma_\varepsilon\bigr)(y) \notag\\
		&\qquad\qquad  \qquad + \int_{|x-y|\geq \varepsilon}K_\delta(x,y)\, d\tilde\sigma_\varepsilon(y)\eqqcolon I_{\delta,\varepsilon}(x) + II_{\delta,\varepsilon}(x) + III_{\delta,\varepsilon}(x).\label{eq:Tdelta}
	\end{align}
	In order to estimate the first term, we observe that, by the definition of $\varphi$ and the growth of $\nu$, we have
	\begin{equation}\label{eq:measure_nu_ep_dy}
		\begin{split}
			\nu_\varepsilon\bigl(B(x,2^{-k}\varepsilon)\bigr)&=\frac{1}{\varepsilon^{n+1}}\int_{B(x,2^{-k}\varepsilon)}\int\varphi\Bigl(\frac{y-z}{\varepsilon}\Bigr)\, d\nu(z)\, dy\\
			&\leq \frac{2}{\varepsilon^{n+1}}\int_{B(x,2^{-k}\varepsilon)}\nu(B(y,\varepsilon))\, dy\leq \frac{2}{\varepsilon}\bigl|B(x,2^{-k}\varepsilon)\bigr|\lesssim_n \frac{\varepsilon^n}{2^{k(n+1)}}.
		\end{split}
	\end{equation}
	Hence, if $\M_{\nu_\varepsilon}$ stands for the centered Hardy-Littlewood maximal function
	\[
	\M_{\nu_\varepsilon}g(x)\coloneqq \sup_{r>0}\frac{1}{\nu_\varepsilon(B(x,r))}\int_{B(x,r)}|g(y)|\, d\nu_\varepsilon(y),\qquad \text{ for }g\in L^1_{\loc}(\nu_\varepsilon),
	\]
	the decay of $K$, a standard integration over dyadic annuli, and the definition of $\sigma_\varepsilon$ give
	\begin{equation}\label{eq:estimateIdeltaep}
		\begin{split}
			\bigl|I_{\delta,\varepsilon}(x)\bigr|&\lesssim \sum_{k=0}^\infty\int_{A(x,2^{-k-1}\varepsilon, 2^{-k}\varepsilon)}\frac{1}{|x-y|^n}\, d|\sigma_\varepsilon|(y)\\
			&\lesssim \sum_{k=0}^\infty \frac{2^{nk}}{\varepsilon^n}|\sigma_\varepsilon|\bigl(B(x,2^{-k-1}\varepsilon)\bigr)\\
			&\lesssim \sum_{k=0}^\infty \frac{2^{nk}}{\varepsilon^n}\nu_\varepsilon\bigl(B(x,2^{-k}\varepsilon)\bigr)\M_{\nu_\varepsilon}f(x)\overset{\eqref{eq:measure_nu_ep_dy}}{\lesssim}\M_{\nu_\varepsilon}f(x).
		\end{split}
	\end{equation}

	Let us estimate $II_{\delta_\varepsilon}(x)$. {Let $\beta>0$ be as in Lemma \ref{lem:estim_fund_sol}. For $i\in \{1,\ldots, 	\tilde N\}$, the fact that $\sigma_{\varepsilon,i}\coloneqq \sigma_\varepsilon|_{Q_i}$ and $\tilde \sigma_{\varepsilon, i}$ have equal total mass, Lemma \ref{lem:estim_fund_sol}, triangle inequality, and the choice of $N$ in the construction yield
		\begin{equation}\label{eq:estim_II_d_e_i}
			\begin{split}
				\biggl|\int_{|x-y|>\varepsilon} &K_\delta (x,y)\, d\bigl(\sigma_{\varepsilon,i} - \tilde \sigma_{\varepsilon,i}\bigr)(y)\biggr|\\
				&=\biggl|\int_{|x-y|>\varepsilon}\bigl(K_\delta (x,y)- K_\delta(x,x_{Q_i})\bigr)\, d\bigl(\sigma_{\varepsilon,i} - \tilde \sigma_{\varepsilon,i}\bigr)(y)\biggr|\\
				&  \leq \int_{|x-y|>\varepsilon}|K(x,y)- K(x,x_{Q_i})|\, d\bigl(|\sigma_{\varepsilon,i}|+|\tilde\sigma_{\varepsilon,i}|\bigr)(y)\\
				& \lesssim \int_{|x-y|>\varepsilon} \frac{|y-x_{Q_i}|^\beta}{|x-y|^{n+\beta}}\, d\bigl(|\sigma_{\varepsilon,i}|+|\tilde\sigma_{\varepsilon,i}|\bigr)(y) \\
				&\qquad\quad+ \int_{|x-y|>\varepsilon} \frac{1}{|x-y|^{n}}\int_0^{\frac{|y-x_{Q_i}|}{|x-y|}}\omega_A(t)\,\frac{dt}{t}\, d\bigl(|\sigma_{\varepsilon,i}|+|\tilde\sigma_{\varepsilon,i}|\bigr)(y)\\
				&\eqqcolon II'_{\delta,\varepsilon}+II''_{\delta,\varepsilon}.
			\end{split}
		\end{equation}
		Thus	
		\begin{equation}\label{eq:estim_II_prime}
			\begin{split}	
				II'_{\delta,\varepsilon}&\lesssim \varepsilon^\beta\int_{|x-y|>\varepsilon}\frac{1}{|x-y|^{n+\beta}}\, d\bigl(|\sigma_{\varepsilon, i}| + |\tilde\sigma_{\varepsilon, i}|\bigr)(y)\\
				&\lesssim \sum_{k=0}^\infty \frac{\varepsilon^\beta}{(2^k\varepsilon)^{n+\beta}}\Bigl(|\sigma_\varepsilon|\bigl(B(x,2^{k+1})\bigr) + |\tilde\sigma_\varepsilon|\bigl(B(x,2^{k+1})\bigr)\Bigr),
			\end{split}
		\end{equation}
		where the last inequality follows from a standard integration on dyadic annuli.
		Similarly, the second term can be bounded using the monotonicity of the function $\mathfrak I_{\omega_A}$ and integration on dyadic annuli. More precisely, we have
		\begin{equation}\label{eq:estim_II_second}
			\begin{split}
				II''_{\delta,\varepsilon}&\lesssim \int_{|x-y|>\varepsilon}\frac{1}{|x-y|^{n}}\mathfrak I_{\omega_A}\Bigl(\frac{\varepsilon}{|x-y|}\Bigr)\, d\bigl(|\sigma_{\varepsilon, i}| + |\tilde\sigma_{\varepsilon, i}|\bigr)(y)\\
				&\lesssim \sum_{k=0}^\infty \frac{\tau_A(2^{-k})}{(2^k\varepsilon)^{n}}\Bigl(|\sigma_\varepsilon|\bigl(B(x,2^{k+1})\bigr) + |\tilde\sigma_\varepsilon|\bigl(B(x,2^{k+1})\bigr)\Bigr),
			\end{split}
		\end{equation}
		where we also used the fact that $\mathfrak I_{\omega_A}(\cdot)\leq \tau_A(\cdot)$.
		Now, for $\mu\in M(\Rn1)$, we define the $n$-dimensional truncated radial maximal operator
		\[
		\M_{\varepsilon}\mu(x)\coloneqq \sup_{r\geq 2\varepsilon}\frac{|\mu|\bigl(B(x,r)\bigr)}{r^n},
		\]
		and the truncated centered maximal function
		\[
		\mathcal M_{\mu, \varepsilon}g(x)\coloneqq \sup_{r>2\varepsilon}\frac{1}{\mu(B(x,r))}\int_{B(x,r)}|g(y)|\, d\mu(y),\qquad \text{ for }g\in L^1_{\loc}(\mu).
		\]
		
		So  if we gather \eqref{eq:estim_II_d_e_i}, \eqref{eq:estim_II_prime}, and \eqref{eq:estim_II_second}, and sum over $i=1,\ldots,\tilde N$, in view of the fact that $\tau_A\in \DS(\kappa)$ for some $\kappa=\kappa(n)$ and \eqref{eq:mod_cont_sum}, we deduce that
		\begin{equation}\label{eq:estimateIIdeltaep}
			\begin{split}
				|II_{\delta,\varepsilon}(x)|&\lesssim \sum_{k=0}^\infty \bigl(2^{-k\beta} + \tau_A(2^{-k})\bigr)\bigl(\M_{\varepsilon}\sigma_\varepsilon(x) + \M_{\varepsilon}\tilde \sigma_\varepsilon(x)\bigr)\\
				&\lesssim \Bigl(\sum_{k=0}^\infty 2^{-k\beta} + \int_0^1\tau_A(t)\, \frac{dt}{t}\Bigr)\bigl(\M_{\varepsilon}\sigma_\varepsilon(x) + \M_{\varepsilon}\tilde \sigma_\varepsilon(x)\bigr)\\
				&\lesssim \M_{\varepsilon}\sigma_\varepsilon(x) + \M_{\varepsilon}\tilde \sigma_\varepsilon(x).
			\end{split}		
		\end{equation}
		
		{We claim that the operator $\M_{\varepsilon}$ is bounded from $M(\Rn1)$ to $L^{1,\infty}(\nu_{\varepsilon})$, with operator norm independent on $\varepsilon$.
			Indeed let $\mu\in M(\Rn1)$, consider $\lambda,m>0$, and let us denote
			\[
			A_\lambda\coloneqq \bigl\{x\in \Rn1: \M_{\varepsilon}\mu(x)>\lambda\bigr\}\qquad \text{ and }\qquad A_{\lambda,m}\coloneqq A_\lambda\cap B(0,m).
			\]
			Thus, for every $x\in A_{\lambda,m}$ there exists $r_x>2\varepsilon$ such that
			\[
			|\mu|(B(x,r_x))>\lambda r^n_x.
			\]
			Fubini's theorem, the normalization $\|\varphi\|_{L^1(\Rn1)}=1$, and the choice of $r_x$ imply that, for all $x\in A_{\lambda,m}$, we have
			\begin{equation}\label{eq:nuleq_32}
				\begin{split}
					\nu_\varepsilon(B(x,r_x))&=\int_{B(x,r_x)}\int \varphi_\varepsilon(y-z)\, d\nu(z)dy
					\leq \nu(B(x,r_x+\varepsilon))\lesssim \Bigl(\frac{3}{2}r_x\Bigr)^n.
				\end{split}
			\end{equation}
			Besicovitch covering Lemma implies that there exists a countable collection of balls $\{B_i\}_i\subset \bigl\{B(x,r_x)\bigr\}_{x\in A_{\lambda,m}}$ with bounded overlaps that covers $A_{\lambda,m}$. Hence
			\[
			\nu_\varepsilon(A_{\lambda, m})\leq \nu_\varepsilon\Bigl(\bigcup_i B_i\Bigr)\leq \sum_i\nu_\varepsilon(B_i)\overset{\eqref{eq:nuleq_32}}{\lesssim} \frac{3^n}{2^n}\sum_i \frac{|\mu|(B_i)}{\lambda}{\lesssim} \frac{3^n}{2^n\lambda}\|\mu\|.
			\]
			Since the latter estimate holds for every $m$, our claim follows.} 
		
		Moreover, the fact that $\nu$ and $\nu_\varepsilon$ have $n$-polynomial growth implies that $\M_{\nu_\varepsilon, \varepsilon}$ and $\M_{\nu, \varepsilon}$ are bounded from $L^\infty(\nu_\varepsilon)$ to $L^\infty(\nu_\varepsilon)$ and from $L^\infty(\nu)$ to $L^\infty(\nu_\varepsilon)$, respectively. Thus, Marcinkiewicz interpolation theorem implies that 
		\begin{equation}\label{eq:maximalfunctionnue}
			\|\M_{\nu_\varepsilon,\varepsilon} \,g \|_{L^2(\nu_\varepsilon)} \lesssim  \| g \|_{L^2(\nu_\varepsilon)}  \quad \text{and}\quad \|\M_{\nu,\varepsilon} \,g \|_{L^2(\nu_\varepsilon)} \lesssim  \| g \|_{L^2(\nu)}.
	\end{equation}}
	
	Now we turn our attention to  $III_{\delta,\varepsilon}(x)$. Since we assumed $\delta<\varepsilon/2$, we have
	\begin{equation*}
		|III_{\delta,\varepsilon}(x)|\leq \bigl|T_\varepsilon\tilde \sigma_\varepsilon(x)\bigr| + \M_{\varepsilon}\tilde \sigma_\varepsilon(x).
	\end{equation*}
	Therefore, by \eqref{eq:Tdelta}, \eqref{eq:estimateIdeltaep}, \eqref{eq:estimateIIdeltaep}, and the inequality above, we infer that
	\begin{align}
		|T_{\nu_\varepsilon,\delta}f(x)| =	|T_\delta\sigma_\varepsilon(x) | &\lesssim \M_{\varepsilon}\sigma_\varepsilon(x) + \M_{\varepsilon}\tilde \sigma_\varepsilon(x)+\M_{\nu_\varepsilon}f(x) + \bigl|T_\varepsilon\tilde \sigma_\varepsilon(x)\bigr| \notag\\
		& = \M_{\nu, \varepsilon} f_\varepsilon(x)+\M_{\nu, \varepsilon}\tilde f_\varepsilon(x) +\M_{\nu_\varepsilon}f(x) + \bigl|T_\varepsilon\tilde \sigma_\varepsilon(x)\bigr|.\label{eq:T_nuepsilon_delta}
	\end{align}
	We claim that, for $i=1,\ldots,\tilde N$, it holds
	\begin{equation}\label{eq:estimTTlem62}
		|T_\varepsilon \tilde \sigma_\varepsilon(x)- T_\varepsilon \tilde \sigma_\varepsilon(x')|\lesssim \M_{\varepsilon}\tilde \sigma_\varepsilon(z)\qquad \text{ for all }x,x',z\in 3Q_i.
	\end{equation}
	Indeed, for $x,x'\in 3Q_i$, observe that the choice of $\varepsilon$ implies $|x-x'|\leq \diam (3Q_i)= 3\sqrt{n+1}\ell(Q_i)\leq 6 n \varepsilon$. Furthermore, by triangle inequality we have 
	$$
	B(x',\varepsilon)\subset B(x,10n\varepsilon)\subset B(x', 20n\varepsilon).
	$$
	Thus, we can write
	\begin{equation*}
		\begin{split}
			&|T_\varepsilon \tilde \sigma_\varepsilon(x)- T_\varepsilon \tilde \sigma_\varepsilon(x')|=\Bigl|\int_{|x-y|>\varepsilon} K(x,y)\, d\tilde\sigma_\varepsilon(y) - \int_{|x'-y|>\varepsilon} K(x',y)\, d\tilde\sigma_\varepsilon(y)\Bigr|\\		
			&\qquad\leq \int_{|x-y|>10n\varepsilon}|K(x,y)-K(x',y)|\, d|\tilde\sigma_\varepsilon|(y) + \int_{\varepsilon<|x-y|\leq 10n\varepsilon}|K(x,y)|\, d|\tilde \sigma_\varepsilon|(y) \\
			&\qquad \qquad\qquad+ \int_{\varepsilon<|x'-y|\leq 20n\varepsilon}|K(x',y)|\, d|\tilde \sigma_\varepsilon|(y) \eqqcolon \mathcal J_1 + \mathcal J_2 + \mathcal J_3.\
		\end{split}
	\end{equation*}
	The Calder\'on-Zygmund properties of $K$ in Lemma \ref{lem:estim_fund_sol}, a standard integration over dyadic annuli analogous to \eqref{eq:estim_II_d_e_i}, the inequality $\mathfrak I_{\omega_A}(\cdot)\leq \tau_A(\cdot)$, and the assumption $\tau_A\in \DS(\kappa)$ imply
	\begin{equation}\label{eq:estimate_mathfrak_1}
		\begin{split}
			\mathcal J_1 &\lesssim \int_{|x-y|>\varepsilon}\Bigl(\frac{|x-x'|^\beta}{|x-y|^{n+\beta}}  + \frac{\int_{0}^{\frac{|x'-y|}{|x-y|}}\omega_A(t) t^{-1}\, dt}{|x-y|^n}\Bigr)\, d|\tilde \sigma_\varepsilon|(y)\\
			&\lesssim \int_{|x-y|>\varepsilon} \Bigl(\frac{\varepsilon^\beta}{|x-y|^{n+\beta}} + \frac{\mathfrak I_{\omega_A}\bigl(\varepsilon/|x-y|\bigr)}{|x-y|^n}\Bigr)\,d|\tilde \sigma_\varepsilon|(y)\\
			&\lesssim \sum_{k=0}^\infty \Bigl(\frac{\varepsilon^\beta}{(2^k\varepsilon)^{n+\beta}} + \frac{\tau_A(2^{-k})}{(2^k\varepsilon)^n} \Bigr)|\tilde \sigma_\varepsilon|\bigl(B(x,2^k\varepsilon)\bigr)\\
			&\lesssim \sum_{k=0}^\infty \bigl(2^{-k\beta} + \tau_A(2^{-k})\bigr)\M_{\varepsilon}\tilde \sigma_\varepsilon(z)\\
			&\overset{\eqref{eq:mod_cont_sum}}{\lesssim}\Bigl(1+\int_0^1 \tau_A(t)\, \frac{dt}{t}\Bigr)\M_{\varepsilon}\tilde \sigma_\varepsilon(z)\lesssim \M_{\varepsilon}\tilde \sigma_\varepsilon(z).
		\end{split}
	\end{equation}
	Analogously, we can  prove that
	\begin{equation}\label{eq:estimate_mathfrak_2}
		\begin{split}
			\mathcal J_2 + \mathcal J_3 &\lesssim \int_{\varepsilon<|x-y|<10n\varepsilon}\frac{1}{|x-y|^{n}}\, d|\tilde \sigma_\varepsilon|(y) + \int_{\varepsilon<|x'-y|<20n\varepsilon}\frac{1}{|x'-y|^{n}}\, d|\tilde \sigma_\varepsilon|(y)\\
			&\lesssim \M_{\varepsilon}\tilde \sigma_\varepsilon(z).
		\end{split}
		\
	\end{equation}
	Combining \eqref{eq:estimate_mathfrak_1} and \eqref{eq:estimate_mathfrak_2} we get  \eqref{eq:estimTTlem62}.
	
	Finally, in light of  $\nu_\varepsilon(Q_i)\leq \nu(3Q_i)$, the definition of $\nu_\varepsilon$, and  \eqref{eq:estimTTlem62}, we have that
	\begin{equation}\label{eq:estimateIIIdeltaep}
		\begin{split}
			\int |T_\varepsilon\tilde \sigma_\varepsilon(x)|^2\, d\nu_\varepsilon(x)&=\sum_{i=1}^{\tilde N}\int_{Q_i}|T_\varepsilon\tilde \sigma_\varepsilon(x)|^2\, d\nu_\varepsilon(x)\\
			&\leq \Bigl(\sum_{i=1}^{\tilde N}\nu_\varepsilon(Q_i)\Bigr)	\sup_{x\in 3Q_i}|T_\varepsilon \tilde \sigma_\varepsilon (x)|^2	\\
			&\leq \Bigl(\sum_{i=1}^{\tilde N}\nu_\varepsilon(Q_i)\Bigr) \Bigl(\inf_{z\in 3Q_i}|T_\varepsilon \tilde \sigma_\varepsilon (z)|^2 + \inf_{z\in 3Q_i} \bigl(\M_{\varepsilon}\tilde \sigma_\varepsilon(z)\bigr)^2\Bigr)\\
			&\leq \sum_{i=1}^{\tilde N}\int_{3Q_i}|T_\varepsilon \tilde \sigma_\varepsilon(x)|^2\, d\nu(x) +  \sum_{i=1}^{\tilde N} \int_{3Q_i}\bigl(\M_{\varepsilon}\tilde\sigma_\varepsilon(x)\bigr)^2\, d\nu(x)\\
			&\leq \|T_\varepsilon \tilde \sigma_\varepsilon\|^2_{L^2(\nu)} + \|\M_{\varepsilon}\tilde\sigma_\varepsilon\|^2_{L^2(\nu)}\\
			&= \|T_{\nu, \varepsilon} \tilde f_\varepsilon\|^2_{L^2(\nu)} + \|\M_{\nu, \varepsilon}\tilde f_\varepsilon\|^2_{L^2(\nu)}.
		\end{split}
	\end{equation}
	Analogously to what we have previously done, we can prove that the operator $\M_{\nu,\varepsilon}$ is bounded from $L^2(\nu)$ to $L^2(\nu)$. Hence, by \eqref{eq:T_nuepsilon_delta}, \eqref{eq:estimateIIIdeltaep}, and \eqref{eq:maximalfunctionnue}, we infer that 
	\begin{equation*}
		\begin{split}
			\|T_{\nu_\varepsilon,\delta}f\|_{L^2(\nu_\varepsilon)}&\lesssim \|\M_{\nu_\varepsilon}f\|_{L^2(\nu_\varepsilon)} + \|\M_{\nu_\varepsilon,\varepsilon}f\|_{L^2(\nu_\varepsilon)} + \|\M_{\nu,\varepsilon}\tilde f_\varepsilon\|_{L^2(\nu_\varepsilon)} \\
			&\qquad\qquad  + \|T_{\varepsilon}\tilde \sigma_\varepsilon\|_{L^2(\nu_\varepsilon)} + \|\M_{\nu,\varepsilon}\tilde f_\varepsilon\|_{L^2(\nu)}\\
			&\lesssim \|f\|_{L^2(\nu_\varepsilon)} + \|\tilde f_\varepsilon\|_{L^2(\nu)} + \|T_\nu\|_{L^2(\nu)\to L^2(\nu)}\|\tilde f_\varepsilon\|_{L^2(\nu)}\\
			&\overset{\eqref{eq:L2boundtildefepsilon2}}{\lesssim} \bigl(1 + \|T_\nu\|_{L^2(\nu)\to L^2(\nu)}\bigr)\|f\|_{L^2(\nu_\varepsilon)},
		\end{split}
	\end{equation*}
	which concludes the proof of the lemma.
\end{proof}
\vv

Conversely to the previous lemma, we prove that the uniform $L^2(\nu_\varepsilon)$-boundedness of $T_{\nu_\varepsilon}$ with respect to $\varepsilon$ controls the $L^2(\nu)$-boundedness of $T_\nu$ at small scales.
\begin{lemma}\label{lemma:approx2}
	Let $\nu$, $\nu_\varepsilon, A$ and $Q$ be as in Lemma \ref{lemma:approx1}. Let us also assume that there exists $C>0$ such that $\|T_{\nu_\varepsilon}\|_{L^2(\nu_\varepsilon)\to L^2(\nu_\varepsilon)}\leq C$ for all $\varepsilon>0$. Then for any fixed $\delta>0$,
	\begin{equation}\label{eq:claim_lemma62}
		\lim_{\varepsilon\to 0}\bigl\|T_{\nu_\varepsilon,\delta}\bigr\|_{L^2(\nu_\varepsilon)}=\|T_{\nu,\delta}\|_{L^2(\nu)} \quad \text{ for any fixed} \,\, \delta>0.
	\end{equation}
	In particular $\|T_{\nu}\|_{L^2(\nu)\to L^2(\nu)}\leq C$.
\end{lemma}

\begin{proof}
	It is clear that $\|T_{\nu}\|_{L^2(\nu)\to L^2(\nu)}\leq C$ follows form \eqref{eq:claim_lemma62} and $\|T_{\nu_\varepsilon}\|_{L^2(\nu_\varepsilon)\to L^2(\nu_\varepsilon)}\leq C$ and so it suffices to prove  \eqref{eq:claim_lemma62}. To this end, if $f\in C^\infty_c(Q)$ 	and  $\delta>0$ is fixed, we have that
	\begin{equation*}
		\begin{split}
			&\biggl|\int \bigl|T_{\nu_{\varepsilon},\delta}f(x)\bigr|^2\, d\nu_\varepsilon (x) -  \int \bigl|T_{\nu,\delta}f(x)\bigr|^2\, d\nu (x)\biggr|\\
			&\qquad \leq \int \bigl|T_{\nu,\delta}f(x)\bigr|^2 \, d\bigl(\nu_\varepsilon -\nu\bigr)(x) + \int  \Bigl|\bigl|T_{\nu_{\varepsilon},\delta}f(x)\bigr|^2 - \bigl|T_{\nu,\delta}f(x)\bigr|^2\Bigr|\, d\nu_\varepsilon (x)\\
			&\qquad \eqqcolon I_{\delta,\varepsilon} + II_{\delta,\varepsilon}.
		\end{split}
	\end{equation*}
	
	The first summand $I_{\delta,\varepsilon}$ vanishes as $\varepsilon\to 0$ because $|T_{\nu,\delta}f(x)|^2$ is a bounded and continuous function, $\nu$ is compactly supported, and $\nu_\varepsilon$ converges weakly to $\nu$.	If $\psi \in C^\infty$ is a non-negative smooth function such that $\chi_{B(0,2)^c} \leq \psi \leq \chi_{B(0,1)^c}$ with $\| \nabla \psi\|_\infty \lesssim 1$. We set $\psi_\delta(\cdot)=\psi(\cdot/\delta)$ and
	$$
	K(x,y)\coloneqq \nabla_1\Gamma_A(x,y), \,\,x,y\in \Rn1\setminus \{0\} \quad 
	\text{and} \quad K_\delta(x,y)\coloneqq K(x,y)\psi_\delta(x-y).
	$$ 
	As Lemma \ref{lem:estim_fund_sol} entails $|K(x,y)|\lesssim |x-y|^{-n}$ for all $x,y\in   B(0,R)$, it follows that
	\begin{equation}\label{eq:decay_K_delta}
		|K_\delta(x,y)|\lesssim \delta^{-n}\qquad \text{ for all }x,y \in B(0,R), x\neq y.
	\end{equation}
	Moreover, if $x, y_1, y_2 \in B(0,R)$ such that  $y_1 \neq y_2$ and $2|y_1-y_2|<\min\{|x-y_1|,|x-y_2|\}$, by the mean value theorem and  \eqref{eq:continuityGamma}, it holds that
	\begin{align}
		| K_\delta(x,y_1) &-  K_\delta(x,y_2) |\notag\\
		&\leq \left| K_\delta(x,y_1) \right| | \psi_\delta(x-y_1)-\psi_\delta(x-y_2)| + \left| K_\delta(x,y_1)-  K_\delta(x,y_2) \right| \notag\\
		&\lesssim \frac{|y_1-y_2|}{\delta^{n+1}} + \left( \frac{|y_1-y_2|^\beta}{\delta^\beta}+ \int_0^{\frac{|y_1-y_2|}{\delta}} \omega_A(t)\,\frac{dt}{t}\right)  \frac{1}{\delta^{n}}\notag\\
		& \overset{\eqref{eq:def_alpha_A}}{\lesssim_\delta} \alpha_A\Bigl(\frac{|y_1-y_2|}{\delta}\Bigr).\notag
	\end{align}
	If $|y_1-y_2|<\ve \ll \delta$, then,  since $\delta<\min\{|x-y_1|,|x-y_2|\}$, we have that
	\begin{equation}\label{eq:continuity_K_delta}
		\left| K_\delta(x,y_1)-  K_\delta(x,y_2) \right| \lesssim_{\Lambda, n, \delta } \alpha_A(\ve/\delta) \to 0, \,\,\,\text{as}\,\,\ve\to 0.
	\end{equation}
	
	In order to estimate $II_{\delta,\varepsilon}$ we first observe that \eqref{eq:decay_K_delta} and the fact that $\nu_\varepsilon(\Rn1)=\nu(\Rn1)$ readily imply
	\begin{equation}\label{eq:lemma_estim_II_plus}
		\bigl|T_{\nu_\varepsilon,\delta}f(x) + T_{\nu,\delta}f(x)\bigr|\lesssim \delta^{-n}\|f\|_{L^\infty(\Rn1)}\nu(\Rn1),
	\end{equation}
	where the implicit constant is independent of $\varepsilon$.
	Also, the definition of $\nu_\varepsilon$ and Fubini's theorem yield
	\begin{equation*}
		\begin{split}
			&\bigl|T_{\nu_\varepsilon,\delta}f(x) - T_{\nu,\delta}f(x)\bigr|=\biggl|\int K_{\delta}(x,y)f(y)\, d\nu_\varepsilon(y) - \int K_\delta (x,z)f(z)\, d\nu(z)\biggr|\\
			&\qquad =\biggl|\int K_\delta(x,y)f(y)\int \varphi_\varepsilon(y-z)\, d\nu(z)\, dy \\
			&\qquad\qquad\qquad\qquad- \int K_{\delta}(x,z)f(z)\int \varphi_\varepsilon(y)\, dy\, d\nu(z)\biggr|\\
			&\qquad=\biggl|\int \biggl[ K_\delta (x,y)f(y)\varphi_\varepsilon(y-z)\, dy - \int K_\delta (x,z)f(z)\varphi_\varepsilon(y)\, dy\biggr]\, d\nu(z)\biggr|\\
			&\qquad\leq \int \int_{B(0,\varepsilon)}\bigl| K_\delta(x,y+z)f(y+z) - K_\delta(x,z)f(z)\bigr|\varphi_\varepsilon(y)\, dy\, d\nu(z).
		\end{split}
	\end{equation*}
	Furthermore, the estimate \eqref{eq:decay_K_delta}, the mean value theorem, and \eqref{eq:continuity_K_delta} imply
	\begin{equation*}
		\begin{split}
			&\bigl|T_{\nu_\varepsilon,\delta}f(x) - T_{\nu,\delta}f(x)\bigr|\\
			&\qquad\qquad \leq \int \int_{B(0,\varepsilon)}\bigl| K_\delta(x,y+z)\bigr||f(y+z) - f(z)|\varphi_\varepsilon(y)\, dy\, d\nu(z)\\
			&\qquad\qquad \qquad \qquad +\int \int_{B(0,\varepsilon)}|f(z)|\bigl| K_\delta(x,y+z)-  K_\delta(x,z)\bigr|\varphi_\varepsilon(y)\, dy\, d\nu(z)\\
			&\qquad\qquad \lesssim \frac{\varepsilon}{\delta^n}\|\nabla f\|_{L^\infty(\Rn1)} \nu(\Rn1) +\alpha_A(\ve/\delta)\|f\|_{L^\infty(\Rn1)} \nu(\Rn1).
		\end{split}
	\end{equation*}
	Hence,  by\eqref{eq:lemma_estim_II_plus} and the latter inequality we infer that
	\begin{equation*}
		\begin{split}
			II_{\delta,\varepsilon}&=\int \bigl|T_{\nu_\varepsilon,\delta}f(x) + T_{\nu,\delta}f(x)\bigr|\bigl|T_{\nu_\varepsilon,\delta}f(x) - T_{\nu,\delta}f(x)\bigr|\, d\nu_\varepsilon(x)\lesssim_{\delta,f,\nu}\varepsilon + \alpha_A(\ve/\delta)
		\end{split}
	\end{equation*}
	and	so $II_{\delta,\varepsilon}\to 0$ as $\varepsilon\to 0$. This proves \eqref{eq:claim_lemma62}, which concludes the proof of the lemma.
\end{proof}

\vvv

\section{The proof of Theorem \ref{theorem:bound_L2_norm_operators}}\label{section:final_section}

\begin{proof}[Proof of Theorem \ref{theorem:bound_L2_norm_operators}]
	Let $R\coloneqq \diam(\supp\mu)$ and assume that $\supp \mu\subset Q_0$ for some cube $Q_0$ such that $\ell(Q_0)=R$. First, we prove that 
	\begin{equation}\label{eq:R_lesssim_T}
		\|\mathcal R_\mu\|_{L^2(\mu)\to L^2(\mu)}\lesssim 1 + \|T_\mu\|_{L^2(\mu)\to L^2(\mu)}.
	\end{equation}

	For $N\in \mathbb N$ we consider a collection of cubes $\{Q_i\}_{1\leq i\leq N^{n+1}}$ with disjoint interior such that $\ell(Q_i)=R/N$ for all $i$ and $Q_0=\bigcup_i Q_i$.
	
	We also denote $\mu_i\coloneqq \mu|_{Q_i}$ and observe that $T_{\mu_i}$ is bounded on $L^2(\mu_i)$ and satisfies
	\begin{equation}\label{eq:lem71boundt}
		\|T_{\mu_i}\|_{L^2(\mu_i)\to L^2(\mu_i)}\leq \|T_{\mu}\|_{L^2(\mu)\to L^2(\mu)}.
	\end{equation}
	We recall that $Q_i\subset B\bigl(x_{{Q_i}},\sqrt{n+1}\ell(Q_i)\bigr)\subset B(x_{{Q_i}},2n\ell(Q_i))$.
	Moreover, we denote $S_i\coloneqq \sqrt{\bar{(A_s)}_{x_{Q_i},4n\Lambda\ell(Q_i)}}$ and we define the change of variables $\psi_i(x)\coloneqq S_ix$ for all $x\in \Rn1$.
	Finally, we consider the measure $\nu_i\coloneqq (\psi_i^{-1})_\sharp \mu_i$, and we denote by $\hat A^i$ the matrix defined  in \eqref{eq:change_var_matrices_oscillation}, namely
	\[
		\hat A^i(x)\coloneqq \frac{|\det S_i| \bigl|S_i^{-1}\bigl(B(x_{Q_i},4n\Lambda\ell(Q_i))\bigr)\bigr|}{|B(x_{Q_i},4n\Lambda\ell(Q_i))|}S_i^{-1}(A\circ S_i)(x)S_i^{-1} \qquad \text{ for all }x\in \Rn1.
	\]
	By Lemma \ref{lemma:change_matrix_average} we have that $\avint_{S_i^{-1}\bigl(B(x_{Q_i}, 4n\Lambda\ell(Q_i))\bigr)}\hat A^i_s = Id$ and, by Lemma \ref{lemma:comparison_moduli_change_of_variables},  the moduli of oscillation {$\omega_{\hat{A}^i}$ and $\tau_{\hat{A}^i}$ belong to $\widetilde \DMO$.} Moreover, by \eqref{eq:inclusion_balls_cov} it holds that
\[
		 \supp \nu_i\subset 	S_i^{-1}(Q_i)\subset S_i^{-1}\bigl(B(x_{Q_i},2n\ell(Q_i))\bigr)\subseteq B\bigl(S_i^{-1}x_{Q_i},4n\Lambda^{1/2}\ell(Q_i)\bigr)
\]
and the ball $B\bigl(S_i^{-1}x_{Q_i},4n\Lambda^{1/2}\ell(Q_i)\bigr)$ is contained in a cube $\tilde Q_i$ with center $x_{\tilde Q_i}=S_i^{-1}x_{Q_i}$ and side-length $\ell(\tilde Q_i)=4n\Lambda^{1/2}\ell(Q_i)$.

	 Given $M\coloneqq 4n\Lambda^{3/2}$, the inclusions \eqref{eq:inclusion_balls_cov} imply
	\[
		B\bigl(x_{\tilde Q_i}, \ell(\tilde Q_i)\bigr)\subset S_i^{-1}\bigl(B(x_{Q_i}, 4n\Lambda \ell( Q_i))\bigr)\subset B\bigl(x_{\tilde Q_i}, M\ell(\tilde Q_i)\bigr).
	\]
	
	The definition of $\nu_i$ and the bilipschitz character of $\psi_i$ yield
	\[
		\nu_i(B(x,r))\leq c_0 M r^n \qquad \text{ for all }x\in \Rn1, r>0.
	\]
	Hence, the measure $\nu_i$ belongs to $M^n_+(\Rn1)$ and is supported on $\tilde Q_i$. For $\varepsilon>0$, let $\nu_{i,\varepsilon}$ be the auxiliary measure defined as in \eqref{eq:definition_measure_nu_varepsilon}. Let $f$ be a compactly supported Lipschitz function in $  L^2(\nu_{i,\ve})$ 
	(this class is clearly dense in $ L^2(\nu_{i,\varepsilon})$).  By triangle inequality, the Main Lemma \ref{lem:main_lemma}, and Lemma \ref{lemma:comparison_moduli_change_of_variables} we have
	\begin{equation}\label{eq:lem71a}
		\begin{split}
			\|\mathcal R_{\nu_{i,\varepsilon}} f \|_{L^2(\nu_{i,\varepsilon})} & \leq \omega_n\|\omega^{-1}_n \mathcal R_{\nu_{i,\varepsilon}} f - T_{\psi_i,\nu_{i,\varepsilon}} f\|_{L^2(\nu_{i,\varepsilon})} + \omega_n\|T_{\psi_i,\nu_{i,\varepsilon}} f\|_{ L^2(\nu_{i,\varepsilon})}\\
			& \leq  C' \left(\mathfrak I_{\tau_A} (\ell(\tilde Q_i)) +  \widehat \tau_A(\ell(\tilde Q_i))\right) \|f\|_{L^2(\nu_{i,\varepsilon})}\\
			&  +\omega_n \|T_{\psi_i, \nu_{i,\varepsilon}} f\|_{L^2(\nu_{i,\varepsilon})}+ C'' \biggl(\int_{0}^{\ell(\tilde Q_i)}\omega_A(t) \,\frac{dt}{t} \biggr)^{1/2}\|\mathcal R_{\nu_{i,\varepsilon}} f\|_{L^2(\nu_{i,\varepsilon})}\\
			&\leq \frac{1}{2} \|\mathcal R_{\nu_{i,\varepsilon}}f\|_{L^2(\nu_{i,\varepsilon})} +\left(\eta  \,\|f\|_{L^2(\nu_{i,\varepsilon})}+ \omega_n  \|T_{\psi_i, \nu_{i,\ve}}f\|_{L^2(\nu_{i.\ve})}\right),
		\end{split}
	\end{equation}
	where in the last inequality we used Lemma \ref{lemma:approx1} and chose  $N$ big enough (see also Remark \ref{rem:rem_right_cont}) so that 
{	\begin{align*}
		&C''  \Bigl(\int_{0}^{\ell(\tilde Q_i)}\omega_A(t) \,\frac{dt}{t} \Bigr)^{1/2} \leq 1/2\\
		&C' \Bigl(\mathfrak I_{\tau_A} \bigl(\ell(\tilde Q_i)\bigr) +  \widehat \tau_A\bigl(\ell(\tilde Q_i)\bigr)\Bigr) \leq \eta,
			\end{align*}
			where $\eta \in (0,1)$ can be taken as  small as we want by adjusting the value of $\ell(\tilde Q_i)$.}
	Hence, since $\|\mathcal R_{\nu_{i,\ve}} f\|_{L^2(\nu_{i,\ve})} <\infty$ by construction of the approximating measures $\nu_{i,\ve}$ and the fact that $f$ is Lipschitz with compact support, the estimate \eqref{eq:lem71a} implies that
\begin{align*}
		\|\mathcal R_{\nu_{i,\varepsilon}} f \|_{L^2(\nu_{i,\varepsilon})}  &\leq 2\eta \,\|f\|_{L^2(\nu_{i,\varepsilon})}+ 2 \omega_n \|T_{\psi_i, \nu_{i,\ve}}f\|_{L^2(\nu_{i,\ve} )}\\
&\overset{\eqref{eq:L2nu_eps_toL2_nu}}{\leq} \left(2\eta  + 2\omega_n \|T_{\psi_i, \nu_{i}}\|_{L^2(\nu_{i})\to L^2(\nu_{i})}\right)\|f\|_{L^2(\nu_{i,\varepsilon})}.
\end{align*}
	So,  if we take the supremum over all compactly supported Lipschitz functions in $  L^2(\nu_{i,\ve})$, by density, we have that for any $\delta>0$,
	\begin{equation*}
		\begin{split}
	\|\mathcal R_{\nu_{i,\ve}, \delta} \|_{L^2(\nu_{i,\ve}) \to L^2(\nu_{i,\varepsilon})} & \leq 2\eta  + 2\omega_n  \|T_{\psi_i, \nu_{i}}\|_{L^2(\nu_{i})\to L^2(\nu_{i})} \overset{\eqref{eq:T_phi_delta_sim_T_delta}}{\leq} 2\eta  + \widetilde C_0 \bigl\|\widehat T^{\psi_i^{-1}}_{\mu_i}\bigr\|_{L^2(\mu_{i})\to L^2(\mu_{i})}\\ &\overset{\eqref{eq:change_of_truncations}}{\leq} 2\eta  + \widetilde C_1 \|T_{\mu_i}\|_{L^2(\mu_i)\to L^2(\mu_i)}.
		\end{split}
	\end{equation*}
Taking limits as $\ve \to 0$,	by Lemma \ref{lemma:approx2},  we infer 
	\begin{equation}\label{eq:final_estimate_L2_norm_Ri}
\|\mathcal R_{\nu_{i}, \delta}\|_{L^2(\nu_{i})\to L^2(\nu_{i})} \leq 2\eta  + \widetilde C_1 \|T_{\mu_i}\|_{L^2(\mu_i)\to L^2(\mu_i)},
	\end{equation}
uniformly in $\delta>0$, and, applying \cite[Corollary 1.3]{To21}, there exists $\widetilde C_2>0$ depending on dimension and $\Lambda$, such that
	\begin{align}\label{eq:7.4}
		\|\mathcal R_{\mu_i}\|_{L^2(\mu_i)\to L^2(\mu_i)} &\leq   \widetilde C_2 \|\mathcal R_{\nu_i}\|_{L^2(\nu_i)\to L^2(\nu_i)} \leq 2\eta \widetilde C_2 +  \widetilde C_1 \widetilde C_2 \|T_{\mu_i}\|_{L^2(\mu_i)\to L^2(\mu_i)}\\
		& \overset{\eqref{eq:lem71boundt}}{\leq} 2\eta \widetilde C_2 +  \widetilde C_1 \widetilde C_2 \|T_{\mu}\|_{L^2(\mu)\to L^2(\mu)}.\notag
	\end{align}
	{Thus, since $0 < \eta <1$,  \cite[Proposition 2.25]{To14} and \cite[Theorem 2.21]{To14} imply that	
	\begin{align*}
	\|\mathcal R_\mu\|_{L^2(\mu)\to L^2(\mu)} &\lesssim \|\mathcal R_{*}\|_{{M}(\mathbb{R}^{n+1}) \to L^{1,\infty}(\mu)} \lesssim_{n, N} \sum_{i=1}^{N^{n+1}} \|\mathcal R_{*}\|_{{M}(\R^{n+1})\to L^{1,\infty}(\mu_i)} \\
	& \lesssim_{n, N}  \sum_{i=1}^{N^{n+1}} ( 1 + \|\mathcal R_{\mu_i}\|_{L^2(\mu_i)\to L^2(\mu_i)} ) \\
	&\overset{\eqref{eq:7.4}}{\lesssim}_{n, N, \Lambda,R}  1+ \|T_\mu\|_{L^2(\mu)\to L^2(\mu)},
	\end{align*}
	which  concludes the proof of \eqref{eq:R_lesssim_T} since $N$ depends on $\Lambda,$ $n$, and $R$. Note that  \eqref{eq:7.4} shows that for every $\epsilon>0$, there exists $R$ small enough  such that if $\diam (\supp \mu) \leq R$ then \eqref{eq:maineq-small_R-T} holds.
	
	{In order to prove the converse inequality it is enough to observe that, for $N$ big enough, \eqref{eq:R_lesssim_T} yields
	\begin{equation*}
		\begin{split}
			\|T_{\psi_i, \nu_{i}}\|_{L^2(\nu_{i,\varepsilon})\to L^2(\nu_{i})} &\leq \|T_{\psi_i,\nu_{i}} - \omega_n^{-1}\mathcal R_{\nu_{i}} \|_{L^2(\nu_{i,\varepsilon})\to L^2(\nu_{i})} + \omega_n^{-1}\|\mathcal R_{\nu_{i}}\|_{L^2(\nu_{i,\varepsilon})\to L^2(\nu_{i})}\\
			&\lesssim 1 + \|\mathcal R_{\nu_{i}}\|_{L^2(\nu_{i})\to L^2(\nu_{i})} \lesssim 1+	\|\mathcal R_{\mu_i}\|_{L^2(\mu_i)\to L^2(\mu_i)}.
		\end{split}
	\end{equation*}
where the last inequality follows	from \cite[Corollary 1.3]{To21}. Then we apply \eqref{eq:T_phi_delta_sim_T_delta} and \eqref{eq:change_of_truncations} and deduce that $\|T_{\mu_i}\|_{L^2(\mu)\to L^2(\mu)} \lesssim 1 + 	\|\mathcal R_{\mu_i}\|_{L^2(\mu_i)\to L^2(\mu_i)}$. Finally, we can  repeat the argument above using \cite[Theorem 2.21]{To14} and  \cite[Proposition 2.25]{To14}, which are still true for the operator $T_\mu$ (the hypothesis that $\tau_A$ is a Dini function makes possible to argue via estimates in terms of the centered maximal function which closely resemble \eqref{eq:estimate_mathfrak_1}}), and conclude the proof of the theorem.}
\end{proof}

\vv

\begin{proof}[Proof of Corollary \ref{theorem:DS_elliptic}]  Let $\mu$ be an $n$-AD-regular measure on $\Rn1$ with compact support such that the gradient of the single layer potential $T_\mu$ is bounded on $L^2(\mu)$. Then, in particular $\mu\in M^n_+(\supp\mu)$ so Proposition \ref{theorem:bound_L2_norm_operators} implies that  $\mathcal R_\mu$ is bounded on $L^2(\mu)$. Hence, the main result of \cite{NToV14} yields that $\mu$ is uniformly $n$-rectifiable. Conversely, it is immediate that Theorem \ref{theorem:bound_L2_norm_operators} and the boundedness of the Riesz transform on uniformly $n$-rectifiable sets imply that the boundedness of $T_\mu$ on uniformly $n$-rectifiable sets. 
\end{proof}

\begin{proof}[Proof of Corollary \ref{theorem:ENV_elliptic}]  Let $\mu$ be a non-trivial totally  irregular measure on $\Rn1$,  i.e.,  it satisfies $0<\Theta^{*,n}(x,\mu)<\infty$ and $\Theta^n_*(x,\mu)=0$ for $\mu$-a.e. $x\in \Rn1$. Arguing by contradiction, we assume that $\|\mathcal T_\mu\|_{L^2(\mu)\to L^2(\mu)}<\infty$ and so, by Lemma \ref{lem:T_bounded_polynomial_growth}, $\mu$ has $n$-polynomial growth. Thus, by Theorem \ref{theorem:bound_L2_norm_operators}, we have that $\|\mathcal R_\mu\|_{L^2(\mu)\to L^2(\mu)}<\infty$, which contradicts the main result of  \cite{ENV12}. 
\end{proof}

\begin{proof}[Proof of Corollary \ref{theorem:NTVpubmat_elliptic}]  
This is a direct  consequence of Theorem \ref{theorem:bound_L2_norm_operators} and the main result of \cite{NToV14b}.
\end{proof}

\vv

{
\begin{proof}[Proof of Corollary \ref{cor:elliptic_GSTo}]
We first apply Lemma \ref{lem:main2} for $N$ large enough (depending on $n$, $\tau, C_1$, and $\diam( \supp \nu)$) so that 
\[
	\mathfrak I_{\alpha_A}(2^{-N}) \leq \tau^{1/2},
\]
where we recall that $\alpha_A(t)=t + t^\beta + \omega_A(t)$, $t>0$.

Then, in view of \eqref{eq:7.4}, we choose $\ell(Q)$ small enough (depending on $N$ and $\tau$)  so that $2^N \ell(Q)$ is as small as required in the proof of Theorem \ref{theorem:bound_L2_norm_operators} in order for the following estimate to hold:
\begin{align*}
\| \mathcal{R}_\mu\|_{L^2(\mu|_{2^{N}Q})  \to L^2(\mu|_{2^{N}Q})} & \leq 1+C_2 \| {T}_\mu\|_{L^2(\mu|_{2^{N}Q}) \to L^2(\mu|_{2^{N}Q})} \leq  1+C_0'\,C_2,
\end{align*}
where $C_2$ depends on $n, \Lambda, M,$ and $C_0$. Finally, we may choose $\ell(Q)$ even smaller (depending on $n$, $N, \tau, C_0', C_1, C_2$, and $\diam(\supp \nu)$) so that
\[
	\mathfrak I_{\tau_A}\bigl(2^N\ell(Q)\bigr) + \widehat\tau_A\bigl(2^N\ell(Q)\bigr) \leq \tau^{1/2}
\]
and also 
\begin{align*}
 \mathfrak I_{\omega_A}(2^N\ell(Q))^{1/2}& \|\mathcal R_{\mu}\|_{L^2(\mu|_{2^NQ})\to L^2(\mu|_{2^N Q})  } \leq   (1+C_0'\, C_2) \mathfrak I_{\omega_A}(2^N\ell(Q))^{1/2}    \leq \tau^{1/2}.
\end{align*}
Collecting all the above estimates and combining them with Lemma \ref{lem:main2}, we deduce that
\[
	\Big(\avint_Q \Big| \mathcal R_{\mu}1(x)  - \avint_Q \mathcal R_{\mu}1 \Big|^2 \,d\mu(x)\Big)^{1/2} \leq 4 C_1\, \tau^{1/2}\Theta_\mu(2^NQ).
\]
If $\tilde \tau=4 C_1 \tau^{1/2}$ is small enough depending on $C_0, C_0', n, \Lambda, M$, and $\diam(\supp\mu)$, in light of Theorem \ref{theorem:bound_L2_norm_operators} applied to the measure $\mu|_{2^N Q}$, we can implement  \cite[Theorem 1.1]{GT18} and conclude the proof of the corollary.
\end{proof}
}

\vvv

\frenchspacing
\bibliographystyle{alpha}

\newcommand{\etalchar}[1]{$^{#1}$}
\def\cprime{$'$}

\end{document}